\documentclass[11pt]{article}
\usepackage[utf8]{inputenc}
\usepackage[T1]{fontenc}    % use 8-bit T1 fonts
\usepackage{lmodern}
\usepackage[margin=0.8in]{geometry}
\usepackage{hyperref}       % hyperlinks
\usepackage{url}            % simple URL typesetting
\usepackage{booktabs}       % professional-quality tables
\usepackage{amsfonts}       % blackboard math symbols
\usepackage{nicefrac}       % compact symbols for 1/2, etc.
\usepackage{microtype}      % microtypography
\usepackage{xcolor}         % colors
\usepackage{amsmath}
\usepackage{amsthm}
\usepackage{amssymb}
\usepackage{yhmath}
\usepackage{stmaryrd}
\usepackage{tikz}
\usepackage{soul}
\usepackage{authblk}

\usetikzlibrary{arrows,calc,shapes,decorations.pathreplacing}

\newtheorem{thm}{Theorem}
\newtheorem{prop}[thm]{Proposition}
\newtheorem{coro}[thm]{Corollary}
\newtheorem{lem}[thm]{Lemma}

\theoremstyle{remark}
\newtheorem*{ex}{Example}

\newtheorem*{rem}{Remark}
\newtheorem{remnum}[thm]{Remark}
\newtheorem*{rems}{Remarks}

\theoremstyle{definition}
\newtheorem{df}[thm]{Definition}

\newtheorem{exnum}[thm]{Example}

\DeclareMathOperator{\Diag}{Diag}
\DeclareMathOperator{\Id}{Id}
\DeclareMathOperator{\Conv}{Conv}

\newcommand{\fonction}[5]{
 \begin{array}{rrcl}
	#1: & #2 & \longrightarrow & #3 \\
    & #4 & \longmapsto & #5 \end{array}
    }

\newcommand{\RR}{\mathbb{R}}

\newcommand{\Pbf}{\mathbf{P}}
\newcommand{\Cbf}{\mathbf{C}}
\newcommand{\Mbf}{\mathbf{M}}
\newcommand{\bbf}{\mathbf{b}}
\newcommand{\Pibf}{\boldsymbol{\Pi}}
\newcommand{\Mbft}{\tilde{\mathbf{M}}}
\newcommand{\bbft}{\tilde{\mathbf{b}}}

\newcommand{\lb}{\llbracket}
\newcommand{\rb}{\rrbracket}
\newcommand{\One}{\mathbf{1}}

\title{Parameter identifiability of a deep feedforward ReLU neural network}
\date{\today}

\author[1]{Joachim Bona-Pellissier}

\author[2]{Fran\c{c}ois Bachoc}

\author[2]{Fran\c{c}ois Malgouyres}

\affil[1]{Institut de Math\'ematiques de Toulouse, UMR5219. Universit\'e de Toulouse, CNRS. UT1, F-31042  Toulouse, France}

\affil[2]{Institut de Math\'ematiques de Toulouse, UMR5219. Universit\'e de Toulouse, CNRS. UPS IMT, F-31062 Toulouse Cedex 9, France}

\begin{document}

\maketitle

\abstract{
The possibility for one to recover the parameters --weights and biases-- of a neural network thanks to the knowledge of its function on a subset of the input space can be, depending on the situation, a curse or a blessing. On one hand, recovering the parameters allows for better adversarial attacks and could also disclose sensitive information from the dataset used to construct the network. On the other hand, if the parameters of a network can be recovered, it guarantees the user that the features in the latent spaces can be interpreted. It also provides foundations to obtain formal guarantees on the performances of the network.

It is therefore important to characterize the networks whose parameters can be identified and those whose parameters cannot.

In this article, we provide a set of conditions on a deep fully-connected feedforward ReLU neural network under which the parameters of the network are uniquely identified --modulo permutation and positive rescaling-- from the function it implements on a subset of the input space.
}
\bigskip




\section{Introduction}

The development of Machine Learning and in particular of Deep Learning in the last decade has led to many breakthroughs in fields such as image classification \cite{krizhevsky2012imagenet}, object recognition \cite{redmon2016you,ren2015faster}, speech recognition \cite{6296526,sak2014long,hannun2014deep}, natural language processing \cite{DBLP:journals/corr/abs-1301-3781,mikolov2010recurrent,kalchbrenner2013recurrent}, anomaly detection \cite{pinto2021deep} or climate sciences \cite{adewoyin2021tru}. Deep neural networks are now widely used in real-life tasks stemming from those fields and beyond. This development and the diversity of contexts in which neural networks are used require to investigate theoretical properties that permit to guarantee that they can be used safely, are robust to attack, and can be used widely without giving access to sensitive information.

One key problem in these regards is the relation between the parameters and the function implemented by the network. If a parameterization of a network uniquely defines a function, the reverse is not true. Which other parameterizations define the same function, and what do they have in common? Which information on the parameters of a network are we able to infer from the knowledge of its function on a given domain? Addressing these questions is important for different reasons: industrial property, privacy, robustness and efficiency guarantee (see Section \ref{relwork-sec} for further discussions and references).

In this article, we consider fully-connected feedforward neural networks with $K$ layers, $K \geq 2$, with the ReLU activation function (see Section \ref{nn-sec} for details). The weights and bias parameterizing a neural network are gathered in a list $\Mbf$ of matrices and a list $\bbf$ of vectors. The corresponding function is denoted\footnote{For clarity of the proofs, we index the layers from $K$ (input) to $0$ (output). The input layer is not counted hence the `$K$ layers'.} $f_{\Mbf,\bbf}:\RR^{n_K} \longrightarrow \RR^{n_0}$. We say that two parameterizations $(\Mbf,\bbf)$ and $(\Mbft,\bbft)$ are \textit{equivalent} if they can be deduced from each other by the permutation of neurons in each hidden layer and by positive rescaling between the inward and outward weights of every neuron of every hidden layer. These two operations, that are precisely defined in Definition \ref{equivalence def-main}, are well-known in the literature \cite{pourzanjani2017improving, petzka2020symmetries,phuong2020functional,pmlr-v119-rolnick20a,stock:tel-03208517} and will be referred to as `permutation and positive rescaling'. As is well known and restated for completeness in Proposition \ref{equivalent networks have same function-main}, if two parameterizations $(\Mbf,\bbf)$ and $(\Mbft,\bbft)$ are equivalent, then the corresponding networks implement the same function: for all $x\in\RR^{n_K}$, $f_{\Mbf,\bbf}(x) = f_{\Mbft,\bbft}(x)$. In other words, parameter equivalence implies \textit{functional equivalence} of the networks.

The main contribution of this article is an {\em identifiablity statement} (see Theorem \ref{main theorem-main}) which establishes a `weak' converse of this statement. We consider a set $\Omega\subset \RR^{n_K}$ and two parameterizations $(\Mbf,\bbf)$ and $(\Mbft,\bbft)$ sharing the same architecture (number of layers and of neurons per layer). We establish a {\em sufficient condition} $\Pbf$ such that, if for all $x\in\Omega$, $f_{\Mbf,\bbf}(x) = f_{\Mbft,\bbft}(x)$ and the condition $\Pbf$ is met, then the two parameterizations $(\Mbf,\bbf)$ and $(\Mbft,\bbft)$ are equivalent. The motivation for the introduction of the set $\Omega$ is that, in practice, we may only test the values of $f_{\Mbf,\bbf}$ and $f_{\Mbft,\bbft}$ on a subset of $\RR^{n_K}$. Typically, $\Omega$ is a subset of the support of the input distribution law. Such a setting also allows to show that two networks which coincide on a given domain actually coincide on the whole input space $\mathbb{R}^{n_K}$. Indeed, if the functions implemented by the networks coincide on $\Omega$ and if the sufficient condition $\Pbf$ is satisfied, then the parameters are equivalent and thus by Proposition \ref{equivalent networks have same function-main} the functions also coincide on the rest of the input space $\RR^{n_K}$. This can be useful to bound the generalization error.

We also reformulate this identifiability statement (see Corollary \ref{risk minimization formulation-main}) in a way that illustrates its interest with regard to risk minimization. The corollary considers a random variable $X$ generating the input and an output of the form $Y = f_{\Mbf,\bbf}(X)$, for some parameters $(\Mbf,\bbf)$. It states that, when the condition $\Pbf$ is met, any estimated neural network $(\Mbft,\bbft)$ for which the population risk equals $0$  belongs to the equivalence class of $(\Mbf,\bbf)$. In words, the only way to have a perfect prediction is to perfectly recover $(\Mbf,\bbf)$, up to permutation and positive rescaling.

We describe the related works in Section \ref{relwork-sec}. In addition to the works providing identifiability, stability or stable recovery statements, we give a few pointers on privacy, robustness and guarantees of efficiency that motivate our study from an applied perspective. We define in Section \ref{nn-sec} the considered neural networks and provide the (known) properties that are useful in our context. The  sufficient condition $\Pbf$ and the main theorems are in Section \ref{main-result-sec}. The sketch of the proofs is in Section \ref{sketch-sec} and the details are in the Appendix.

\section{Related work}\label{relwork-sec}

\subsection{Identifiability, stability and stable recovery}
\subsubsection{Identifiability}

Identifiability of the parameters of neural networks has been the topic of a fair amount of work.  For smooth activation functions, some results were already established in the 1990s. For shallow networks, results exist for activation functions amongst which $\mathrm{tanh}$ \cite{sussmann1992uniqueness, albertini1993uniqueness}, the logistic sigmoid \cite{kuurkova1994functionally}, or the Gaussian and rational functions \cite{kainen1994uniqueness}. For deep networks, \cite{fefferman1994reconstructing} shows that with $\mathrm{tanh}$ as activation function, with only a few generic conditions on the parameters, two networks that implement the same function have the same architecture and the same parameters up to some permutations and sign-flip operations.

In the case of ReLU networks, we have seen that two operations are well known to preserve the function implemented by the network: permutation and positive rescaling. These operations define equivalence classes on the set of parameters, and we can at best identify the parameters of a network up to these equivalences. It is shown in \cite{phuong2020functional} that these operations are the only generic operations of this kind for ReLU networks with nonincreasing number of neurons per layer. Indeed, they show that for any fully-connected ReLU network architecture with nonincreasing number of neurons per layer, for any nonempty open set $\Omega$, there exists a parameterization $(\Mbf,\bbf)$ such that for any other parameterization $(\tilde{\Mbf}, \tilde{\bbf})$ satisfying some generic assumption, if $f_{\tilde{\Mbf}, \tilde{\bbf}}$ coincides with $f_{\Mbf,\bbf}$ on $\Omega$, then $(\tilde{\Mbf}, \tilde{\bbf})$ in the equivalence class of $(\Mbf,\bbf)$.

In this work, in order to establish identifiability, we take advantage of the piecewise linear geometry of the functions implemented by ReLU networks to identify the parameters. Indeed, it is well known that the function defined by a deep ReLU network is continuous piecewise-linear, i.e. we can partition the input space into polyhedral regions, sometimes called `linear regions', over which the function is affine. These regions are separated by boundaries that are made of pieces of hyperplanes and that correspond to the non differentiabilities of the function. One crosses such a boundary when the pre-activation value of a neuron (before applying the ReLU function) changes sign. By observing the boundary, one can infer information about the weights and bias of the said neuron.

Other articles adopt similar strategies for shallow \cite{petzka2020symmetries} or deep networks \cite{pmlr-v119-rolnick20a, phuong2020functional, stock:tel-03208517, stock2022embedding}. The specificity of our proof is to proceed by induction, identifying the weights and bias layer after layer. We discuss the differences between our condition $\mathbf{P}$ and the sufficient conditions given in \cite{phuong2020functional,pmlr-v119-rolnick20a,stock2022embedding} in detail in Section \ref{comparisonexisting-sec}.

In the case of shallow ReLU networks, \cite{petzka2020symmetries} establishes a sufficient condition on the parameters for identifiability. If the condition is satisfied by two two-layer fully-connected feedforward ReLU networks whose functions coincide on all the input space, then the parameters of one network can be obtained from the parameters of the other network by permutation and positive rescaling.

In the case of deep ReLU networks, \cite{pmlr-v119-rolnick20a} gives a sufficient condition to be able to reconstruct the architecture, weights and biases of a deep ReLU network by knowing its input-output map on all the input space. The condition concerns the boundaries mentioned above: for each neuron in a hidden layer, the authors define the boundary associated to the neuron as the points at which the pre-activation value of the neuron is zero.  Then, the condition requires each boundary associated to a neuron in a layer $k$ to intersect the boundaries associated to all the neurons in layer $k+1$ and $k-1$ (see Section \ref{comparisonexisting-sec} for more details).

Another kind of property is local identifiability, which is identifiability of a parameter $(\Mbf,\bbf)$ amongst a set of parameters that are close to $(\Mbf,\bbf)$. \cite{stock2022embedding} studies this property for shallow and deep networks. For a deep ReLU network, it first shows that under a trivial assumption, general identifiability up to permutation and positive rescaling implies local identifiability up to positive rescaling, and that the non-existence of `twin' neurons is necessary to identifiability and local identifiability. Then, \cite{stock2022embedding} makes a breakthrough by giving an abstract necessary and sufficient condition on $\mathbf{(M,b)}$ such that there exists a well-chosen \textit{finite} set $\Omega$ from which local identifiability holds up to positive rescalings, and it gives a bound on the size of the set. 
Furthermore, more recently, \cite{bonalocal} provided a numerically testable condition for local identifiability also from a finite set $\Omega$.

Finally, another line of work that can be linked to identifiability is the field of lossless compression of neural networks \cite{serra2020lossless, serra2021scaling}.

\subsubsection{Inverse stability and stable recovery}

Establishing identifiability properties is a first step towards establishing inverse stability properties and studying stable recovery algorithms. Given a norm between functions, we say that inverse stability holds when the proximity of the functions implemented by two networks with the same architecture implies the proximity of the corresponding parameters -up to equivalences of parameters, for instance permutations and positive rescalings in the case of ReLU networks. Inverse stability is a stronger property than identifiability, and is necessary for stable recovery algorithms, which goal is to practically recover the parameters of a network from its function. 

Inverse stability does not hold in general with the uniform norm for fully-connected feedforward neural networks. Indeed, \cite{petersen2020topological} shows that for any depth, for any architecture with at least 3 neurons in the first hidden layer and any practically used activation function, there exists a sequence of networks whose function tends uniformly to $0$ while any parameterization of these networks tends to infinity.

Many inverse stability and stable recovery results already exist for shallow networks. \cite{berner2019degenerate} studies inverse stability directly up to functional equivalence classes, without specifying the nature of these classes in terms of parameters -which interests us in this paper. The authors show that inverse stability has interesting implications in terms of optimization, allowing to link the minima in the parameter space to minima in the realization space (the space of all the functions that can be implemented by a network) and to estimate the quality of local minima in the parameter space based on their radii. Referring to the counter-example given by \cite{petersen2020topological}, the authors of \cite{berner2019degenerate} argue that the Sobolev norm is more suited than the uniform norm to the problem of inverse stability. With this norm, they concretely establish an inverse stability result on shallow ReLU networks without bias, under a few conditions on the parameters. 

When it comes to stable recovery algorithms, \cite{fu2020guaranteed} provides a sample complexity under which one can recover the parameters of a shallow network with sigmoid activation function using cross-entropy as a loss. For shallow fully-connected ReLU networks, without bias and with Gaussian input, \cite{ge2018learning, zhang2019learning, zhong2017recovery, zhou2021local} study the stable recovery of the parameters of a teacher network. They give a sample complexity under which minimizing the empirical risk allows to recover the parameters of the network. \cite{li2017convergence} studies the same configuration but with an identity mapping that skips one layer. ReLU networks can also be used to recover a network with absolute value as activation function \cite{li2020learning}. In fact, a neuron with absolute value can be seen as a sum of two ReLU neurons.

Some results also exist in the case of shallow convolutional networks. 
\cite{brutzkus2017globally, zhang2020improved, zhang2020guaranteed, du2018whenis} establish stable recovery results for convolutional ReLU networks with no overlapping. \cite{janzamin2015beating} gives a result in the case of a sigmoidal activation function. The case of convolutional ReLU networks with overlapping is studied in \cite{goel2018learning}.

Stability and stable recovery for \textit{deep} networks is a more complicated question. A few results exist on the subject, but it stays mostly unexplored.

Among them, for deep structured linear networks, \cite{MalgouyresLandsbergITW, MalgouyresLandsberg_long, malgouyres2020stable} use a tensorial lifting technique to  establish inverse stability properties. \cite{MalgouyresLandsbergITW, MalgouyresLandsberg_long} establish necessary and sufficient conditions of inverse stability for a general constraint on the parameters defining the network. \cite{malgouyres2020stable} specializes the analysis to the sparsity constraint on the parameters, and obtains necessary and sufficient conditions of inverse stability.
 
The authors of \cite{arora2014provable} consider deep feed-forward networks with Heavyside activation function which are very sparse and randomly generated. They show that these can be learned with high probability one layer after another.

The authors of \cite{sedghi2014provable} consider a deep feed-forward neural network, with an activation function that can be, inter alia, ReLU, sigmoid or softmax. They show that, if the input is Gaussian or its distribution is known, and if the weight matrix of the first layer is sparse, then a method based on moments and sparse dictionary learning can retrieve it exactly. Nothing is said about the stability or the estimation of the other layers. 

For deep ReLU networks, in the case where one has full access to the function implemented by the network \cite{pmlr-v119-rolnick20a} provides a practical algorithm able to approximately recover the parameters modulo permutation and rescaling, and \cite{carlini2020cryptanalytic} reconstructs a functionally equivalent network, formulating it as a cryptanalytic problem.

Further inverse stability and stable recovery results for deep ReLU networks are still to be established. Studying identifiability for these networks, as we do in this article, is a first step towards this goal.

\subsection{Motivations: privacy, robustness and interpretability}

The generalization of deep networks in various applications such as life style choices or medical diagnosis has raised new concerns about privacy and security. Indeed, to perform well, neural networks need to be trained with many examples. The training of some models can take up to several weeks, and need huge datasets such as ImageNet, which contains millions of images. For instance, the training of the giant GPT-3 neural network costed an estimated 12 millions of dollars \cite{NEURIPS2020_1457c0d6}. For this reason, trained models are valuable and their owners may want to protect them from replication. 

In many applications, the training dataset also contains sensitive information that could be uncovered
\cite{narayanan2008robust, loukides2010disclosure, carlini2019secret, fredrikson2015model, chen2021protect}. It is crucial, to the deployment of the solutions relying on deep networks, to guarantee that this cannot occur. For example, when the system returns a confidence indicator in the prediction or a notion of margin, the {\em Model Inversion Attack} described in \cite{fredrikson2015model} uncovers learning examples $x$ by maximizing the confidence/margin, under a constraint that $\|f_{\Mbf,\bbf}(x) -y\|\leq \varepsilon$, where $y$ is a target output. In moderate dimension, this can be achieved by simply applying $f_{\Mbf,\bbf}$ several times. In large dimension, the complexity of the computation is too large unless the adversary can compute $\nabla f_{\Mbf,\bbf}(x)$, for any $x$. To perform this computation, the adversary needs to know $(\Mbf,\bbf)$. Guaranteeing that the parameters cannot be recovered prevents this. With a slightly different objective, (differential) {\em privacy deep learning} also assumes that the adversary has the knowledge of the network parameters \cite{abadi2016deep}.

Furthermore, knowing the architecture and parameters of a network could make easier for a malicious user to attack it, for instance with adversarial attacks. Indeed, if some black-box adversarial attacks do exist \cite{su2019one, sarkar2017upset, cisse2017houdini}, many of them use the knowledge of the parameters of the network, at least to compute the gradients \cite{szegedy2013intriguing,goodfellow2014explaining, kurakin2016adversarial, papernot2016limitations, carlini2017towards, moosavi2016deepfool, moosavi2017universal, baluja2017adversarial}.

For all these reasons, the authors of \cite{chabanne2020protection} developed a method of preventing parameters extraction by artificially complexifying the network without changing its global behavior. This method builds on previous works on stable recovery of the parameters of the ReLU networks, and in particular on the fact that the piecewise-linear structure of the functions implemented by such networks can be used to recover the parameters. Further understanding of stable recovery for deep networks could help improve protection methods.

Another interest of our work is interpretability of deep neural networks. In some uses of deep networks we want to understand what happens at a layer level and how we can interpret the feature spaces defined by the different layers. But such an interpretation is more meaningful if we know that, for a given function implemented by the network, the parameterization is unique -up to elementary operations such as permutations and positive rescalings for ReLU networks.

\section{Neural networks}\label{nn-sec}

In this section, we provide known definitions and properties of neural networks with ReLU activation functions. For a self-contained reading, all the corresponding proofs are provided in the appendix.

\subsection{Parameterization of neural networks}

We consider deep feedforward ReLU networks with $K \geq 2$ layers. To clarify any ambiguity, note that the input layer is not actually counted, as it does not gather any weights. As evoked in the introduction, we index the layers of a deep neural network in reverse order, from $K$ to $0$, for some $K \geq 2$. The input layer is the layer $K$, the output layer is the layer $0$, and between them are $K-1$ hidden layers. We denote by $n_k \in \mathbb{N}^*$ the number of neurons of the layer $k$. The information contained at the layer $k$ is a $n_k$-dimensional vector.

Let $k \in \llbracket 0, K-1 \rrbracket$. We denote the weights between the layer $k+1$ and the layer $k$ with a matrix $M^{k} \in \mathbb{R}^{n_k \times n_{k+1}}$. We also consider a bias vector $b^{k} \in \mathbb{R}^{n_k}$ at the layer $k$, and the ReLU activation function, that is
$\sigma(x) = \max(x,0)$. By extension, for a vector $x = (x_1, \dots , x_p)^T \in \mathbb{R}^p$ we also write $\sigma(x) = (\sigma(x_1), \dots , \sigma(x_p))^T$. We denote by $h_k$ the mapping implemented by the network between the layer $k+1$ and the layer $k$. If $x \in \mathbb{R}^{n_{k+1}}$ is the information contained at the layer $k+1$, the information contained at the layer $k$ is:
\begin{equation}\label{hk-main} h_k(x) = \begin{cases}\sigma(M^{k} x + b^{k}) & \text{if } k \neq 0 \\M^{k} x + b^{k} & \text{if } k=0.\end{cases}
\end{equation}

The parameters of the network can be summarized in the couple $(\mathbf{M},\mathbf{b})$, where $\mathbf{M} = (M^{0}, M^{1}, \dots , M^{K-1}) \in \mathbb{R}^{n_0 \times n_1} \times \dots \times \mathbb{R}^{n_{K-1} \times n_K}$ and $\mathbf{b} = (b^{0}, b^{1}, \dots , b^{K-1}) \in \mathbb{R}^{n_0} \times \dots \times \mathbb{R}^{n_{K-1}}$. The function implemented by the network is then $f_{\mathbf{M},\mathbf{b}} = h_0 \circ h_1 \circ \dots \circ h_{K-1}$, from $\mathbb{R}^{n_K}$ to $\mathbb{R}^{n_0}$. We refer to Figure \ref{NN representation} for a representation of a neural network and its parameters.

\begin{figure}
\center
\includegraphics[scale=0.8]{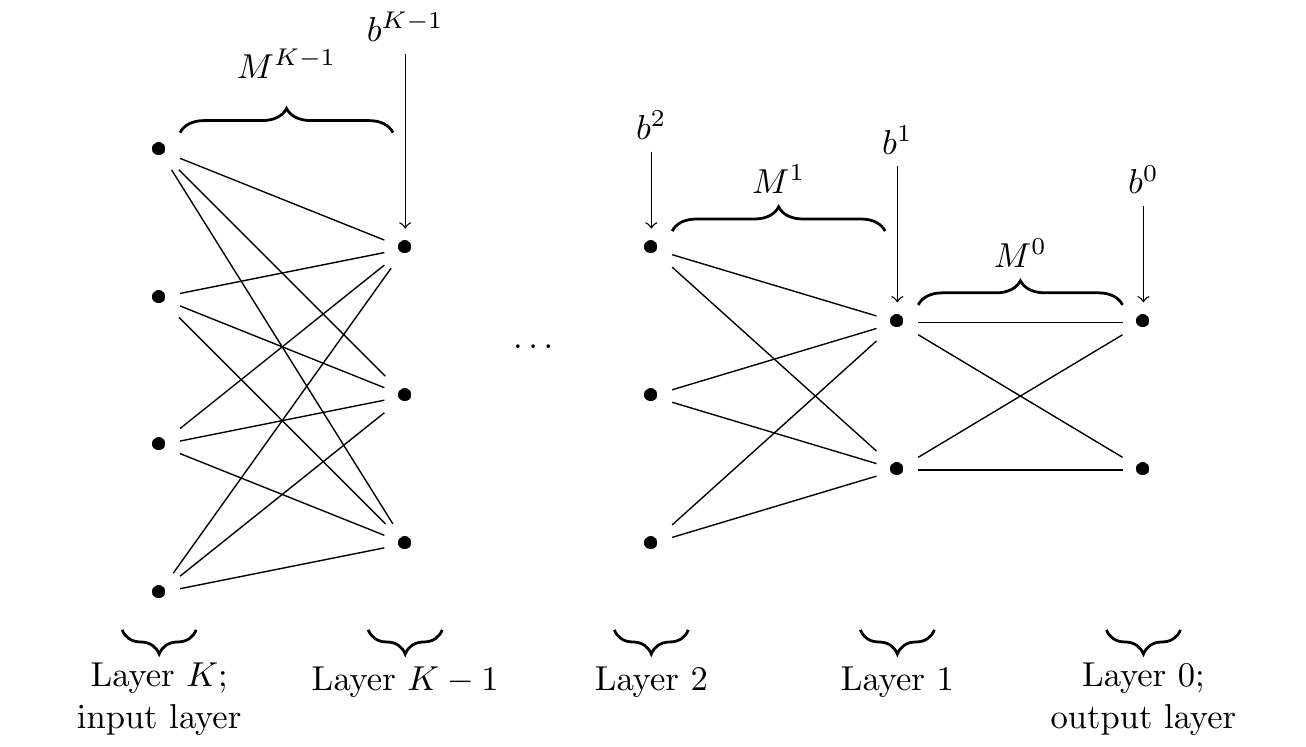}
\caption{The parameters $\mathbf{M}$ and $\mathbf{b}$ of a neural network.}
\label{NN representation}
\end{figure}

\subsection{Continuous piecewise linear functions and neural networks \label{CPL functions-sec}}

We will actively use the fact that the function implemented by a deep ReLU network as well as the intermediate functions between layers are continuous piecewise linear, which means that we can partition their domain of definition in closed polyhedral subsets such that they are linear on each subset. In this paper we use indifferently `linear' or `affine' to describe functions of the form $x \mapsto A x + b$, with $A \in \mathbb{R}^{n \times m}$ some matrix and $b \in \mathbb{R}^n$ some vector.

More precisely, for $m \in \mathbb{N}$, a subset $D \subset \mathbb{R}^m$ is a \textit{closed polyhedron} iif there exist $q \in \mathbb{N}$, $a_1, \dots, a_q \in \mathbb{R}^m$ and $b_1, \dots b_q \in \mathbb{R}$ such that for all $x \in \mathbb{R}^m$, 
\begin{equation} \label{polyhedron definition-main}
x \in D \quad \Longleftrightarrow \quad \begin{cases}
a_1^T x + b_1 \leq 0 \\
\vdots \\
a_q^T x + b_q \leq 0.
\end{cases}
\end{equation}
By convention, if $q=0$, we obtain an empty system of equations which is satisfied for any $x \in \mathbb{R}^m$, meaning the set $\mathbb{R}^m$ is a closed polyhedron.

We say that a function $g: \mathbb{R}^m \rightarrow \mathbb{R}^n$ is \textit{continuous piecewise linear} if there exists a finite set of closed polyhedra whose union is $\mathbb{R}^m$ and such that $g$ is linear over each polyhedron.

It is easy to show (see Proposition \ref{continuity of CPL functions} in the appendix) that this definition implies the continuity of the function, hence the `continuous' in the name. We do not require here the polyhedra to be disjoint and in fact, there are always some overlaps between the borders of adjacent polyhedra. For a given continuous piecewise linear function $g$, there are infinitely many possible sets of closed polyhedra that match the definition. Among them, we can always find one such that all the polyhedra $D$ have nonempty interior $\ring{D}$ (see Proposition \ref{proposition CPL functions} in the appendix). We call such a set admissible, as in the following definition.

\begin{df} \label{admissible set definition-main}
Let $g: \mathbb{R}^m \rightarrow \mathbb{R}^n$ be a continuous piecewise linear function. Let $\Pi$ be a set of closed polyhedra of $\mathbb{R}^m$. We say that $\Pi$ is \textit{admissible} with respect to $g$ if and only if:
\begin{equation} 
\begin{cases}
\bigcup_{D \in \Pi} D = \mathbb{R}^m, \\
\text{for all } D \in \Pi, \ g \text{ is linear on } D, \\
\text{for all } D \in \Pi, \ \ring{D} \neq \emptyset.
\end{cases}
\end{equation}
\end{df}

We now define additional functions associated to a network. Recall the layer functions $h_k$ defined in \eqref{hk-main}, that represent the actions of the network between successive layers. Let $k \in \llbracket 0, K \rrbracket$. We define the following functions:
\begin{equation}\label{definition fk and gk-main}
\begin{aligned}
f_k & = h_k \circ h_{k+1} \circ \dots \circ h_{K-1};\\
g_k & = h_0 \circ h_1 \circ \dots \circ h_{k-1}.
\end{aligned}
\end{equation}
Above, by convention, we let $f_K = id_{\mathbb{R}^{n_K}}$ and $g_0 = id_{\mathbb{R}^{n_0}}$, where $id_{\mathbb{R}^m}$ denotes the identity function on $\mathbb{R}^m$.
The function $f_k:\mathbb{R}^{n_K} \mapsto \mathbb{R}^{n_k}$ represents the mapping implemented by the network between the input layer and the layer $k$. The function $g_k: \mathbb{R}^{n_k} \mapsto \mathbb{R}^{n_0}$ represents the mapping implemented by the network between the layer $k$ and the output layer. Hence, for all $k \in \llbracket 0 , K \rrbracket$ we have $g_k \circ f_k =  f_{\mathbf{M},\mathbf{b}}$, and in particular $f_0 = g_K = f_{\mathbf{M},\mathbf{b}}$.

For any $\Omega \subset \mathbb{R}^{n_K}$, we also denote for all $k \in \llbracket 0,K \rrbracket$, 
\begin{equation}\label{def omega_k-main} \Omega_k = f_k (\Omega).
\end{equation}
In particular, $\Omega_K = f_K(\Omega) = \Omega$.

The following proposition is easy to show by induction and using the fact that the composition of two continuous piecewise linear functions is also continuous piecewise linear (see Proposition \ref{fk and gk are CPL} in the appendix).

\begin{prop}\label{f_k and g_k are piecewise linear-main}
For all $k \in \llbracket 0 , K\rrbracket$, $f_k$ and $g_k$ are continuous piecewise linear.
\end{prop}

In particular, $f_{\mathbf{M,b}}$ is continuous piecewise linear.

We say that a list of sets of closed polyhedra $\boldsymbol{\Pi} = (\Pi_1, \dots, \Pi_{K-1})$ is \textit{admissible} with respect to $(\mathbf{M}, \mathbf{b})$ iif for all $k \in \llbracket 1, K-1 \rrbracket$, the set of closed polyhedra $\Pi_k$ is admissible with respect to $g_k$. Since there always exist such $\Pi_k$ (from Proposition \ref{f_k and g_k are piecewise linear-main} and Proposition  \ref{proposition CPL functions} in Appendix \ref{Appendix-A}), there always exists an admissible list $\mathbf{\Pi}$.

\subsection{Equivalence between two parameterizations}\label{Equivalence between two parameterizations}

We are interested in sufficient conditions to identify the parameters of a network from its function. As mentioned in the introduction, some elementary operations on the parameters are well known to preserve the function of a network, so what we shall actually identify is the equivalence class of the parameters modulo these operations. There are two such operations:

\begin{itemize}
\item the permutation of neurons of a hidden layer;
\item the positive rescalings, that is, multiplying all the outward weights of a hidden neuron by a strictly positive number and dividing the inward weights by the same number.
\end{itemize}
The invariance to permutation is classical and common to many feedforward architectures. It is described in the foundational articles \cite{hecht1990algebraic, chen1993geometry}. The invariance to positive rescalings is more specific to ReLU (and homogeneous activation functions), and is also well-studied, as for instance in \cite{pourzanjani2017improving, petzka2020symmetries,phuong2020functional,pmlr-v119-rolnick20a,stock:tel-03208517}.

We give in Definition \ref{equivalence def-main} below the formalization we use for the equivalence relation modulo these operations, after introducing some notations.
For all $m \in \mathbb{N}^*$, we denote by $\mathfrak{S}_m $ the set of all permutations of $\llbracket 1 , m \rrbracket$. For any permutation $\varphi \in \mathfrak{S}_m$, we denote by $P_\varphi$ the $m \times m$ permutation matrix associated to $\varphi$, whose coefficients are defined as
\[(P_\varphi)_{i,j} = \begin{cases}  1 & \text{if } \varphi(j) = i\\ 0 & \text{otherwise.} \end{cases}\]
We also denote by $\One_m$ the vector $(1,1, \dots, 1)^T \in \mathbb{R}^m$, by $\mathbb{R}^*_+$ the set of strictly positive real numbers and by $\Id_m$ the $m \times m$ identity matrix.

\begin{df}[Equivalence between parameters]\label{equivalence def-main}
If $(\mathbf{M},\mathbf{b})$ and $(\tilde{\mathbf{M}}, \tilde{\mathbf{b}})$ are two parameterizations of a network, we say that $(\mathbf{M},\mathbf{b})$ is equivalent to $(\tilde{\mathbf{M}}, \tilde{\mathbf{b}})$, and we write $(\mathbf{M},\mathbf{b}) \sim (\tilde{\mathbf{M}}, \tilde{\mathbf{b}})$, if and only if there exist:
\begin{itemize}
\item a family of permutations $\boldsymbol{\varphi} = (\varphi_0, \dots , \varphi_{K}) \in \mathfrak{S}_{n_0} \times \dots \times \mathfrak{S}_{n_{K}}$, with $P_{\varphi_0} = \Id_{ n_0}$ and $P_{\varphi_{K}} = \Id_{ n_K}$,
\item a family of vectors $\boldsymbol{\lambda}=(\lambda^{0}, \lambda^{1}, \dots, \lambda^{K}) \in (\mathbb{R_+^*})^{n_0} \times \dots \times( \mathbb{R_+^*})^{n_{K}}$, with $\lambda^{0} = \One_{n_0}$ and $\lambda^{K} = \One_{n_K}$,
\end{itemize}
such that for all $k \in \llbracket 0 , K-1 \rrbracket$,
\begin{equation}\label{equivalence expression-main}\begin{cases}\tilde{M}^{k}= P_{\varphi_{k}} \Diag(\lambda^{k} )M^{k}\Diag(\lambda^{(k+1)})^{-1}P_{\varphi_{k+1}}^{-1} \\
\tilde{b}^{k} = P_{\varphi_k}\Diag(\lambda^{k} ) b^{k}.\end{cases} \end{equation}

The relation $\mathbf{(M,b)} \sim \mathbf{(\tilde{M}, \tilde{b})}$ is an equivalence relation \cite{pourzanjani2017improving, petzka2020symmetries,phuong2020functional}. We include a proof of this fact for completeness in Appendix \ref{Appendix-A} (see Proposition \ref{the relation is an equivalence relation}). We denote by $[\mathbf{M},\mathbf{b}]$ the equivalence class of $(\mathbf{M},\mathbf{b})$.
\end{df}
We can now formalize in Proposition \ref{equivalent networks have same function-main} the fact discussed at the beginning of the section: two equivalent parameterizations modulo permutation and positive rescaling implement the same function. As mentioned, this result is well-known \cite{pourzanjani2017improving, petzka2020symmetries,phuong2020functional,pmlr-v119-rolnick20a,stock:tel-03208517}, but we prove it for completeness in Appendix \ref{Appendix-A} (see Proposition \ref{object relations for equivalent networks} and Corollary \ref{equivalent networks have same function}).

\begin{prop}\label{equivalent networks have same function-main}
If $(\mathbf{M},\mathbf{b}) \sim (\mathbf{\tilde{M}}, \mathbf{\tilde{b}})$, then $f_{\mathbf{M},\mathbf{b}} = f_{\mathbf{\tilde{M}}, \mathbf{\tilde{b}}}$.
\end{prop}

In this article we give a set of conditions under which we have a reciprocal, i.e. if two parameterizations $(\mathbf{M},\mathbf{b})$ and $(\mathbf{\tilde{M}}, \mathbf{\tilde{b}})$ satisfying the conditions lead to the same function on a set $\Omega$, i.e.  $f_{\mathbf{M,b}}(x) = f_{\mathbf{\tilde{M},\tilde{b}}}(x)$ for all $x \in \Omega$, then they are equivalent: $\mathbf{(M,b)} \sim \mathbf{(\tilde{M}, \tilde{b})}$.

\section{Main result}\label{main-result-sec}

The core of our work is exposed in this section. It is structured as follows. In Section \ref{Section conditions-main}, we expose the conditions $\Pbf$ and in Section \ref{main theorems-sec} we state our main theorems of identifiability. Then, Section \ref{discussion of the conditions-sec} is dedicated to an extensive discussion of the conditions $\Pbf$, with motivating examples and comparison to the state of the art.

\subsection{Conditions}\label{Section conditions-main}

We expose in this section the conditions under which the main theorem holds. They are formalized in Definition \ref{Conditions P-main} and referred to as conditions $\mathbf{P}$. 

First, we introduce a few notations. We consider a network with $K \geq 2$ layers and with parameters $\mathbf{(M,b)}$, a list of sets of closed polyhedra $\boldsymbol{\Pi} = (\Pi_1, \dots, \Pi_{K-1})$ admissible with respect to $\mathbf{(M,b)}$ and a domain $\Omega \subset \mathbb{R}^{n_K}$. Recall the definitions \eqref{hk-main} and \eqref{definition fk and gk-main} of the functions $h_k$, $f_k$ and $g_k$ associated to the network. For all $k \in \llbracket 1, K-1 \rrbracket$, $g_{k}$ is continuous piecewise linear, and since $\boldsymbol{\Pi}$ is admissible with respect to $\mathbf{(M,b)}$, by definition, the set of closed polyhedra $\Pi_{k}$ is admissible with respect to $g_k$ in the sense of Definition \ref{admissible set definition-main}. For all $D \in \Pi_k$, the function $g_k$ thus coincides with a linear function on $D$. Since by definition the interior of $D$ is nonempty, we define $V^k(D) \in \mathbb{R}^{n_0 \times n_k}$ and $c^k(D) \in \mathbb{R}^{n_0}$ as the unique couple satisfying, for all $x \in D$:
\begin{equation}\label{V_k definition-main}
g_k(x) = V^k(D)x + c^k(D).
\end{equation}
For $\Omega \subset \RR^{n_K}$, recall the definition \eqref{def omega_k} of $\Omega_k$, for all $k \in [0,K]$. For any $m,n \in \mathbb{N}^*$, for any $m \times n$ matrix $\Sigma$, for any $i \in \llbracket 1 , m \rrbracket,  j \in \llbracket 1 , n \rrbracket$, we denote by $\Sigma_{i,.}$ the $i^{\text{th}}$ row vector of $\Sigma$ and by $\Sigma_{.,j}$ the $j^{\text{th}}$ column vector of $\Sigma$.  We denote $E^k_i = \{x \in \mathbb{R}^{n_k}, \ x_i = 0 \}$, and $h_k^{lin}(x) = M^kx + b^k$. For any $m \in \mathbb{N}^*$ and any subset $A \subset \mathbb{R}^m$, we denote by $\partial A$ the topological boundary with respect to the standard topology of $\mathbb{R}^m$.

\begin{df}\label{Conditions P-main}
We say that $(\mathbf{M},\mathbf{b},\Omega, \boldsymbol{\Pi})$ satisfies the conditions $\mathbf{P}$ iif for all $k \in \llbracket 1, K-1 \rrbracket$:
\begin{itemize}
\item[$\mathbf{P}.a)$] $M^k$ is full row rank;
\item[$\mathbf{P}.b)$] for all $i \in \llbracket 1 , n_k \rrbracket$, there exists $x \in \ring{\Omega}_{k+1}$ such that
\[M^k_{i,.}x + b^k_i = 0,\]
or equivalently
\[E_i^k \cap h_k^{lin}(\ring{\Omega}_{k+1}) \neq \emptyset;\]
\item[$\mathbf{P}.c)$] for all $D \in \Pi_k$, for all $i \in \llbracket 1 , n_k \rrbracket$, if $E^k_i \cap D \cap \Omega_{k} \neq \emptyset$ then $V^k_{.,i}(D) \neq 0$;
\item[$\mathbf{P}.d)$] for any affine hyperplane $H \subset \mathbb{R}^{n_{k+1}}$,
\[H \cap \ring{\Omega}_{k+1} \ \not\subset \bigcup_{D \in \Pi_k} \partial h_k^{-1}(D).\]
\end{itemize}
\end{df}

The conditions $\Pbf$ are invariant modulo equivalences of parameters. Indeed, as shown by the following proposition, if some parameters $\mathbf{(M,b)}$ satisfy the conditions $\mathbf{P}$, then all the parameters in their equivalence class satisfy them too.

\begin{prop}\label{conditions P are stable by equivalence-main}
Suppose $(\mathbf{M}, \mathbf{b})$ and $(\mathbf{\tilde{M}}, \mathbf{\tilde{b}})$ are two equivalent network parameterizations, and suppose that there exists a list $\Pibf$ admissible with respect to $(\Mbf, \bbf)$ such that $(\mathbf{M}, \mathbf{b}, \Omega, \boldsymbol{\Pi})$ satisfies the conditions $\mathbf{P}$. 

Then, there exists a list $\boldsymbol{\tilde{\Pi}}$ that is admissible with respect to $\mathbf{(\tilde{M}, \tilde{b})}$, and such that $(\mathbf{\tilde{M}}, \mathbf{\tilde{b}}, \Omega, \boldsymbol{\tilde{\Pi}})$ satisfies the conditions $\mathbf{P}$.
\end{prop}

Proposition \ref{conditions P are stable by equivalence-main} is proven as Proposition \ref{conditions P are stable by equivalence} in Appendix \ref{Appendix-B}.

\subsection{Main theorems}\label{main theorems-sec}

We have now introduced all the necessary material to expose our main result, Theorem \ref{main theorem-main}, as well as an application in terms of risk minimization in Section \ref{risk minimization corollary-sec}. 

\subsubsection{Identifiability statement}\label{main-result-subsec}

Our main theorem is the following one. We provide a sketch of the proof in Section \ref{sketch-sec}. For the complete proof, see Theorem \ref{main theorem} in Appendix $\ref{Appendix-B}$ and its proof in Section \ref{Proof of the main theorem-apdx}.

\begin{thm}\label{main theorem-main}
Let $K \in \mathbb{N}$, $K \geq 2$. Suppose we are given two networks with $K$ layers, identical number of neurons per layer, and with respective parameters $(\mathbf{M}, \mathbf{b})$ and $(\mathbf{\tilde{M}}, \mathbf{\tilde{b}})$. Assume $\boldsymbol{\Pi}$ and $\boldsymbol{\tilde{\Pi}}$ are two lists of sets of closed polyhedra that are admissible with respect to $(\mathbf{M}, \mathbf{b})$ and $(\mathbf{\tilde{M}}, \mathbf{\tilde{b}})$ respectively. Denote by $n_K$ the number of neurons of the input layer, and suppose we are given a set $\Omega \subset \mathbb{R}^{n_K}$ such that $(\mathbf{M}, \mathbf{b},\Omega,\boldsymbol{\Pi})$ and $(\mathbf{\tilde{M}}, \mathbf{\tilde{b}},\Omega, \boldsymbol{\tilde{\Pi}})$ satisfy the conditions $\mathbf{\mathbf{P}}$, and such that, for all $x \in \Omega$:
 \[f_{\mathbf{M},\mathbf{b}}(x) = f_{\tilde{\mathbf{M}}, \tilde{\mathbf{b}}}(x).\]
 
Then:
\[(\mathbf{M},\mathbf{b}) \sim (\mathbf{\tilde{M}}, \mathbf{\tilde{b}}).\]
\end{thm}

As mentioned before, this theorem can be seen as a partial reciprocal to Proposition \ref{equivalent networks have same function-main}. Indeed, the latter shows that two networks with equivalent parameters modulo permutation and positive rescaling implement the same function. In other words, parameter equivalence implies functional equivalence of the networks. In Theorem \ref{main theorem-main}, we state that under the conditions $\mathbf{P}$, functional equivalence (on a given domain $\Omega$) implies parameter equivalence modulo permutation and positive rescaling.

\subsubsection{An application to risk minimization}\label{risk minimization corollary-sec}

Assume we are given a couple of input-output variables $(X,Y)$ generated by a ground truth network with parameters $(\mathbf{M}, \mathbf{b})$:
\[Y = f_{\mathbf{M,b}}(X).\]

We can use Theorem \ref{main theorem-main} to show that the only way to bring the population risk to $0$ is to find the ground truth parameters -modulo permutation and positive rescaling. 

Indeed, let $\Omega \subset \mathbb{R}^{n_K}$ be a domain that is contained in the support of $X$, and suppose $L : \mathbb{R}^{n_0} \times \mathbb{R}^{n_0} \rightarrow \mathbb{R}_+$ is a loss function such that $L(y,y') = 0 \Rightarrow y = y'$. Consider the population risk:
\[R(\mathbf{\tilde{M}}, \mathbf{\tilde{b}}) = \mathbb{E}[L(f_{\mathbf{\tilde{M}},\mathbf{\tilde{b}}}(X), Y)].\]
We have the following result.

\begin{coro}\label{risk minimization formulation-main}
Suppose there exists a list of sets of closed polyhedra $\boldsymbol{\Pi}$ admissible with respect to $(\mathbf{M}, \mathbf{b})$ such that $(\mathbf{M}, \mathbf{b},\Omega, \boldsymbol{\Pi})$ satisfies the conditions $\mathbf{P}$. 

If $(\mathbf{\tilde{M}}, \mathbf{\tilde{b}}) $ is such that there exists a list $\boldsymbol{\tilde{\Pi}}$ admissible with respect to $(\mathbf{\tilde{M}}, \mathbf{\tilde{b}})$ such that $(\mathbf{\tilde{M}}, \mathbf{\tilde{b}},\Omega, \boldsymbol{\tilde{\Pi}})$ satisfies the conditions $\mathbf{P}$, and if $(\mathbf{M}, \mathbf{b}) \not\sim (\mathbf{\tilde{M}}, \mathbf{\tilde{b}})$, then:
\[R(\mathbf{\tilde{M}}, \mathbf{\tilde{b}}) > 0.\]
\end{coro}

 For the proof, see Corollary \ref{risk minimization formulation} in Appendix \ref{Appendix-B} and its proof in Section \ref{Proof of the corollary-sec}.

\subsection{Discussion on the conditions}\label{discussion of the conditions-sec}

This section is dedicated to discussing the conditions $\Pbf$. We start by explaining the different conditions $\Pbf.a) - \Pbf.d)$ and their purpose in Section \ref{the conditions explained-sec}. Then, in Section \ref{illustrative counter-examples-sec}, we provide counter-examples illustrating how non-identifiability arises when they are not satisfied. Finally, we compare the conditions $\Pbf$ to the state of the art in Sections \ref{comparisonexisting-sec} and \ref{comparative example-sec}.

\subsubsection{The conditions explained}\label{the conditions explained-sec}

Let us explain the conditions $\mathbf{P}$. The first condition, $\mathbf{P}.a)$, requires the matrix $M^k \in \RR^{n_k \times n_{k+1}}$ to have full row rank. This implies that for all $k \in \llbracket 1, K-1 \rrbracket$, the layer $k$ has no more neurons than its predecessor, the layer $k+1$:
\[n_{k} \leq n_{k+1}.\]
Once this is satisfied, the condition is mild in the sense that it is satisfied for all matrices except a set of matrices of empty Lebesgue measure. 

As a first remark about $\Pbf.b)$, notice that by taking $k=K-1$, it implies that $\ring{\Omega} = \ring{\Omega}_K \neq \emptyset$. Thus, in the main result, the set $\Omega$ over which the function implemented by the network is assumed to be known needs to have nonempty interior. In particular, $\Omega$ cannot be a finite sample set. This limitation is already present in \cite{pmlr-v119-rolnick20a}, which assumes an access to the function on the whole input space and \cite{phuong2020functional} which considers the function of the network on a bounded open nonempty domain. However, as we discuss in the conclusion, it seems possible to establish a result for a finite $\Omega$, and the conditions formulated here should be a basis for future work.

The conditions $\Pbf.b), \Pbf.c)$ and $\Pbf.d)$ must be satisfied for all $k \in \lb 1, K-1 \rb$, but to give a sense of them, let us see what they mean for $k = K-1$.

\begin{figure}
\center
\begin{tabular}{l|l}
\includegraphics[scale=0.51]{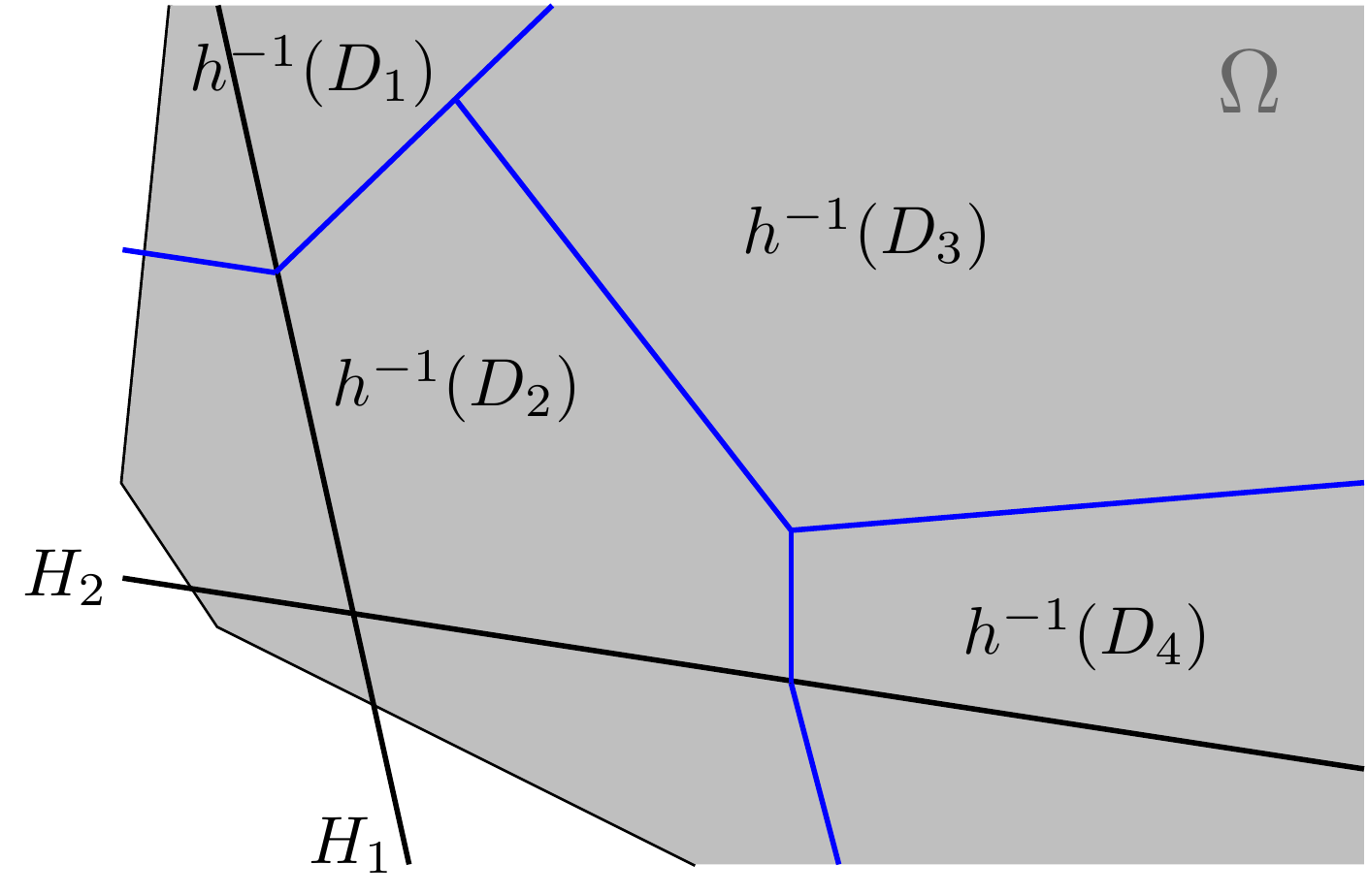} \hspace{15pt} & \hspace{15pt} \includegraphics[scale=0.51]{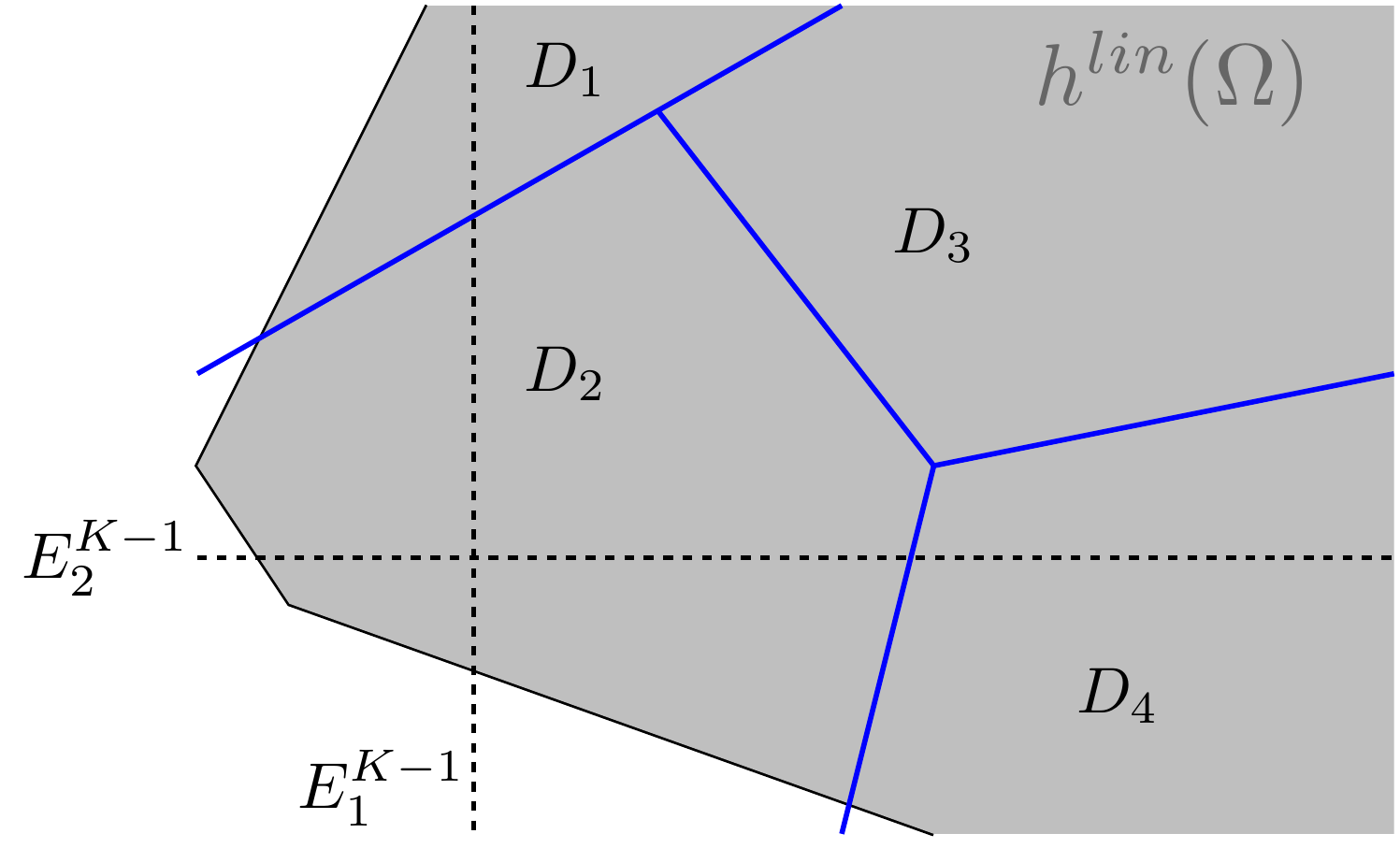} \\
$\mathbb{R}^{n_{K}}$ \hspace{15pt} &  \hspace{15pt} $\mathbb{R}^{n_{K-1}}$
\end{tabular}
\caption{Left. In $\RR^{n_K}$, the inverse image by $h_{K-1}$ of the polyhedra $D \in \Pi_{K-1}$. To make the figure lighter we write $h$ instead of $h_{K-1}$. The grey zone represents $\Omega$. For $i \in \{1,2\}$, $H_i$ is the hyperplane defined by the equation $M^{K-1}_{i,.}x + b^{K-1}_i = 0$. As a direct consequence, we have $h(H_i) \subset E_i^{K-1}$. Right. In $\RR^{n_{K-1}}$, the admissible polyhedra $D \in \Pi_{K-1}$ with respect to $g_{K-1}$. The grey zone corresponds to the image $h^{lin}(\Omega) = M^{K-1} \Omega + b^{K-1}$.}
\label{Geometry of the functions-main}
\end{figure}

As explained in Section \ref{CPL functions-sec}, the function implemented by a ReLU network is continuous piecewise linear: we can divide the input space $\mathbb{R}^{n_{K}}$ into polyhedral regions, over each of which the function is linear. We take advantage of this structure to acquire information about the parameters of the network. The boundaries of the polyhedral regions are of particular interest. They are made of pieces of hyperplanes, and they roughly correspond to the points where the function implemented by the network is not differentiable. We use this non differentiability property to identify the boundaries. We go from one linear region to another when there is a change of sign in the pre-activation value (input of $\sigma$) of one hidden neuron. The boundary between two linear regions is thus associated to a particular neuron of a particular hidden layer.

We separate the function implemented by the first layer of the network and the function implemented by the rest of the layers thanks to the functions defined in \eqref{hk-main} and \eqref{definition fk and gk-main}, writing 
\[f_{\Mbf, \bbf} = g_{K} = g_{K-1} \circ h_{K-1},\]
\[\mathbb{R}^{n_{K}} \ \overset{h_{K-1}}{\longrightarrow} \ \mathbb{R}^{n_{K-1}} \ \overset{g_{K-1}}{\longrightarrow} \ \mathbb{R}^{n_0}.\]

The goal is first to identify the weights and bias of the first layer, $M^{K-1}$ and $b^{K-1}$. To do so, we focus on the boundaries associated to the neurons in the first hidden layer. These `first-order' boundaries are hyperplanes defined by the equations $M^{K-1}_{i,.} x + b^{K-1}_i = 0$, for all $i \in \llbracket 1, n_{K-1} \rrbracket$.  The conditions $\mathbf{P}.b)$, $\mathbf{P}.c)$ and $\mathbf{P}.d)$ are made to ensure that we are able to identify the hyperplanes, and consequently, the parameters $M^{K-1}$ and $b^{K-1}$. The two relevant spaces to visualize the conditions are the input space, $\mathbb{R}^{n_{K}}$, and the first hidden space, $\mathbb{R}^{n_{K-1}}$, which are represented in Figure \ref{Geometry of the functions-main}. Let us explain the conditions $\Pbf.b), \Pbf.c)$ and $\Pbf.d)$.

\begin{itemize}
\item[$\Pbf.b)$] The condition $\mathbf{P}.b)$ in the case $k = K-1$ requires the hyperplane defined by the equation $M^{K-1}_{i,.} x + b^{K-1} = 0$ to intersect $\ring{\Omega}_{K} = \ring{\Omega}$. Indeed, we only consider the function implemented by the network over $\Omega$, so the hyperplane must intersect $\ring{\Omega}_{K}$ in order to be detectable as a non differentiability. In the example of Figure \ref{Geometry of the functions-main}, we see that the two such hyperplanes, which are $H_1$ and $H_2$, intersect $\Omega$, so the condition is satisfied.

\item[$\Pbf.c)$] Consider a polyhedron $D \in \Pi_{K-1}$. The function $g_{K-1}$ is linear over $D$, and using the notations defined in \eqref{V_k definition-main}, we have for all $u \in D$, 
\begin{equation}\label{expression of g_(K-1) over D}
g_{K-1}(u) = V^{K-1}(D) u + c^{K-1}(D).
\end{equation} 
For all $x \in \mathbb{R}^{n_K}$ such that $h_K(x) \in D$, using \eqref{expression of g_(K-1) over D} we obtain:
\[f_{\Mbf, \bbf}(x) = g_{K-1} \circ h_{K-1}(x) = \sum_{i=1}^{n_{K-1}} V^{K-1}_{.,i} (D) \sigma \left(M^{K-1}_{i,.}x + b^{K-1}_i \right) + c^{K-1}(D). \] 
In particular, at the points $x$ such that $M^{K-1}_{i,.}x + b^{K-1}_i = 0$, the function $\sigma(M^{K-1}_{i,.} x + b^{K-1}_i)$ is not differentiable, and this non differentiability can only be reflected in the function $f_{\Mbf, \bbf}$ if $V^{K-1}_{.,i} (D) \neq 0$. The condition $\mathbf{P}.c)$ ensures that. In the example of Figure \ref{Geometry of the functions-main} (right part), we see that $D_1$ intersects $E^{K-1}_1$ so to satisfy $\mathbf{P}.c)$, we must have $V^{K-1}_{.,1}(D_1) \neq 0$. Similarly, $D_2$ intersects $E^{K-1}_1$ and $E^{K-1}_2$ so we must have $V^{K-1}_{.,1}(D_2) \neq 0$ and $V^{K-1}_{.,2}(D_2) \neq 0$, and the polyhedron $D_4$ intersects $E_2^{K-1}$ so we must have $V^{K-1}_{.,2}(D_4) \neq 0$.

\item[$\Pbf.d)$] For the last condition, $\mathbf{P}.d)$, we consider the inverse images $h_{K-1}^{-1}(D)$, for all the polyhedra $D \in \Pi_{K-1}$. Since $h_{K-1}$ is piecewise linear and $D$ is a closed polyhedron,  $h_{K-1}^{-1}(D)$ is a finite union of closed polyhedra (see the first point of Proposition \ref{boundary included in a union of hyperplanes} in the appendix). In particular, its boundary $\partial h_{K-1}^{-1}(D)$ is made of pieces of hyperplanes. We require the union of these boundaries not to contain any full hyperplane (within the domain $\Omega$). In the example of Figure \ref{Geometry of the functions-main}, the condition is satisfied.
\end{itemize}

\subsubsection{Illustrative counter-examples}\label{illustrative counter-examples-sec}

To illustrate the necessity for the conditions in $\Pbf$, we give for each of the conditions $\Pbf.a) - \Pbf.d)$ a simple example of a parameterization $(\Mbf, \bbf)$ and a set $\Omega$ which do not satisfy it, and we show that $(\Mbf, \bbf)$  is not identifiable by constructing a parameterization $(\tilde{\Mbf}, \tilde{\bbf})$ that is not equivalent to $(\Mbf, \bbf)$, but such that $f_{\tilde{\Mbf}, \tilde{\bbf}}$ coincides with $f_{\Mbf, \bbf}$ over $\Omega$. These four examples illustrate the behaviors we want to prevent with the conditions $\Pbf$.

\begin{exnum}
We consider an architecture with one hidden layer, i.e. $K=2$, with $n_2 = 2$, $n_1 = 3$, $n_0 = 1$. We consider the parameterization $(\Mbf, \bbf)$ defined as follows.
\[M^1 = \begin{pmatrix} 0 & 2 \\ 1 & -1 \\ -1 & - 1 \end{pmatrix} \qquad b^1 = \begin{pmatrix} 0 \\ 0 \\ 0 \end{pmatrix},\]
\[M^0 = \begin{pmatrix} 1 & 1 & 1 \end{pmatrix} \qquad b^0 = 0.\]
For this example, we consider $\Omega = \RR$.

The condition $\Pbf.a)$ is not satisfied: the matrix $M^1$ cannot have full row rank since its dimension is $3 \times 2$, and more specifically we have the relation 
\begin{equation}\label{example1-the sum of the lines of M^1 are zero}
M^1_{1,.} + M^1_{2,.} + M^1_{3,.} = 0.
\end{equation}
Let us define $\tilde{M}^1 = - M^1$ and $\tilde{\Mbf} = (\tilde{M}^1, M^0)$. Let us show that $f_{\tilde{\Mbf}, \bbf} = f_{\Mbf, \bbf}$.

Let $x \in \RR^2$. We have
\[f_{\Mbf, \bbf}(x) = \sigma\left( M^1_{1,.} x \right) + \sigma\left( M^1_{2,.} x\right) + \sigma\left( M^1_{3,.} x\right).\]
There exist activations $\epsilon_1, \epsilon_2, \epsilon_3 \in \{ 0,1 \}$, depending on $x$, such that\\
\begin{equation}\label{example1-expression of f}f_{\Mbf, \bbf}(x) = \epsilon_1 M^1_{1,.} x  + \epsilon_2 M^1_{2,.} x + \epsilon_3 M^1_{3,.} x.
\end{equation}

Since $\tilde{M}^1 = - M^1$ and $b^1 = 0$, the signs of the activations are switched in $f_{\tilde{\Mbf}, \bbf}$ and thus
\begin{equation*}
\begin{aligned}
f_{\tilde{\Mbf}, \bbf}(x) & = (1 - \epsilon_1) \tilde{M}^1_{1,.} x  + (1 - \epsilon_2) \tilde{M}^1_{2,.} x + (1 - \epsilon_3) \tilde{M}^1_{3,.} x \\
 & = \sum_{i=1}^3 (1 - \epsilon_i) (-M^1_{i,.}x)\\
 & = \sum_{i=1}^3\epsilon_i M^1_{i,.}x - (\sum_{i=1}^3 M^1_{i,.} )x \\
 & = f_{\Mbf, \bbf}(x),\\
\end{aligned}
\end{equation*}
where we obtain the last equality thanks to \eqref{example1-the sum of the lines of M^1 are zero} and \eqref{example1-expression of f}.

Now since only the positive rescalings are authorized, $(\tilde{\Mbf}, \bbf)$ is not equivalent to $(\Mbf, \bbf)$, which shows that $(\Mbf, \bbf)$ is not identifiable modulo permutation and rescaling.
\end{exnum}

\begin{exnum}
 Let us consider a very simple architecture with one hidden layer and only one neuron per layer, i.e. $n_2 = n_1 = n_0 = 1$. We consider the parameterization $(\Mbf, \bbf_a)$, for $a > 0$, defined by
\[M^1 = 1, \quad b_a^1 = a, \quad M^0 = 1, \quad b_a^0 = -a.\]

The function implemented by the network satisfies, for all $x \in \RR,$ 
\begin{equation}\label{example2-expression of f}
f_{\Mbf, \bbf_a} (x) = \sigma(x + a ) - a = \begin{cases} -a & \text{if } x < - a \\ x & \text{if } x \geq - a. \end{cases}\end{equation}

Let us consider $\Omega = [1, + \infty[$. With such a choice of $\Omega$, none of the parameterizations $(\Mbf, \bbf_a)$  satisfy $\Pbf.b)$, because for all $x \in \Omega$, $M^1x + b_a^1 = x + a > 0$.

For any $a > 0$, for any $x \in \Omega$, we have $x \geq 1 > -a$, so \eqref{example2-expression of f} shows that $f_{\Mbf, \bbf_a}(x) = x$, i.e. the functions implemented by the parameterizations $(\Mbf, \bbf_a)$ all coincide over $\Omega$. However, since \eqref{example2-expression of f} shows they do not implement the same function over $\RR$, Proposition \ref{equivalent networks have same function-main} shows they are not equivalent.
\end{exnum}
\begin{exnum}\label{counter-example P.c)}
We consider an architecture with $2$ hidden layers, that is $K=3$, and again one neuron per layer: $n_3 = n_2 = n_1 = n_0 = 1$. Let us consider the parameterizations $(\Mbf, \bbf_a)$, defined for $a>0$ by
\begin{align*}M^2 = 1, \qquad & b^2_a = a,\\
M^1 = 1, \qquad & b^1_a = -1 - a,\\
M^0 = 1, \qquad & b^0_a = 0.
\end{align*}

We consider $\Omega = \RR$. Let us show that for any $a > 0$ and any admissible $\Pibf$, the condition $\Pbf.c)$ is not satisfied by $(\Mbf, \bbf, \Omega, \Pibf)$ in the case $k = 2$.

Indeed, we have $h_{2,a} (x) = \sigma(M^2x + b^2_a) = \sigma(x+a)$, and thus $\Omega_2 = h_{2,a}(\Omega) = \RR_+$. Further, the expression of $g_{2,a}$ is
\[g_{2,a}(x) = M^0 \sigma(M^1x + b^1_a) + b^0_a = \sigma(x-1-a) = \begin{cases} 0 & \text{if } x \leq 1 + a \\ x-1-a & \text{if } x > 1 + a .\end{cases}\]
For any set of closed polyhedra $\Pi_2$ admissible with respect to $g_{2,a}$, a polyhedron $D \in \Pi_2$ intersecting $E_1^2 = \{ 0 \}$ must satisfy $V^2(D) = 0$ since $g_{2,a}(x) = 0$ for $x \in ] - \infty, 1+a ]$. This contradicts $\Pbf.c)$ for $k=2$.

To exhibit functionally equivalent parameterizations, we now show that for all $a > 0$ and $x \in \RR$, 
\begin{equation}\label{example3-f(x)=sigma(x-1)}
f_{\Mbf, \bbf_a}(x) = \sigma(x - 1).
\end{equation}
Indeed, let $x \in \RR$.
\begin{itemize}
\item if $x \in ] - \infty, -a [$, we have $\sigma(M^2x + b^2_a) = \sigma(x+a) = 0$, so 
\begin{equation*}
\begin{aligned}
f_{\Mbf, \bbf_a}(x) & = M^0 \sigma(M^1 \cdot 0 + b^1_a) + b^0_a \\
& =  \sigma( b^1_a) \\
& = 0 = \sigma(x-1).
\end{aligned}
\end{equation*}
\item if $x \in [-a, + \infty [$, we have $\sigma(M^2x + b^2_a) = \sigma(x+a) = x+a$, so 
\begin{equation*}
\begin{aligned}
f_{\Mbf, \bbf_a}(x) & = M^0 \sigma ( M^1 (x + a) + b^1_a) + b^0_a \\
& = \sigma(x + a - 1 - a) \\
& = \sigma(x - 1).
\end{aligned}
\end{equation*}
\end{itemize}

This shows \eqref{example3-f(x)=sigma(x-1)}. The function $f_{\Mbf, \bbf_a}$ is therefore independent of $a > 0$, but if $a \neq a'$, $(\Mbf, \bbf_a) \not\sim (\Mbf, \bbf_{a'})$.

The lack of identifiability comes here from the fact that we do not `observe' the non differentiability induced by the first hidden neuron, because $V^2(D) = 0$ for $D$ containing $0$. Indeed, if $\Pbf.c)$ was satisfied, we would observe a non differentiability at the point at which the sign of $M^2x + b_a^2$ changes, which is $x = - a$, and we thus would have $f_{\Mbf, \bbf_a} \neq f_{\Mbf, \bbf_{a'}}$ for $a \neq a'$.

We remark that here, even if $\Pbf.c)$ was satisfied, the condition $\Pbf.d)$ would not be satisfied, as we see next in Example \ref{counter-example P.d)}.
\end{exnum}

\begin{exnum}\label{counter-example P.d)}
We consider again the architecture of Example \ref{counter-example P.c)}, wih two hidden layers and one neuron per layer. This time, we consider the parameterizations $(\Mbf, \bbf)$ and $(\tilde{\Mbf}, \tilde{\bbf})$, defined by
\begin{align*}& M^2 = 1, & \qquad & b^2 = 0, \\
& M^1 = -1, & \qquad & b^1 = 1, \\
& M^0 = -1, & \qquad & b^0 = 0,
\end{align*}
and
\begin{align*}& \tilde{M}^2 = -1, & \qquad & \tilde{b}^2 = 1,\\
& \tilde{M}^1 = -1, & \qquad & \tilde{b}^1 = 1,\\
& \tilde{M}^0 = 1, & \qquad & \tilde{b}^0 = -1.
\end{align*}
We can remark without waiting further that $(\Mbf, \bbf) \not\sim ( \tilde{\Mbf}, \tilde{\bbf})$, for instance because $b^0 = 0$, and $\tilde{b}^0 = -1$, and the rescalings do not permit such a transformation.

Let $\Omega = \RR$. Let us consider the sets $\Pi_2 = \{ ] - \infty, 1], [1, + \infty [\}$, $\Pi_1 = \{ \RR \}$ and the list $\Pibf = ( \Pi_1, \Pi_2)$. After showing that $\Pibf$ is admissible with respect to $(\Mbf, \bbf)$, we will first show that $(\Mbf, \bbf)$ does not satisfy the condition $\Pbf.d)$ and we will then show that $f_{\Mbf,\bbf} = f_{\tilde{\Mbf}, \tilde{\bbf}}$.

Let us show that $\Pibf$ is admissible with respect to $(\Mbf, \bbf)$. Indeed, for all $x \in \RR$, we have
\[g_2(x) = M^0 \sigma(M^1 x + b^1) + b^0 = - \sigma(-x + 1).\]
The function $g_2$ is linear over both the intervals $]-\infty, 1]$ and $[1, + \infty[$, so $\Pi_2$ is admissible with respect to $g_2$. The function $g_1$ is linear, so $\Pi_1$ is admissible with respect to $g_1$. 

Let us now show that $(\Mbf, \bbf, \Omega, \Pibf)$ does not satisfy the condition $\Pbf.d)$. Let us first determine $\bigcup_{D \in \Pi_2} \partial h_2^{-1} (D)$. Since $h_2(x) = \sigma(x)$, we have $h_2^{-1}(]- \infty, 1[) = ] - \infty, 1]$ and $h_2^{-1}([1, + \infty[) = [1, + \infty[$. Hence, 
\[\bigcup_{D \in \Pi_2} \partial h_2^{-1} (D) = \{1\} .\]
Now, since $\ring{\Omega}_3 = \ring{\Omega} = \RR$, we have 
\[\ring{\Omega}_3 \cap \{1 \} \subset \bigcup_{D \in \Pi_2} \partial h_2^{-1} (D),\]
and since $\{1\}$ is an affine hyperplane of $\RR$, this shows that $\Pbf.d)$ is not satisfied for $k=2$.

Let us now show that for all $x \in \Omega = \RR$, we have
\begin{equation}\label{example4-expression of f}
f_{\Mbf, \bbf} (x ) = f_{\tilde{\Mbf}, \tilde{\bbf}}(x) = \begin{cases} \sigma(x)-1 & \mbox{if } x \leq 1 \\ 0 & \text{if } x > 1.
\end{cases}
\end{equation}
Let us first determine $f_{\Mbf, \bbf}(x)$, for $x \in \RR$. We have 
\[f_{\Mbf, \bbf}(x) = M^0 \sigma(M^1 \sigma(M^2x + b^2) + b^1) + b^0 = - \sigma(-\sigma(x) + 1).\]
\begin{itemize}
    \item If $x \leq 1$, then $\sigma(x) \leq 1$ and thus $-\sigma(x) + 1 \geq 0$. Thus, $- \sigma(-\sigma(x) + 1) = \sigma(x) - 1$, and  $f_{\Mbf, \bbf}(x) = \sigma(x) - 1.$
 %   \begin{align*}
 %       f_{\Mbf, \bbf}(x) = \sigma(x) - 1.
%    \end{align*}
    \item If $x > 1$, then $-\sigma(x) + 1 < 0$ and thus $- \sigma(-\sigma(x) + 1) = 0$. We thus have $f_{\Mbf, \bbf}(x) = 0.$
%    \begin{align*}
%        f_{\Mbf, \bbf}(x) = 0.
%    \end{align*}
\end{itemize}
Let us now determine $f_{\tilde{\Mbf}, \tilde{\bbf}}(x)$, for $x \in \RR$. We have
\[f_{\tilde{\Mbf}, \tilde{\bbf}}(x) = \tilde M^0 \sigma(\tilde M^1 \sigma(\tilde M^2x + \tilde b^2) + \tilde b^1) + \tilde b^0  = \sigma(-\sigma(-x + 1) + 1) -1.\]
\begin{itemize}
    \item If $x \leq 1$, then $ -x + 1 \geq 0$ and thus $       f_{\tilde{\Mbf}, \tilde{\bbf}}(x)  = \sigma((x -1) + 1) -1  = \sigma(x) - 1$.
 %   \begin{align*}
 %       f_{\tilde{\Mbf}, \tilde{\bbf}}(x) & = \sigma((x -1) + 1) -1  = \sigma(x) - 1.
 %   \end{align*}
    \item If $x > 1$, then $ -x + 1 \leq 0$ and thus $f_{\tilde{\Mbf}, \tilde{\bbf}}(x)  = \sigma(1) - 1 = 0$.
%    \begin{align*}
%        f_{\tilde{\Mbf}, \tilde{\bbf}}(x)  = \sigma(1) - 1 = 0.
%    \end{align*}
\end{itemize}
This shows \eqref{example4-expression of f}, and as a consequence, $(\Mbf, \bbf)$ is not identifiable.

In this example, the lack of identifiability comes from the fact that the sets of non differentiabilities induced by the first and the second layer are indistinguishable: they are both reduced to a point. This will always be the case for networks with only one neuron per layer and more than one hidden layer. When the input dimension is 2 or higher and the condition $\Pbf.d)$ is satisfied, the non differentiabilities induced by neurons in the first hidden layer are the only ones that correspond to full hyperplanes, and this is how they can be identified, as illustrated for instance in the example of Section \ref{comparative example-sec}.

\end{exnum}

\subsubsection{Comparison with the existing work}\label{comparisonexisting-sec}

To our knowledge, there are only two existing results on global identifiability of deep ReLU networks (with bias), as we consider here, exposed in the recent contributions \cite{phuong2020functional} and \cite{pmlr-v119-rolnick20a}. Let us compare our hypotheses with theirs.

The authors of \cite{phuong2020functional} introduce two notions: the notion of \textit{general} network and the notion of \textit{transparent network}. They note the fact that some boundaries of non differentiablity bend over some others to build a graph of dependency. The main result in \cite{phuong2020functional} applies to networks whose number of neurons per layer $n_k$ is non-increasing, as is the case in the present paper, that are transparent and general, and for which the graphs of dependency of the functions $g_{k}$ satisfy additional technical conditions. 

It can be verified that these hypotheses imply our conditions $\mathbf{P}.a)$, $\mathbf{P}.b)$ and $\mathbf{P}.c)$, which makes $\mathbf{P}.a)$, $\mathbf{P}.b)$ and $\mathbf{P}.c)$ more applicable.

When it comes to our last condition $\mathbf{P}.d)$, it can be compared to the technical conditions on the graph of dependency. These conditions address the way the boundaries associated to some neurons bend over the boundaries associated to neurons in previous layers. $\mathbf{P}.d)$ and this set of conditions are different, and neither implies the other. 

The result exposed in \cite{pmlr-v119-rolnick20a} has a main strength compared to \cite{phuong2020functional} and to us: it does not require the number of neurons per layer to be non-increasing. However, when it comes to the intersection of boundaries of linear regions, it requires each boundary, associated to some neuron, to intersect the boundaries associated to all the neurons in the previous layer, which appears to be a strong hypothesis to us. In comparison, we ask each boundary to intersect at least one of the boundaries associated to a neuron in a previous layer. Also, in \cite{pmlr-v119-rolnick20a}, the function is supposed to be known on the whole input space, while \cite{phuong2020functional} as well as us propose conditions on a domain $\Omega$ such that the knowledge of the function on $\Omega$ is enough. In both cases $\Omega$ has nonempty interior. \cite{stock2022embedding, bonalocal} open the way for considering a finite $\Omega$ by giving conditions of local identifiability in that case. To our knowledge global identifiability from a finite set has not been tackled yet for deep ReLU networks.

\subsubsection{A simple comparative example}\label{comparative example-sec}

To shed a better light on the interest of the conditions $\Pbf$, we describe in this section a simple network parameterization for which the conditions $\Pbf$ apply, in contrast to the conditions described in \cite{phuong2020functional, pmlr-v119-rolnick20a}. 

Let us consider a network architecture with $2$ hidden layers (i.e. $K = 3$) and $2$ neurons per layer, except the output layer containing $1$ neuron: $n_3 = n_2 = n_1 = 2$ and $n_0 = 1$. Let us consider the parameterization $\mathbf{(M,b)}$ defined by 
\[M^2 = \begin{pmatrix}
1 & 0 \\ 0 & 1
\end{pmatrix}, \quad M^1 = \begin{pmatrix}
1 & -1 \\ -1 & 2
\end{pmatrix}, \quad M^0 = \begin{pmatrix}
1 & 1
\end{pmatrix}, \]

\[ b^2 = \begin{pmatrix}
0 \\ 0
\end{pmatrix}, \quad b^1 = \begin{pmatrix}
-1 \\ 2
\end{pmatrix}, \quad b^0 = 0. \]
The network implements a function $f_{\mathbf{M,b}} : \RR^2 \rightarrow \RR$. Here we simply consider $\Omega = \RR^2$. 

First, we are going to show that there exists a list $\Pibf$ that is admissible with respect to $(\Mbf, \bbf)$, and such that $(\Mbf, \bbf, \Omega, \Pibf)$ satisfies the conditions $\Pbf$. Then, we shall discuss why this network parameterization does not satisfy the conditions in \cite{phuong2020functional, pmlr-v119-rolnick20a}.

Let us define the list $\Pibf$ as follows. For $\epsilon_1, \epsilon_2 \in \{ -1, 1 \}$, we denote by $D_{\epsilon_1, \epsilon_2}$ the closed polyhedron satisfying, for all $x \in \RR^2$:
\begin{equation}\label{definition of D_epsilon} x \in D_{\epsilon_1, \epsilon_2} \quad \Leftrightarrow \quad \begin{cases} \epsilon_1 \left( M^1_{1, .} x + b^1_1 \right) \geq 0 \\ \epsilon_2 \left( M^1_{2,.} x + b^1_2 \right) \geq 0. \end{cases} \end{equation}
These $4$ polyhedra are displayed in Figure \ref{figure - reciprocal images of polyhedrons}. In other words, the polyhedron to which $x$ belongs depends on the sign of both components of the vector $M^1x + b^1$. We define the set $\Pi_2 = \{ D_{1,1}, D_{1, -1}, D_{-1, 1}, D_{-1, -1} \}$.
We also define the set $\Pi_1 = \{\RR^2\}$, containing the single polyhedron $\RR^2$, and we denote $\Pibf = (\Pi_1, \Pi_2)$. Let us show that $\Pibf$ is admissible with respect to $(\Mbf, \bbf)$. Indeed, the closed polyhedra of $\Pi_2$ cover $\RR^2$. Furthermore, their interior is nonempty. Finally, for all $x \in \RR^2$, we have
\begin{align}\label{example-expression of g2}
g_2(x) & = M^0 \sigma \left(M^1x + b^1 \right) + b^0 \nonumber \\
& = \sigma \left( M^1_{1,.} x + b^1_1 \right) + \sigma \left( M^1_{2,.} x + b^1_2 \right). 
\end{align}
We derive from \eqref{example-expression of g2}, \eqref{definition of D_epsilon} and the definition of the ReLU activation that for all $D_{\epsilon_1, \epsilon_2} \in \Pi_2$, the function $g_2$ is affine over $D_{\epsilon_1, \epsilon_2}$ of the form $g_2(x) = V^2(D_{\epsilon_1, \epsilon_2})x + c^2(D_{\epsilon_1, \epsilon_2})$, with the following values
\begin{equation}\label{example-values of V^2 and c^2}
\begin{aligned}
&V^2(D_{1,1})  = M^1_{1,.} + M^1_{2,.} = \begin{pmatrix} 0 & 1 \end{pmatrix} & \qquad & c^2(D_{1,1}) =  b^1_1 + b^1_2 = 1  \\
&V^2(D_{1,-1})  = M^1_{1,.} = \begin{pmatrix} 1 & -1 \end{pmatrix}  & & c^2(D_{1,-1}) = b^1_1 = -1\\
&V^2(D_{-1,1})  =  M^1_{2,.} = \begin{pmatrix} -1 & 2 \end{pmatrix}  & & c^2(D_{-1,1}) =  b^1_2 = 2\\
&V^2(D_{-1,-1})  =  \begin{pmatrix} 0 & 0 \end{pmatrix}  & & c^2(D_{-1,-1}) =0.
\end{aligned}
\end{equation}
This shows that the set of closed polyhedra $\Pi_2$ is admissible with respect to $g_2$, and the values in \eqref{example-values of V^2 and c^2} correspond to those of the definition \eqref{V_k definition-main}. Moreover, since $g_1$ is affine, the set $\Pi_1$ is trivially admissible with respect to $g_1$. We conclude that the list $\Pibf$ is admissible with respect to $(\Mbf, \bbf)$.

Let us show that $(\Mbf, \bbf, \Omega, \Pibf)$ satisfies the conditions $\Pbf$.

The conditions $\mathbf{P}$ must hold for $k \in \lb 1, K-1 \rb$, so in our case, for $k =2$ and $k=1$. To check them, we will need to compute $\Omega_3$ and $\Omega_2$. Recalling the definition in \eqref{def omega_k-main}, we have $\Omega_3 = \Omega = \RR^2$. Then, since $h_2(x) = \sigma\left( M^2x + b^2 \right) = \sigma(x)$, we have $\Omega_2 = \sigma(\RR^2) = (\RR_+)^2$.

\begin{figure}
\center
\begin{tabular}{c|cc}
\includegraphics[scale=0.6]{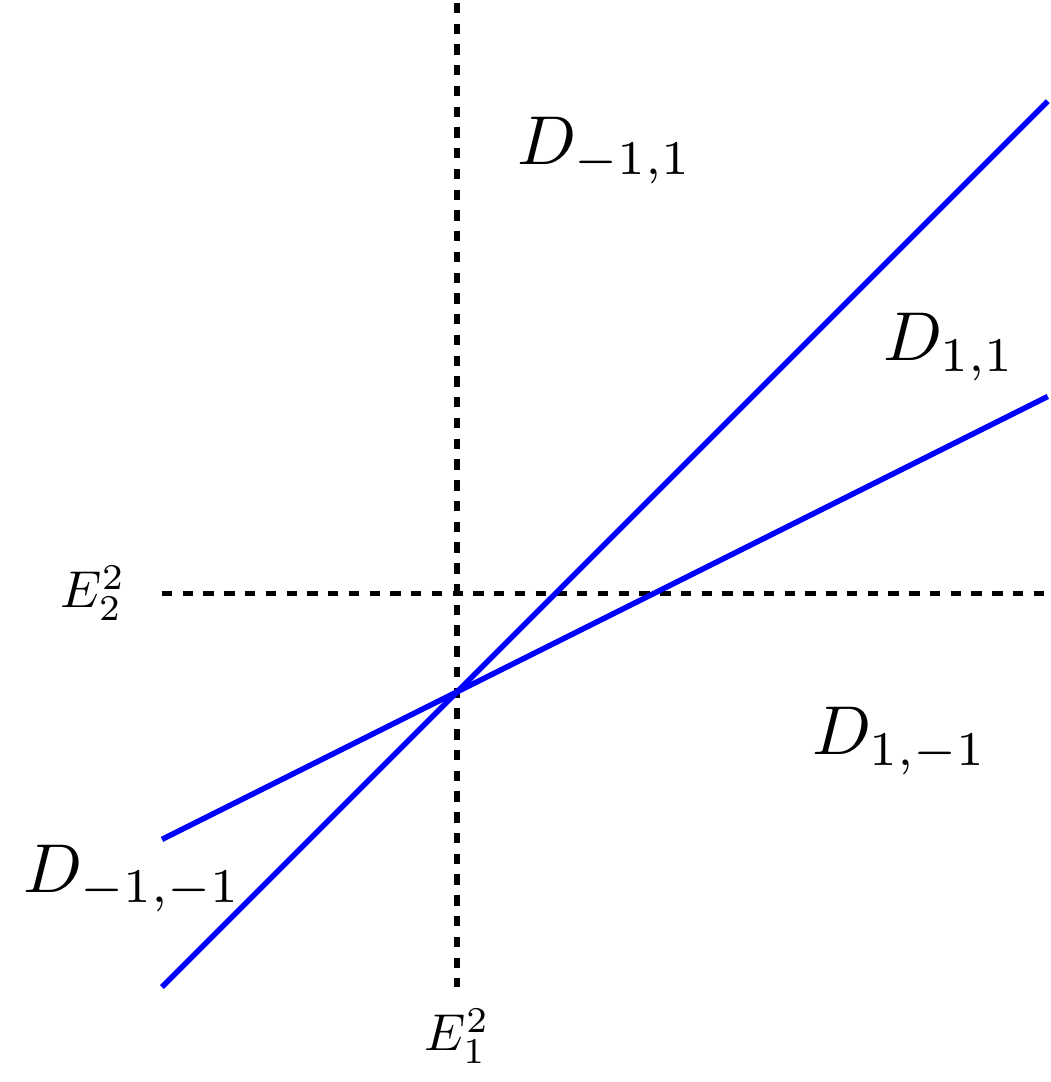} \hspace{20pt}
&  \hspace{10pt} \includegraphics[scale=0.6]{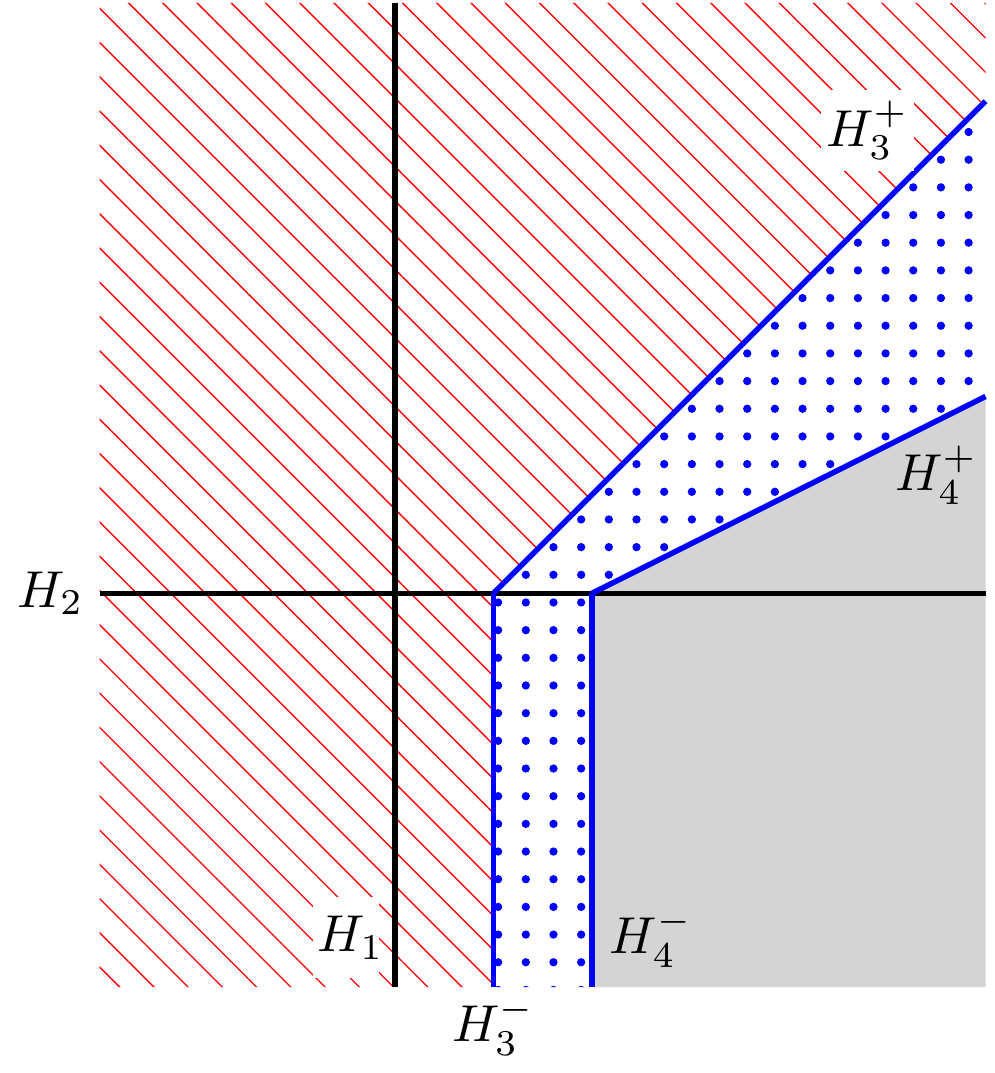} 
& \includegraphics[scale=0.5]{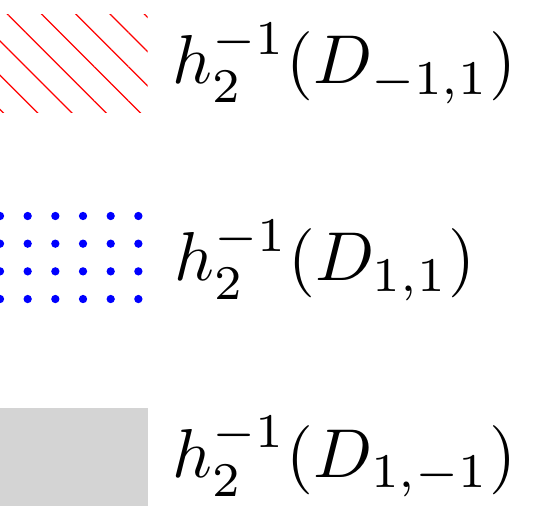}
\end{tabular}
\caption{Left. The closed polyhedra of $\Pi_2$. Right. The reciprocal images by $h_2$ of the closed polyhedra of $\Pi_2$.}\label{figure - reciprocal images of polyhedrons}
\end{figure}

Let us now check the conditions one by one.
\begin{itemize}
\item[$\Pbf.a)$] The matrices $M^2$ and $M^1$ are both full row rank, so $\Pbf.a)$ is satisfied for $k=2$ and $k=1$.
\item[$\Pbf.b)$] Let us first show the condition for $k = 2$. We have $\ring{\Omega}_3 = \ring{\Omega} = \RR^2$, so taking $x_1 = (0,1)^T \in \ring{\Omega}_3$ and $x_2 = (1,0)^T \in \ring{\Omega}_3$, we find
\[\begin{cases} M^2_{1, .} x_1 + b^2_1 = \begin{pmatrix} 1 & 0 \end{pmatrix} \begin{pmatrix} 0 \\ 1 \end{pmatrix}  = 0 \\ 
M^2_{2, .} x_2 + b^2_2 = \begin{pmatrix} 0 & 1 \end{pmatrix} \begin{pmatrix} 1 \\ 0 \end{pmatrix}  = 0.\end{cases}\] 
Let us now show the condition for $k=1$. We have $\ring{\Omega}_{k+1} = \ring{\Omega}_2 = (\RR_+^*)^2$. Let us choose $x_3 = (2,1)^T \in \ring{\Omega}_2$ and $x_4 = (4, 1)^T \in \ring{\Omega}_2$. We have
\[\begin{cases} M^1_{1, .} x_3 + b^1_1 = \begin{pmatrix} 1 & -1 \end{pmatrix} \begin{pmatrix} 2 \\ 1 \end{pmatrix} -1  = 0 \\ 
M^1_{2, .} x_4 + b^1_2 = \begin{pmatrix} -1 & 2 \end{pmatrix} \begin{pmatrix} 4 \\ 1 \end{pmatrix} +2 = 0.\end{cases}\]
This shows that $\Pbf.b)$ is satisfied for $k=2$ and $k=1$.
\item[$\Pbf.c)$] For $k=2$, let us recall from \eqref{example-values of V^2 and c^2} the values of $V^2(D)$ for all $D \in \Pi_2$. 
In the case of $V^2(D_{1,-1})$ and $V^2(D_{-1,1})$, $\Pbf.c)$ is clearly satisfied. When it comes to $D_{1,1}$, we have $V^2_{.,1}(D_{1,1}) = 0$, but $D_{1,1}$ does not intersect $E^2_1$ in $\Omega_2$. Finally, we have $V^2_{.,1}(D_{-1,-1}) = V^2_{.,2}(D_{-1,-1}) = 0$, but $D_{-1,-1} \cap \Omega_2 = \emptyset$.

We thus conclude that $\Pbf.c)$ is satisfied for $k=2$. 

The case $k=1$ is easier, $\Pi_1 = \{ \RR^2 \}$ and for all $x \in \RR^2$, we have $g_1(x) = M^0x + b^0$, so we have $V^1(\RR^2) = M^0 = \begin{pmatrix} 1 & 1 \end{pmatrix}$, and $\Pbf.c)$ is clearly satisfied.

\item[$\Pbf.d)$] Here, the case $k=1$ is trivial since $\Pi_1 = \{ \RR^2 \}$, and $h_1^{-1}( \RR^2) = \RR^2$, and thus $\partial h_1^{-1} ( \RR^2)= \emptyset$.

We thus only need to study the case $k =2$. Let us first determine the sets $h_2^{-1}(D)$, for $D \in \Pi_2$. We remind that for all $x \in \RR^2$, $h_2(x) = \sigma(x)$.

For this, let us divide $\RR^2$ in $3$ regions. Let $x = (x_1,x_2) \in \RR^2$.
\begin{itemize}
\item If $x_1 < 0$, then $h_2(x) = (0, \sigma(x_2))^T$. We thus have 
\[M^1 h_2(x) + b^1 = \begin{pmatrix} -1 ( \sigma(x_2) + 1 ) \\ 2 ( \sigma(x_2)+1) \end{pmatrix}.\]
Since $\sigma(x_2) \geq 0$, we see that $h_2(x) \in D_{-1,1}$.
\item If $x_1 \geq 0$ and $x_2 < 0$, then $h_2(x) = (x_1,0)^T$. We thus have 
\[M^1 h_2(x) + b^1 = \begin{pmatrix} x_1 - 1 \\ - x_1 + 2 \end{pmatrix}.\]
There are $3$ possibilities: if $x_1 \leq 1$, then $h_2(x) \in D_{-1,1}$, if $1 \leq x_1 \leq 2$, $h_2(x) \in D_{1,1}$, and if $2 \leq x_1$, $h_2(x) \in D_{1, -1}$.
\item If $x_1, x_2 \geq 0$, then $h_2(x) = x$ and for all $D_{\epsilon_1, \epsilon_2} \in \Pi_2$, $h_2(x) \in D_{\epsilon_1, \epsilon_2} \Longleftrightarrow x \in D_{\epsilon_1, \epsilon_2}$.
There are $3$ possibilities, $x \in D_{-1,1}, x \in D_{1,-1}$ and $x \in D_{1,1}$ since $D_{-1,-1} \cap (\RR_+)^2 = \emptyset$.
\end{itemize}
Summarizing, we find (see also Figure \ref{figure - reciprocal images of polyhedrons}):

\[h_2^{-1}(D_{-1,1}) = \RR_- \times \RR \quad \cup \quad [0,1] \times \RR_- \quad \cup \quad (\RR_+)^2 \cap D_{-1,1}\]
\[h_2^{-1}(D_{1,1}) = [1,2] \times \RR_- \quad \cup \quad (\RR_+)^2 \cap D_{1,1}\]
\[h_2^{-1}(D_{1,-1}) = [2, + \infty[ \times \RR_- \quad \cup \quad (\RR_+)^2 \cap D_{1,-1}\]
\[h_2^{-1}(D_{-1,-1}) = \emptyset.\]

To express the boundaries of these regions, we define the following pieces of hyperplanes:

\[H_3^+ = \{ x = (x_1,x_2) \in \RR^2, \quad  x_1 \geq 0, x_2 \geq 0, M^1_{1,.} x + b^1_1 = 0\}\]
\[H_4^+ = \{ x = (x_1,x_2) \in \RR^2, \quad  x_1 \geq 0, x_2 \geq 0, M^1_{2,.} x + b^1_2 = 0\}\]
\[H_3^- = \{ x = (x_1,x_2) \in \RR^2, \quad  x_1 = 1, x_2 \leq 0\}\]
\[H_4^- = \{ x = (x_1,x_2) \in \RR^2, \quad  x_1 = 2 , x_2 \leq 0\}.\]

We have 

\[\partial h_2^{-1}(D_{-1,1}) = H_3^- \cup H_3^+\]
\[\partial h_2^{-1}(D_{1,1}) = H_3^- \cup H_3^+ \cup H_4^- \cup H_4^+\]
\[\partial h_2^{-1}(D_{1,-1}) = H_4^- \cup H_4^+\]
\[\partial h_2^{-1}(D_{-1,-1}) = \emptyset,\]

and thus,
\begin{equation}\label{border of the reciprocal images of polyhedrons} \bigcup_{D \in \Pi_2} \partial h_2^{-1}(D) =  H_3^- \cup H_3^+ \cup  H_4^- \cup H_4^+.\end{equation}

Let us check the condition $\Pbf.d)$ for $k=2$. Since $\Omega_3 = \RR^2$, here $\ring{\Omega}_3 \cap H = H$. The condition is thus satisfied if and only if $\bigcup_{D \in \Pi_2} \partial h_2^{-1}(D)$ does not contain any full hyperplane $H$, and \eqref{border of the reciprocal images of polyhedrons} shows that it is the case. The condition $\Pbf.d)$ is satisfied.
\end{itemize}

Let us now discuss the conditions given in \cite{phuong2020functional, pmlr-v119-rolnick20a} for this example. Let us first define the following hyperplanes:
\[H_1 = \{ x \in \RR^2, M^2_{1,.} x + b^2_1 = 0 \}\]
\[H_2 = \{ x \in \RR^2, M^2_{2,.} x + b^2_2 = 0 \}.\]

To discuss the conditions in \cite{phuong2020functional}, we will refer to their concepts of fold-set, of piece-wise linear surface, of canonical representation and dependency graph of a piece-wise linear surface, as well as to their Lemma 4. We also use their notations $\square_1 S$ and $\square_2 S$. Let us now consider the set $S$ as the fold-set of the function $f_{\Mbf, \bbf}$ implemented by the network. Here, it corresponds to the points $x$ satisfying one of the following equations
\[\begin{cases}
M^2_{1,.} x + b^2_1 = 0 \\
M^2_{2,.} x + b^2_2 = 0 \\
M^1_{1,.} h_2(x) + b^1_1 = 0 \\
M^1_{2,.} h_2(x) + b^1_2 = 0 .
\end{cases}\]
In other words, $S = H_1 \cup H_2  \cup  H_3^- \cup H_3^+  \cup  H_4^- \cup H_4^+ $. The canonical representation of $S$ is the following
\[S = \left( H_1 \cup H_2 \right)  \cup \left( H_3^- \cup H_3^+  \cup  H_4^- \cup H_4^+ \right), \]
where $\square_1 S = \left( H_1 \cup H_2 \right)$ and $\square_2 S = S$. Further, it can be checked that the dependency graph of $S$ only contains the edges:
$H_2 \rightarrow H_3^-$, $H_2 \rightarrow H_3^+$, $H_2 \rightarrow H_4^-$ and $H_2 \rightarrow H_4^+$. 

The identifiable networks considered in \cite{phuong2020functional} must satisfy the conditions of Lemma 4 in \cite{phuong2020functional}. In particular, the dependency graph of $S$ must contain at least $2$ directed paths of length $1$ with distinct starting vertices, which is not the case here since all the paths of length $1$ start from $H_2$. Hence, this network does not fall under the conditions of \cite{phuong2020functional}.

Now if we use the concepts and notations of \cite{pmlr-v119-rolnick20a}, let us denote by $z_1$ the first neuron of the first hidden layer, whose associated parameters are $M^2_{1,.}$ and $b^2_1$. Following the definition in \cite{pmlr-v119-rolnick20a}, the boundary associated to $z_1$ is $B_{z_1} = H_1$. Let us denote by $z_3$ the first neuron of the second hidden layer, whose associated parameters are $M^1_{1,.}$ and $b^1_1$. Its boundary is $B_{z_3} = H_3^- \cup H_3^+$. Since $z_1$ and $z_3$ belong to two consecutive layers and are thus linked by an edge of the network, the conditions in \cite{pmlr-v119-rolnick20a} (see Theorem 2) require that $B_{z_1}$ and $B_{z_3}$ intersect. It is however clear that they do not (see Figure \ref{figure - reciprocal images of polyhedrons}).

\section{Sketch of proof of Theorem \ref{main theorem-main}}\label{sketch-sec}

Our main result, Theorem \ref{main theorem-main}, is proven in details in Appendix \ref{Proof of the main theorem-apdx}, and we give a sketch of the proof in this section. It is proven by induction. We are given two parameterizations $\mathbf{(M,b)}$ and $\mathbf{(\tilde{M}, \tilde{b})}$, two lists $\boldsymbol{\Pi}$ and $\boldsymbol{\tilde{\Pi}}$ that are admissible with respect to $\mathbf{(M,b)}$ and $\mathbf{(\tilde{M}, \tilde{b})}$ respectively, and a domain $\Omega$ that satisfy the hypotheses of Theorem \ref{main theorem-main} and we want to show that the two parameterizations are equivalent. For this, we identify the layers one after the other. To facilitate identification at a layer level, we begin with a normalisation step.

\subsection{Normalisation step} 

Two equivalent parameterizations do not necessarily have equal weights on their layers. Indeed, the neuron permutations but more importantly the rescalings can change the structure of the intermediate layers. We are going to assume the following normalisation property: for all $k \in \llbracket  1, K-1 \rrbracket$, for all $i \in \llbracket 1 , n_k \rrbracket$, we have
\begin{equation}\begin{aligned}
\|M^k_{i,.}\|=1 ;\\
\|\tilde{M}^k_{i,.} \|=1.
\end{aligned}
\end{equation} 

Indeed, we show in the appendix that for a parameterization satisfying the conditions $\mathbf{P}$, there always exists an equivalent parameterization that is normalised and that satisfies the conditions $\mathbf{P}$ (see Propositions \ref{equivalent normalized network} and \ref{conditions P are stable by equivalence}). We can thus replace each parameterization $\mathbf{(M,b)}$ and $\mathbf{(\tilde{M}, \tilde{b})}$ by an equivalent normalised parameterization. If we are able to show the normalised parameterizations are equivalent, then the original parameterizations are equivalent too.

\subsection{Induction}
 
The induction proof relies on Lemma \ref{Fundamental lemma-main} below. Let $K$ be the number of layers of the network, and suppose the theorem is true for the networks with $K-1$ layers. As explained in section \ref{Section conditions-main}, to identify the parameters $M^{K-1}$ and $b^{K-1}$, we separate the function implemented by the first layer of the network and the function implemented by the rest of the layers. For each network:
\[g_{K} = g_{K-1} \circ h_{K-1},\]
\[\tilde{g}_{K} = \tilde{g}_{K-1} \circ \tilde{h}_{K-1}.\]

We know that $g_{K-1}$ and $\tilde{g}_{K-1}$ are continuous piecewise linear, and this will allow us to apply Lemma \ref{Fundamental lemma-main}. Before stating it, we introduce a set of conditions, called $\mathbf{C}$, that need to be satisfied in order to apply it. These conditions come immediately from $\mathbf{P}$, and one can easily check that $(g_{K-1},M^{K-1},b^{K-1},\Omega_{K}, \Pi_{K-1})$ and $(\tilde{g}_{K-1}, \tilde{M}^{K-1}, \tilde{b}^{K-1}, \Omega_{K}, \tilde{\Pi}_{K-1})$ satisfy $\mathbf{C}$, as a direct consequence of the conditions $\mathbf{P}$ being satisfied by $\mathbf{(M,b, \Omega, \boldsymbol{\Pi})}$ and $\mathbf{(\tilde{M}, \tilde{b}, \Omega, \boldsymbol{\tilde{\Pi}})}$.

\begin{df}
Let $l,m,n$ be integers, $M \in \mathbb{R}^{m \times l}$, $b \in \mathbb{R}^m$, $\Omega \subset \mathbb{R}^l$ be an open domain, let $g : \mathbb{R}^m \rightarrow \mathbb{R}^n$ a continuous piecewise linear function, and let $\Pi$ be an admissible set of polyhedra with respect to $g$.

Let $D \in \Pi$.  The function $g$ coincides with a linear function on $D$. Since the interior of $D$ is nonempty, we define $V(D) \in \mathbb{R}^{n \times m}$ and $c(D) \in \mathbb{R}^n$ as the unique couple satisfying, for all $x \in D$:
\[g(x) = V(D)x + c(D).\]

We denote $E_i = \{ x \in \mathbb{R}^m, x_i = 0 \}$.

 We say that $(g,M,b,\Omega, \Pi)$ satisfies the conditions $\mathbf{C}$ iif
\begin{itemize}
\item[$\mathbf{C}.a)$] $M$ is full row rank;
\item[$\mathbf{C}.b)$] for all $i \in \llbracket 1 , m \rrbracket$, there exists $x \in \ring{\Omega}$ such that
\[M_{i,.}x + b_i = 0,\]
or equivalently, if we denote by $h^{lin}$ the function $x \mapsto M x + b$, then \[E_i \cap h^{lin}(\ring{\Omega}) \neq \emptyset;\]
\item[$\mathbf{C}.c)$] for all $D \in \Pi$, for all $i \in \llbracket 1 , m \rrbracket$, if $E_i \cap D \cap h(\Omega) \neq \emptyset$ then $V_{.,i}(D) \neq 0$;
\item[$\mathbf{C}.d)$] for any affine hyperplane $H \subset \mathbb{R}^{l}$,
\[H \cap \ring{\Omega} \ \not\subset \bigcup_{D \in \Pi} \partial h^{-1}(D).\]
\end{itemize}
\end{df}

We can now state the lemma.

\begin{lem}\label{Fundamental lemma-main} Let $l,m,n \in \mathbb{N}^*$. Suppose $g,\tilde{g} : \mathbb{R}^m \rightarrow \mathbb{R}^n$ are continuous piecewise linear functions, $\Omega \subset \mathbb{R}^l$ is a subset and let $M, \tilde{M} \in \mathbb{R}^{m \times l}$, $b, \tilde{b} \in \mathbb{R}^{m}$. Denote $h: x \mapsto \sigma(Mx + b)$ and $\tilde{h}: x \mapsto \sigma(\tilde{M}x + \tilde{b})$. Assume $\Pi$ and $\tilde{\Pi}$ are two sets of polyhedra admissible with respect to $g$ and $\tilde{g}$.

Suppose  $(g,M,b,\Omega, \Pi)$ and $(\tilde{g}, \tilde{M}, \tilde{b}, \Omega, \tilde{\Pi})$ satisfy the conditions $\mathbf{C}$, and for all $i \in \llbracket 1 , m \rrbracket$, $\|M_{i,.} \| = \| \tilde{M}_{i,.} \|= 1$.

Suppose for all $x \in \Omega$:
\[g \circ h (x) = \tilde{g} \circ \tilde{h}(x).\]

Then, there exists a permutation $\varphi \in \mathfrak{S}_m$, such that:
\begin{itemize}
\item $\tilde{M} = P_{\varphi}M $;
\item $\tilde{b} = P_{\varphi}b$;
\item $g$ and $y \mapsto \tilde{g} (P_{\varphi}y)$ coincide on ${h}(\Omega)$.
\end{itemize}
\end{lem}

Lemma \ref{Fundamental lemma-main} is restated in Appendix \ref{Appendix-B} as Lemma \ref{Fundamental lemma} and proven in Appendix \ref{proof of the lemma}.

Applying this lemma to $(g_{K-1},M^{K-1},b^{K-1},\Omega_{K}, \Pi_{K-1})$ and $(\tilde{g}_{K-1}, \tilde{M}^{K-1}, \tilde{b}^{K-1}, \Omega_{K}, \tilde{\Pi}_{K-1})$, we conclude that there exists a permutation $\varphi_{K-1}$ such that
\begin{equation}\begin{cases}\tilde{M}^{K-1}= P_{\varphi_{K-1}}M^{K-1}\\
\tilde{b}^{K-1} = P_{\varphi_{K-1}} b^{K-1},\end{cases} \end{equation}
and that $g_{K-1}$ and $y \mapsto \tilde{g}(P_{\varphi_{K-1}}y)$ coincide on $h_{K-1}(\Omega)$.

The functions $g_{K-1}$ and $y \mapsto \tilde{g}(P_{\varphi_{K-1}}y)$ are the functions implemented by the networks $\mathbf{(M,b)}$ and $\mathbf{(\tilde{M}, \tilde{b})}$ once we have removed the first layer, with a permutation of the input for the second one. Since they coincide on $\Omega_{K-1} = h_{K-1}(\Omega)$ and they satisfy the conditions $\mathbf{P}$, we can apply the induction hypothesis to conclude the proof of Theorem \ref{main theorem-main}. The complete proof is detailed in the appendices, as discussed above.

\section{Conclusion}

We established a set of conditions $\mathbf{P}$ under which the function implemented by a deep feedforward ReLU neural network on a subset $\Omega$ of the input space uniquely characterizes its parameters, up to permutation and positive rescaling. This contributes to the understanding of identifiability and stable recovery for deep ReLU networks, which is still largely unexplored. The conditions under which our result holds differ from the conditions of the results established in \cite{phuong2020functional} and \cite{pmlr-v119-rolnick20a}, which allows us to cover new situations.
To be satisfied the conditions $\mathbf{P}$ need $\Omega$ to have nonempty interior, which prevents it from being a sample set.  The authors of \cite{stock2022embedding, bonalocal} are able to give a result with a finite set $\Omega$, but for local identifiability only. Obtaining the best of both worlds, that is establishing a global identifiability result for deep ReLU networks with a finite set $\Omega$, would be a major step forward.

\section*{Acknowledgements}

Our work has benefited from the AI Interdisciplinary Institute ANITI. ANITI is funded by the French ``Investing for
the Future – PIA3” program under the Grant agreement n°ANR-19-PI3A-0004.

\bigskip

\begin{appendix}

\newpage

In the appendices, we restate all the notations, definitions and results of the main text, for clarity of reading. The appendices are then organized as follows. In Appendix \ref{Appendix-A}, we give the complete definitions and basic properties necessary to state and prove the main theorem. 
In Appendix \ref{Appendix-B}, we state the main result, Theorem \ref{main theorem} (Theorem \ref{main theorem-main} in the main text), and we prove it. 
Finally, we prove the fundamental lemma used in the proof of the main theorem, Lemma \ref{Fundamental lemma} (Lemma \ref{Fundamental lemma-main} in the main text), in Appendix \ref{proof of the lemma}.

\section{Definitions, notations and preliminary results}\label{Appendix-A}

Appendix \ref{Appendix-A} is structured as follows: after giving some notations in Section \ref{Basic notations-sec}, we recall the definition of a continuous piecewise linear function and some corresponding basic properties in Section \ref{CPL dunctions-sec} and we give our formalization of deep ReLU networks as well as some well-known properties in Section \ref{Neural networks-sec}.

\subsection{Basic notations and definitions}\label{Basic notations-sec}

We denote by 
\[\fonction{\sigma}{\mathbb{R}}{\mathbb{R}}{t}{\max(t,0)}\]
 the ReLU activation function. If $x = (x_1, \dots , x_m)^T \in \mathbb{R}^m$ is a vector, we denote $\sigma(x) = (\sigma(x_1), \dots , \sigma(x_m))^T$.

If $A \subset \mathbb{R}^m$, we denote by $\ring{A}$ the interior of $A$ and $\overline{A}$ the closure of $A$ with respect to the standard topology of $\mathbb{R}^m$. We denote by $\partial  A = \overline{A} \backslash \ring{A}$ the topological boundary of $A$.

For $m,n \in \mathbb{N}^*$, we denote by $\mathbb{R}^n$ the vector space of $n$-dimensional real vectors and $\mathbb{R}^{m \times n}$ the vector space of real matrices with $m$ lines and $n$ columns. On the space of vectors, we use the norm $\| x \| = \sqrt{\sum_{i=1}^n x_i^2}$. For $x \in \mathbb{R}^n$ and $r>0$, we denote $B(x,r) = \{ y \in \mathbb{R}^n , \| y - x \| < r \}$.

For any vector $x \in \mathbb{R}^n$ whose coefficients $x_i$ are all different from zero, we denote by $x^{-1}$ or $\frac{1}{x}$ the vector $\left(\frac{1}{x_1}, \frac{1}{x_2}, \dots, \frac{1}{x_n}\right)^T$.

For any matrix $M \in \mathbb{R}^{m\times n}$, for all $i \in \llbracket 1 , m \rrbracket$, we denote by $M_{i, .}$ the $i^{\text{th}}$ line of $M$. The vector $M_{i,.}$ is a line vector whose $j^{\text{th}}$ component is $M_{i,j}$. Similarly, for $j \in \llbracket 1 , n \rrbracket$, we denote by $M_{.,j}$ the $j^{\text{th}}$ column of $M$, which is the column vector whose $i^{\text{th}}$ component is $M_{i,j}$. For any matrix $M \in \mathbb{R}^{m \times n}$, we denote by $M^T \in \mathbb{R}^{n\times m}$ the transpose matrix of $M$. 

To avoid any confusion, we will denote by $(M^T)_{i,.}$ the $i^{\text{th}}$ line of the matrix $M^T$ and by $M_{i,.}^{ \hspace{5pt} T}$ the transpose of the line vector $M_{i,.}$, which is a column vector. Similarly, we will denote by $(M^T)_{.,j}$ the $j^{\text{th}}$ column of $M^T$ and $M_{.,j}^{ \hspace{5pt} T}$ the transpose of the column vector $M_{.,j}$.

For $n \in \mathbb{N}^*$, we denote by $\Id_n$ the $n \times n$ identity matrix and by $\One_n$ the vector \mbox{$(1,1, \dots, 1)^T \in \mathbb{R}^n$}.

If $\lambda \in \mathbb{R}^n$ is a vector of size $n$, for some $n \in \mathbb{N}^*$, we denote by $\Diag(\lambda)$ the $n \times n$ matrix defined by:
\[\Diag(\lambda)_{i,j} = \begin{cases} \lambda_i & \text{if } i=j \\ 0 & \text{otherwise.}\end{cases}\] 

For any integer $m \in \mathbb{N}^*$, we denote by $\mathfrak{S}_m $ the set of all permutations of $\llbracket 1 , m \rrbracket$. We denote by $id_{\llbracket 1, m \rrbracket}$ and $id_{\mathbb{R}^m}$ the identity functions on $\llbracket 1 , m \rrbracket$ and $\mathbb{R}^m$ respectively.

For any permutation $\varphi \in \mathfrak{S}_m$, we denote by $P_\varphi$ the $m \times m$ permutation matrix associated to $\varphi$:
\begin{equation}\label{permutation matrix-0}\forall i, j \in \llbracket 1 , m \rrbracket, \quad (P_\varphi)_{i,j} = \begin{cases}  1 & \text{if } \varphi(j) = i\\ 0 & \text{otherwise.} \end{cases}\end {equation}
For all $x \in \mathbb{R}^m$, we have:
\begin{equation}\label{permutation matrix}
(P_{\varphi}x)_i = x_{\varphi^{-1}(i)}.
\end{equation}
Using \eqref{permutation matrix} we see that $P_{\varphi^{-1}} P_{\varphi}x = x$, which shows, since $P_{\varphi}$ is orthogonal, that we have
\begin{equation}\label{Inverse of a permutation matrix}
P_{\varphi}^{-1} = P_{\varphi^{-1}} = P_{\varphi}^T.
\end{equation}

Let $l,m,n \in \mathbb{N}^*$. For any matrix $M \in \mathbb{R}^{m \times l}$ and any function $f : \mathbb{R}^m \rightarrow \mathbb{R}^n$, we denote with a slight abuse of notation $f \circ M$ the function $x \mapsto f(Mx)$.

If $X$ and $Y$ are two sets and $h : X \rightarrow Y$ is a function, for a subset $A \subset Y$, we denote by $h^{-1}(A)$ the following set:
\[\{x \in X, h(x) \in A \}.\]
Note that this does not require the function $h$ to be injective.

\subsection{Continuous piecewise linear functions}\label{CPL dunctions-sec}

We now introduce a few definitions and properties around the notion of continuous piecewise linear function.

\begin{df}
Let $m \in \mathbb{N^*}$. A subset $D \subset \mathbb{R}^m$ is a closed polyhedron iif there exist $q \in \mathbb{N}^*$, $a_1, \dots, a_q \in \mathbb{R}^m$ and $b_1, \dots b_q \in \mathbb{R}$ such that for all $x \in \mathbb{R}^m$, 
\[x \in D \quad \Longleftrightarrow \quad \begin{cases}
a_1^T x + b_1 \leq 0 \\
\vdots \\
a_q^T x + b_q \leq 0.
\end{cases}\]
\end{df}

\begin{rems}
\begin{itemize}
\item A closed polyhedron is convex as an intersection of convex sets.
\item Since we can fuse the inequation systems of several closed polyhedrons into one system, we see that an intersection of closed polyhedrons is a closed polyhedron.
\item For $q=1$ and $a_1 = 0$, taking $b_1 > 0$ and $b_1 \leq 0$ respectively we can show that $\emptyset$ and $\mathbb{R}^m$ are both closed polyhedra.
\end{itemize}
\end{rems}

\begin{prop}\label{reciprocal image of a polyhedron}
Let $m,l \in \mathbb{N}^*$. If $h : \mathbb{R}^l \rightarrow \mathbb{R}^m$ is linear and $C$ is a closed polyhedron of $\mathbb{R}^m$, then $h^{-1}(C)$ is a closed polyhedron of $\mathbb{R}^l$.
\end{prop}

\begin{proof}
The function $h$ is linear so there exist $M \in \mathbb{R}^{m \times l}$ and $b \in \mathbb{R}^m$ such that for all $x \in \mathbb{R}^l$, 
\[h(x) = Mx + b.\]
The set $C$ is a closed polyhedron so there exist $a_1, \dots, a_q \in \mathbb{R}^m$ and $b_1, \dots b_q \in \mathbb{R}$ such that $y \in C$ if and only if
\[\begin{cases}
a_1^T y + b_1 \leq 0 \\
\vdots \\
a_q^T y + b_q \leq 0.
\end{cases}\]

For all $x \in \mathbb{R}^l$, 
\begin{equation*}
\begin{aligned}
x \in h^{-1}(C) \quad & \Longleftrightarrow \quad h(x) \in C \\
& \Longleftrightarrow \quad \begin{cases}
a_1^T (Mx + b) + b_1 \leq 0 \\
\vdots \\
a_q^T (Mx + b) + b_q \leq 0
\end{cases} \\
& \Longleftrightarrow \quad \begin{cases}
(a_1^T M)x + (a_1^Tb + b_1) \leq 0 \\
\vdots \\
(a_q^T M)x + (a_q^Tb + b_q) \leq 0.
\end{cases}
\end{aligned}
\end{equation*}

This shows that $h^{-1}(C)$ is a closed polyhedron.
\end{proof}

\begin{df}
We say that a function $g: \mathbb{R}^m \rightarrow \mathbb{R}^n$ is continuous piecewise linear if there exists a finite set of closed polyhedra whose union is $\mathbb{R}^m$ and such that $g$ is linear over each polyhedron.
\end{df}

\begin{ex}
Since $\mathbb{R}^m$ is a closed polyhedron, we see in particular that an affine function $x \mapsto Ax + b$, with $A \in \mathbb{R}^{n \times m}$ and $b \in \mathbb{R}^n$, is continuous piecewise linear from $\mathbb{R}^m$ to $\mathbb{R}^n$.
\end{ex}

\begin{exnum}
The vectorial ReLU function $\sigma : \mathbb{R}^m \rightarrow \mathbb{R}^m$ is continuous piecewise linear. Indeed, each of the $2^m$ closed orthants is a closed polyhedron, defined by a system of the form 
\[\begin{cases}\epsilon_1 x_1 \geq 0 \\
 \vdots \\
  \epsilon_m x_m \geq 0, \end{cases}\]
  with $\epsilon_i \in \{-1, 1 \}$, and over such an orthant, the ReLU coincides with the affine function
\[(x_1, \dots, x_m ) \mapsto \left(\frac{1+\epsilon_1}{2} x_1, \dots,\frac{1+\epsilon_m}{2} x_m\right) .\]

\end{exnum}

In this definition the continuity is not obvious. We show it in the following proposition.

\begin{prop}\label{continuity of CPL functions}
A continuous piecewise linear function is continuous.
\end{prop}

\begin{proof}
Let $g: \mathbb{R}^m \rightarrow \mathbb{R}^n$ be a continuous piecewise linear function. 
There exists a finite family of closed polyhedra $C_1, \dots, C_r$ such that $\bigcup_{i=1}^r C_i = \mathbb{R}^m$ and $g$ is linear on each closed polyhedron $C_i$.

Let $x \in \mathbb{R}^m$.  Let $\epsilon > 0$.

Let us denote $I = \{ i \in \llbracket 1 , n \rrbracket, \ x \in C_i \} $. Since the polyhedrons are closed, there exists $r_0 > 0 $ such that for all $i \notin I, B(x,r_0) \cap C_i = \emptyset $. We thus have 
\[B(x,r_0) = \bigcup_{i=1}^m (B(x,r_0) \cap C_i) = \bigcup_{i \in I} \left( B(x,r_0) \cap C_i \right).\]

For all $i \in I$, $g$ is linear -therefore continuous- on $C_i$ so there exists $r_i > 0$, such that 
\[y \in C_i \cap B(x,r_i) \ \Rightarrow \ \| g(y) - g(x) \| \leq \epsilon.\]

Let $r = \min(r_0, \min_{i \in I} (r_i) )$. For all $y \in B(x,r)$ there exists $i \in I$ such that $y \in C_i$, and since $r \leq r_i$, we have
\[\| g(y) - g(x) \| \leq \epsilon.\]

Summarizing, for any $x \in \mathbb{R}^n$ and for any $\epsilon > 0$, there exists $r> 0$ such that 
\[y \in B(x,r) \ \Rightarrow \ \| g(y) - g(x) \| \leq \epsilon.\] This shows $g$ is continuous.
\end{proof}

\begin{prop}\label{composition CPL functions}
If $h : \mathbb{R}^l \rightarrow \mathbb{R}^m$ and $g : \mathbb{R}^m \rightarrow \mathbb{R}^n$ are two continuous piecewise linear functions, then $g \circ h$ is continuous piecewise linear.
\end{prop}

\begin{proof}
By definition there exist a family $C_1, \dots, C_r$ of closed polyhedra of $\mathbb{R}^l$ such that $\bigcup_{i=1}^r C_i = \mathbb{R}^l$ and $h$ is linear on each $C_i$ and a family $D_1, \dots, D_s$ of closed polyhedra of $\mathbb{R}^m$ such that $\bigcup_{i=1}^s D_i = \mathbb{R}^m$ and $g$ is linear on each $D_i$. Let $i \in \llbracket 1, r \rrbracket$ and $j \in \llbracket 1, s \rrbracket$. The function $h$ coincides with a linear map $\tilde{h}: \mathbb{R}^l \rightarrow \mathbb{R}^m $ on $C_i$ and the inverse image of a closed polyhedron by a linear map is a closed polyhedron (Proposition \ref{reciprocal image of a polyhedron}) so $\tilde{h}^{-1}(D_j)$ is a closed polyhedron. Thus $h^{-1}(D_j)\cap C_i = \tilde{h}^{-1}(D_j)\cap C_i$ is a closed polyhedron as an intersection of closed polyhedra. The function $h $ is linear on $C_i$ and $g$ is linear on $D_j$ so $g \circ h$ is linear on $h^{-1}(D_j) \cap C_i$. We have a family of closed polyhedra,
\[\left( h^{-1}(D_j) \cap C_i\right)_{\substack{i \in \llbracket 1 , r \rrbracket \\ j \in \llbracket 1, s \rrbracket}},\] each of which $g \circ h$ is linear over. Given that
\[\bigcup_{i = 1}^r \bigcup_{j=1}^s h^{-1}(D_j) \cap C_i = \bigcup_{i = 1}^r  C_i = \mathbb{R}^l,\]
we can conclude that $g \circ h$ is continuous piecewise linear.
\end{proof}

\begin{df}\label{def admissible sets of polyhedra}
Let $g: \mathbb{R}^m \rightarrow \mathbb{R}^n$ be a continuous piecewise linear function. Let $\Pi$ be a set of closed polyhedra of $\mathbb{R}^m$. We say that $\Pi$ is admissible with respect to the function $g$ if and only if:
\begin{itemize}
\item $\bigcup_{D \in \Pi} D = \mathbb{R}^m$,
\item for all $D \in \Pi$, $g$ is linear on $D$,
\item for all $D \in \Pi$, $\ring{D} \neq \emptyset$.
\end{itemize}
\end{df}

\begin{prop}\label{proposition CPL functions}
For all $g: \mathbb{R}^m \rightarrow \mathbb{R}^n$ continuous piecewise linear, there exists a set of closed polyhedra $\Pi$ admissible with respect to $g$.
\end{prop}

\begin{proof}
Let $g : \mathbb{R}^m \rightarrow \mathbb{R}^n$ be a continuous piecewise linear function. By definition there exists a finite set of closed polyhedra $D_1, \dots, D_s$ such that $\bigcup_{i=1}^s D_i = \mathbb{R}^m$ and $g$ is linear on each $D_i$. 

Let $I = \{ i \in \llbracket 1, s \rrbracket, \ring{D_i} \neq \emptyset\}$. Let us show that $\bigcup_{i \in I} D_i = \mathbb{R}^m$.

We first show that if a polyhedron $D_i$ has empty interior, then it is contained in an affine hyperplane.  Indeed, if it is not contained in an affine hyperplane, then there exist $m+1$ affinely independent points $x_1, \dots, x_{m+1} \in D_i$. Since a closed polyhedron is convex, the convex hull of the points $\Conv(x_1, \dots, x_{m+1})$, which is a $m$-simplex, is contained in $D_i$, and thus $D_i$ has nonempty interior.

Let $x \in \mathbb{R}^m$. For all $i \notin I$, $D_i$ is contained in an affine hyperplane, and a finite union of affine hyperplanes does not contain any nontrivial ball. As a consequence, for all $n \in \mathbb{N}$, the ball $B(x,\frac{1}{n})$ is not contained in $\bigcup_{i \notin I} D_i$ and thus there exists $i_n \in I$ such that $D_{i_n} \cap B(x,\frac{1}{n}) \neq \emptyset$. Since $I$ is finite, there exists $i \in I$ such that $i_n = i$ for infinitely many $n$, and thus $x \in \overline{D_i}$. 

We have shown that for all $x \in \mathbb{R}^m$ there exists $i \in I$ such that $x \in \overline{D_i} = D_i$, which means that
\[\bigcup_{i \in I} D_i = \mathbb{R}^m.\]

Hence, the set $\Pi : = \{ D_i, i \in I \}$ is admissible with respect to $g$.
\end{proof}

\begin{prop} \label{boundary included in a union of hyperplanes}
Let $h : \mathbb{R}^l \rightarrow \mathbb{R}^m$ be a continuous piecewise linear function and let $\mathcal{P}$ be a finite set of closed polyhedra of $\mathbb{R}^m$. Then 
\begin{itemize}
    \item for all $D \in \mathcal{P}$, $h^{-1}(D)$ is a finite union of closed polyhedra;
    \item $\bigcup_{D \in \mathcal{P}} \partial h^{-1}(D)$ is contained in a finite union of hyperplanes $\bigcup_{k=1}^s A_k$.
\end{itemize}
\end{prop}

\begin{proof}
Consider $\Pi$ an admissible set of closed polyhedra with respect to $h$. Let $D \in \mathcal{P}$. Since $\bigcup_{C \in \Pi} C = \mathbb{R}^l$, we can write
\[h^{-1}(D) = h^{-1}(D) \cap \left(\bigcup_{C \in \Pi} C\right) = \bigcup_{C \in \Pi} \left(h^{-1}(D) \cap  C \right).\]
For all $C \in \Pi$, $h$ is linear over $C$, so $h^{-1}(D) \cap C$ is a polyhedron (see Proposition \ref{reciprocal image of a polyhedron}). This shows the first point of the proposition.

Since $h^{-1}(D) \cap C$ is a polyhedron, $\partial \left(h^{-1}(D) \cap C \right)$ is contained in a finite union of hyperplanes. In topology, we have
\[\partial \left[ \bigcup_{C \in \Pi} \left(h^{-1}(D) \cap  C \right)\right] \ \subset \ \bigcup_{C \in \Pi} \partial  \left( h^{-1}(D) \cap  C \right),\]
which shows that $\partial   \left[ \bigcup_{C \in \Pi} \left(h^{-1}(D) \cap  C \right)\right]$ i.e. $\partial h^{-1}(D)$ is contained in a finite union of hyperplanes too. 
This is true for any $D \in \mathcal{P}$, and since $\mathcal{P}$ is finite, this is also true of the union $\bigcup_{D \in \mathcal{P}} \partial h^{-1}(D)$.
\end{proof}

\subsection{Neural networks}\label{Neural networks-sec}

We consider fully connected feedforward neural networks, with ReLU activation function. We index the layers in reverse order, from $K$ to $0$, for some $K \geq 2$. The input layer is the layer $K$, the output layer is the layer $0$, and between them are $K-1$ \emph{hidden} layers. For $k \in \llbracket 0, K \rrbracket$, we denote by $n_k \in \mathbb{N}$ the number of neurons of the layer $k$. This means the information contained at the layer $k$ is a $n_k$-dimensional vector. 

Let $k \in \llbracket 0, K-1 \rrbracket$. We denote the weights between the layer $k+1$ and the layer $k$ with a matrix $M^{k} \in \mathbb{R}^{n_k \times n_{k+1}}$, and we consider a bias $b^{k} \in \mathbb{R}^{n_k}$ in the layer $k$. If $k \neq 0$, we add a ReLU activation function. If $x \in \mathbb{R}^{n_{k+1}}$ is the information contained at the layer $k+1$, the layer $k$ contains:
\[\begin{cases}\sigma(M^{k} x + b^{k}) & \text{if } k \neq 0 \\M^{0} x + b^{0} & \text{if } k=0.\end{cases}\]

The parameters of the network can be summarized in the couple $(\mathbf{M},\mathbf{b})$, where $\mathbf{M} = (M^{0}, M^{1}, \dots , M^{K-1}) \in \mathbb{R}^{n_0 \times n_1} \times \dots \times \mathbb{R}^{n_{K-1} \times n_K}$ and $\mathbf{b} = (b^{0}, b^{1}, \dots , b^{K-1}) \in \mathbb{R}^{n_0} \times \dots \times \mathbb{R}^{n_{K-1}}$.
We formalize the transformation implemented by one layer of the network with the following definition.
\begin{df}\label{functions h}
For a network with parameters $(\mathbf{M}, \mathbf{b})$, we define the family of functions $(h_0,\dots, h_{K-1})$ such that for all $k \in \llbracket 0 , K-1 \rrbracket$, $h_k: \mathbb{R}^{n_{k+1}} \rightarrow \mathbb{R}^{n_k}$ and for all $x \in \mathbb{R}^{n_k}$,
\[h_k(x)=\begin{cases}\sigma(M^{k} x + b^{k}) & \text{if } k \neq 0 \\M^{0} x + b^{0} & \text{if } k=0.\end{cases}\]
\end{df}

 The function implemented by the network is then
 \begin{equation}f_{\mathbf{M},\mathbf{b}} = h_0 \circ h_1 \circ \dots \circ h_{K-1} : \mathbb{R}^{n_K} \longrightarrow \mathbb{R}^{n_0}.\end{equation}
The network with its parameters are represented in Figure \ref{NN representation} in the main part.

For all $l \in \llbracket 0, K-1 \rrbracket$, we denote $\mathbf{M}^{\leq l} = (M^{0}, M^{1}, \dots , M^{l})$ and $\mathbf{b}^{\leq l} = (b^{0}, b^{1}, \dots, b^{l})$.

\begin{remnum}\label{hk is CPL}
Since the vectorial ReLU function is continuous piecewise linear, Proposition \ref{composition CPL functions} guarantees that the functions $h_k$ are continuous piecewise linear.
\end{remnum}

We now define a few more functions associated to a network. 

\begin{df}\label{functions hlin}
For a network with parameters $(\mathbf{M}, \mathbf{b})$, we define the family of functions $(h^{lin}_0,\dots, h^{lin}_{K-1})$ such that for all $k \in \llbracket 0 , K-1 \rrbracket$, $h_k^{lin}: \mathbb{R}^{n_{k+1}} \rightarrow \mathbb{R}^{n_k}$ and for all $x \in \mathbb{R}^{n_{k+1}}$,
\[h_k^{lin} (x) = M^kx + b^k.\]
\end{df}

The functions $h_k^{lin}$ correspond to the linear part of the transformation implemented by the network between two layers, before applying the activation $\sigma$.

\begin{df}\label{functions f}
For a network with parameters $(\mathbf{M}, \mathbf{b})$, we define the family of functions $(f_K, f_{K-1}, \dots, f_{0})$ as follows:

\begin{itemize}
\item $f_{K} = id_{\mathbb{R}^{n_K}}$,
\item for all $k \in \llbracket 0 , K-1 \rrbracket$, \quad  $f_k = h_k \circ h_{k+1} \circ \dots \circ h_{K-1}$.
\end{itemize}

\end{df}

\begin{rem}
In particular we have $f_0 = f_{\mathbf{M},\mathbf{b}}$.
\end{rem}

The function $f_k:\mathbb{R}^{n_K} \mapsto \mathbb{R}^{n_k}$ represents the transformation implemented by the network between the input layer and the layer $k$.

\begin{df}\label{functions g}
For a network with parameters $(\mathbf{M}, \mathbf{b})$, we define the sequence $(g_0,\dots, g_{K})$ as follows:

\begin{itemize}
\item $g_{0} = id_{\mathbb{R}^{n_0}}$,
\item for all $k \in \llbracket 1 , K \rrbracket$, \quad $g_k  = h_0 \circ h_1 \circ \dots \circ h_{k-1}$.
\end{itemize}

\end{df}

\begin{rem}We have in particular
\begin{itemize}
\item $g_K = f_{\mathbf{M,b}}$;
\item for all $k \in \llbracket 0, K \rrbracket$, $f_{\mathbf{M},\mathbf{b}} = g_k\circ f_k$.
\end{itemize}
\end{rem}

The function $g_k: \mathbb{R}^{n_k} \mapsto \mathbb{R}^{n_0}$ represents the transformation implemented by the network between the layer $k$ and the output layer.

In this paper the functions implemented by the networks are considered on a subset $\Omega \subset \mathbb{R}^{n_K}$. The successive layers of a network project this subset onto the spaces $\mathbb{R}^{n_k}$, inducing a subset $\Omega_k$ of $\mathbb{R}^{n_k}$ for all $k$, as in the following definition.
\begin{df}\label{def omega_k}
For a network with parameters $(\mathbf{M}, \mathbf{b})$, for any $\Omega \subset \mathbb{R}^{n_K}$, we denote for all $k \in \llbracket 0,K \rrbracket$, 
\[\Omega_k = f_k (\Omega).\]
\end{df}

\begin{df}\label{def hyperplanes}
For a network with parameters $(\mathbf{M}, \mathbf{b})$, for all $k \in \llbracket 2, K \rrbracket$, for all $i \in \llbracket 1, n_{k-1} \rrbracket$, we define
\[H_i^k = \{x \in \mathbb{R}^{n_k} , \ M^{k-1}_{i,.}x + b_i^{k-1} = 0\}.\]
\end{df}

\begin{rem}
When $M^{k-1}_{i,.} \neq 0$, the set $H^{k}_i$ is a hyperplane.
\end{rem}

\begin{remnum}\label{notation tilde}
The objects defined in Definitions \ref{functions h}, \ref{functions hlin}, \ref{functions f}, \ref{functions g}, \ref{def omega_k} and \ref{def hyperplanes} all depend on $(\mathbf{M}, \mathbf{b})$, but to simplify the notation we do not write it explicitly. To disambiguate when manipulating a second network, whose parameters we will denote by $(\mathbf{\tilde{M}}, \mathbf{\tilde{b}})$, we will denote by $\tilde{h}_k$, $\tilde{h}_k^{lin}$,  $\tilde{f}_k$, $\tilde{g}_k$, $\tilde{\Omega}_k$ and $\tilde{H}^{k}_i $ the corresponding objects.
\end{remnum}

\begin{prop}\label{fk and gk are CPL}
For all $k \in \llbracket 0 , K\rrbracket$, $f_k$ and $g_k$ are continuous piecewise linear.
\end{prop}

\begin{proof}
We show this by induction: for the initialisation we have $f_K = id_{\mathbb{R}^{n_K}}$ which is continuous piecewise linear. Now let $k \in \llbracket 0, K-1 \rrbracket$ and assume $f_{k+1}$ is continuous piecewise linear. By definition, we have $f_k = h_k \circ f_{k+1}$. The function $h_k$ is continuous piecewise linear as noted in Remark \ref{hk is CPL}. By Proposition \ref{composition CPL functions}, the composition of two continuous piecewise linear functions is continuous piecewise linear, so $f_k$ is continuous piecewise linear. The conclusion follows by induction.

We do the same for $(g_0, \dots, g_K)$ starting with $g_0$: first we have $g_0 = id_{\mathbb{R}^{n_0}}$ which is continuous piecewise linear, then for all $k \in \llbracket 1, K \rrbracket$, we have $g_k = g_{k-1} \circ h_{k-1}$, and we conclude by composition of two continuous piecewise linear functions. 
\end{proof}

\begin{coro}
The function $f_{\mathbf{M},\mathbf{b}}$ is continuous piecewise linear.
\end{coro}

\begin{proof}
It comes immediately from $f_{\mathbf{M},\mathbf{b}} = f_0$ and Proposition \ref{fk and gk are CPL}.
\end{proof}
Recall the definition of an admissible set with respect to a continuous piecewise linear function (Definition \ref{def admissible sets of polyhedra}). Proposition \ref{fk and gk are CPL} allows the following definition.
\begin{df}\label{def of an admissible list Pi}
Consider a network parameterization $\mathbf{(M,b)}$, and the functions $g_k$ associated to it. We say that a list of sets of closed polyhedra $\boldsymbol{\Pi} = (\Pi_1, \dots, \Pi_{K-1})$ is \textit{admissible} with respect to $(\mathbf{M}, \mathbf{b})$ iif for all $k \in \llbracket 1, K-1 \rrbracket$, the set $\Pi_k$ is admissible with respect to $g_k$.
\end{df}
\begin{rem}For a list $\boldsymbol{\Pi} = (\Pi_1, \dots, \Pi_{K-1})$, for all $l \in \llbracket 1, K-1 \rrbracket$, we denote $\boldsymbol{\Pi}^{\leq l} = (\Pi_1, \dots, \Pi_{l})$. If $\boldsymbol{\Pi}$ is admissible with respect to $\mathbf{(M,b)}$, then $\boldsymbol{\Pi}^{\leq l}$ is admissible with respect to $(\mathbf{M}^{\leq l}, \mathbf{b}^{\leq l})$.
\end{rem}

\begin{prop}For any network parameterization $\mathbf{(M,b)}$, there always exists a list of sets of closed polyhedra $\boldsymbol{\Pi}$ that is admissible with respect to $\mathbf{(M,b)}$.
\end{prop}

\begin{proof}
For all $k\in \llbracket 1 , K-1 \rrbracket$, since $g_k$ is continuous piecewise linear, Proposition \ref{proposition CPL functions} guarantees that there exists an admissible set of polyhedra $\Pi_k$ with respect to $g_k$. We simply define $\boldsymbol{\Pi} = (\Pi_1, \dots, \Pi_{K-1})$.
\end{proof}

\begin{df}
For a parameterization $\mathbf{(M,b)}$ and a list $\boldsymbol{\Pi}$ admissible with respect to $\mathbf{(M,b)}$, for all $k \in \llbracket 1, K-1 \rrbracket$, for all $D \in \Pi_k$, since $g_k$ is linear over $D$ and $D$ has nonempty interior, we can define $V^k(D) \in \mathbb{R}^{n_0 \times n_k}$ and $c^k(D) \in \mathbb{R}^{n_0}$ as the unique couple that satisfies:
\[\forall x \in D, \quad g_k(x) = V^k(D) x + c^k(D).\]
\end{df}

We now introduce the equivalence relation between parameterizations, often referred to as \emph{equivalence modulo permutation and positive rescaling}.

\begin{df}[Equivalent parameterizations]\label{equivalence def}
If $(\mathbf{M},\mathbf{b})$ and $(\mathbf{\tilde{M}}, \mathbf{\tilde{b}})$ are two network parameterizations, we say that $(\mathbf{M},\mathbf{b})$ is \emph{equivalent modulo permutation and positive rescaling}, or simply \emph{equivalent}, to $(\mathbf{\tilde{M}}, \mathbf{\tilde{b}})$, and we write $(\mathbf{M},\mathbf{b}) \sim (\mathbf{\tilde{M}}, \mathbf{\tilde{b}})$, if and only if there exist:
\begin{itemize}
\item a family of permutations $\boldsymbol{\varphi} = (\varphi_0, \dots , \varphi_{K}) \in \mathfrak{S}_{n_0} \times \dots \times \mathfrak{S}_{n_{K}}$, with $\varphi_0 = id_{\llbracket 1 , n_0 \rrbracket}$ and $\varphi_{K} = id_{\llbracket 1 , n_K \rrbracket}$,
\item a family of vectors $\boldsymbol{\lambda}=(\lambda^{0}, \lambda^{1}, \dots, \lambda^{K}) \in (\mathbb{R_+^*})^{n_0} \times \dots \times( \mathbb{R_+^*})^{n_{K}}$, with $\lambda^{0} = \One_{n_0}$ and $\lambda^{K} = \One_{n_K}$,
\end{itemize}
such that for all $k \in \llbracket 0 , K-1 \rrbracket$,
\begin{equation}\label{equivalence expression}\begin{cases}\tilde{M}^{k}= P_{\varphi_{k}} \Diag(\lambda^{k} )M^{k}\Diag(\lambda^{k+1})^{-1}P_{\varphi_{k+1}}^{-1} \\
\tilde{b}^{k} = P_{\varphi_k}\Diag(\lambda^{k} ) b^{k}.\end{cases} \end{equation} 
\end{df}

\begin{rems}
\begin{enumerate} \ \newline
\vspace{-10pt}
\item Recall that we denote by $\frac{1}{\lambda^{k+1}}$ the vector whose components are $\frac{1}{\lambda^{k+1}_i}$. Note that $\Diag(\lambda^{k+1})^{-1} = \Diag(\frac{1}{\lambda^{k+1}})$. Using \eqref{permutation matrix}, for all \mbox{$k \in \llbracket0 , K-1 \rrbracket$}, (\ref{equivalence expression}) means that for all $(i,j) \in \llbracket 1 , n_k \rrbracket \times \llbracket 1 , n_{k+1} \rrbracket$,
\[\tilde{M}^{k}_{i,j} = \frac{\lambda^{k}_{\varphi_k^{-1}(i)}}{\lambda^{k+1}_{\varphi_{k+1}^{-1}(j)}} M^{k}_{\varphi_k^{-1}(i),\varphi_{k+1}^{-1}(j)}\]
and
\[\tilde{b}^{k}_{i} = \lambda^{k}_{\varphi_k^{-1}(i)} b^{k}_{\varphi_k^{-1}(i)}.\]
\item We go from a parameterization to an equivalent one by:
\begin{itemize}
\item permuting the neurons of each hidden layer $k$ with a permutation $\varphi_k$;
\item for each hidden layer $k$, multiplying all the weights of the edges arriving (from the layer $k+1$) to the neuron $j$, as well as the bias $b^k_j$, by some positive number $\lambda^k_j$, and multiplying all the weights of the edges leaving (towards the layer $k-1$) the neuron $j$ by $\frac{1}{\lambda^k_j}$.
\end{itemize}
\end{enumerate}
\end{rems}

\begin{prop}\label{the relation is an equivalence relation}
The relation $\sim$ is an equivalence relation. 
\end{prop}
\begin{proof}
Let us first show the following equality, that we are going to use in the proof. For any $n \in \mathbb{N}^*$, $\lambda \in \mathbb{R}^{n}$ and $\varphi \in \mathfrak{S}_n$, 
\begin{equation}\label{Permuting P and D}
\Diag(\lambda) P_{\varphi} = P_{\varphi} \Diag(P_{\varphi}^{-1} \lambda).
\end{equation}
Indeed, $\Diag(\lambda) P_{\varphi}$ is the matrix obtained by multiplying each line $i$ of $P_{\varphi}$ by $\lambda_i$, so recalling \eqref{permutation matrix-0}, for all $i,j \in \llbracket 1,m\rrbracket$, we have
\begin{equation*}
(\Diag(\lambda) P_{\varphi})_{i,j} = \begin{cases}  \lambda_i & \text{if } \varphi(j) = i\\ 0 & \text{otherwise.} \end{cases}
\end{equation*}
At the same time, $P_{\varphi} \Diag(P_{\varphi}^{-1} \lambda)$ is the matrix obtained by multiplying each column $j$ of $P_{\varphi}$ by $(P_{\varphi}^{-1} \lambda)_j = \lambda_{\varphi(j)}$ (see \eqref{permutation matrix} and \eqref{Inverse of a permutation matrix}), so for all $i,j \in \llbracket 1,m\rrbracket$, we have
\begin{equation*}
(P_{\varphi} \Diag(P_{\varphi}^{-1} \lambda))_{i,j} = \begin{cases}  \lambda_{\varphi(j)} & \text{if } \varphi(j) = i\\ 0 & \text{otherwise.} \end{cases}
\end{equation*}
The two matrices are clearly equal.

We can now show the proposition.
\begin{itemize}
\item To show reflexivity we can take $\lambda^{k} = \One_{n_k}$ and $\varphi_k = id_{\llbracket 1 , n_k \rrbracket}$ for all $k \in \llbracket 0, K \rrbracket$.
\item Let us show symmetry. Assume a parameterization $\mathbf{(M,b)}$ is equivalent to another parameterization $\mathbf{(\tilde{M}, \tilde{b})}$. Let us denote by $\boldsymbol{\varphi}$ and $\boldsymbol{\lambda}$ the corresponding families of permutations and vectors, as in Definition \ref{equivalence def}. Inverting the expression of $\tilde{M}^k$ in Definition \ref{equivalence def} and using \eqref{Permuting P and D} twice, we have for all $k \in \llbracket 0, K-1 \rrbracket$:
\begin{equation*}
\begin{aligned}
& \tilde{M}^{k} = P_{\varphi_{k}} \Diag(\lambda^{k} )M^{k}\Diag(\lambda^{k+1})^{-1}P_{\varphi_{k+1}}^{-1} \\
 \Longleftrightarrow  \quad & \Diag(\lambda^{k})^{-1} P_{\varphi_{k}}^{-1} \tilde{M}^{k} P_{\varphi_{k+1}} \Diag(\lambda^{k+1}) = M^{k} \\
 \Longleftrightarrow  \quad & P_{\varphi_{k}}^{-1} \Diag(P_{\varphi_{k}} \lambda^{k})^{-1} \tilde{M}^{k}  \Diag(P_{\varphi_{k+1}}\lambda^{k+1}) P_{\varphi_{k+1}}  = M^{k},
\end{aligned}
\end{equation*}
so denoting $\tilde{\varphi}_k = \varphi_{k}^{-1}$ and $\tilde{\lambda}^{k} = (P_{\varphi_{k}} \lambda^{k})^{-1}$, and recalling that $P_{\varphi_k^{-1}} = P_{\varphi_k}^{-1}$, we have, for all \mbox{$k \in \llbracket 0, K-1 \rrbracket$},
\[ M^{k} = P_{\tilde{\varphi}_{k}} \Diag(\tilde{\lambda}^{k}) \tilde{M}^{k}  \Diag(\tilde{\lambda
}^{k+1})^{-1} P_{\tilde{\varphi}_{k+1}}^{-1} .\]
We show similarly that for all \mbox{$k \in \llbracket 0, K-1 \rrbracket$},
\[b^k = P_{\tilde{\varphi}_k} \Diag(\tilde{\lambda}_k)\tilde{b}^k.\]
We naturally have $\tilde{\varphi}_0 = id_{\llbracket 1, n_0 \rrbracket}$ and $\tilde{\varphi}_K = id_{\llbracket 1, n_K \rrbracket}$, as well as $\tilde{\lambda}^0 = \One_{n_0}$ and $\tilde{\lambda}^K = \One_{n_K}$.

This proves the symmetry of the relation.

\item  Let us show transitivity. Assume $\mathbf{(M,b)}$, $\mathbf{(\tilde{M}, \tilde{b})}$ and $\mathbf{(\check{M}, \check{b})}$ are three parameterizations such that $\mathbf{(M,b)} \sim \mathbf{(\tilde{M}, \tilde{b})}$ and $\mathbf{(\tilde{M}, \tilde{b})} \sim \mathbf{(\check{M}, \check{b})}$.

As in Definition \ref{equivalence def}, we denote by $\boldsymbol{\varphi}$, $\tilde{\boldsymbol{\varphi}}$, $\boldsymbol{\lambda}$ and $\tilde{\boldsymbol{\lambda}}$ the families of permutations and vectors such that, for all $k \in \llbracket 0 , K-1 \rrbracket$,
\[\begin{cases}\tilde{M}^{k} = P_{\varphi_{k}} \Diag(\lambda^{k} )M^{k}\Diag(\lambda^{k+1})^{-1}P_{\varphi_{k+1}}^{-1}  \\ \tilde{b}^{k} = P_{{\varphi}_{k}} \Diag({\lambda}^{k} ) b^{k},
\end{cases}\]
and
\[\begin{cases}
\check{M}^{k} = P_{\tilde{\varphi}_{k}} \Diag(\tilde{\lambda}^{k} )\tilde{M}^{k}\Diag(\tilde{\lambda}^{k+1})^{-1}P_{\tilde{\varphi}_{k+1}}^{-1} \\ \check{b}^{k} = P_{\tilde{\varphi}_{k}} \Diag(\tilde{\lambda}^{k} ) \tilde{b}^{k}.
\end{cases}\]

Combining these and using \eqref{Permuting P and D}, we have
\begin{equation*}
\begin{aligned}
\check{M}^{k} & = P_{\tilde{\varphi}_{k}} \Diag(\tilde{\lambda}^{k} )P_{\varphi_{k}} \Diag(\lambda^{k} )M^{k}\Diag(\lambda^{k+1})^{-1}P_{\varphi_{k+1}}^{-1} \Diag(\tilde{\lambda}^{k+1})^{-1}P_{\tilde{\varphi}_{k+1}}^{-1} \\
& = P_{\tilde{\varphi}_{k}} \left( \Diag(\tilde{\lambda}^{k} )P_{\varphi_{k}} \right) \Diag(\lambda^{k} )M^{k} \\ & \hspace{90pt} \cdot \Diag(\lambda^{k+1})^{-1}\left( \Diag(\tilde{\lambda}^{k+1})P_{\varphi_{k+1}}\right)^{-1} P_{\tilde{\varphi}_{k+1}}^{-1} \\
& = P_{\tilde{\varphi}_{k}} \left( P_{\varphi_{k}} \Diag(P_{\varphi_{k}}^{-1}\tilde{\lambda}^{k} ) \right) \Diag(\lambda^{k} )M^{k} \\
& \hspace{90pt} \cdot \Diag(\lambda^{k+1})^{-1}\left( P_{\varphi_{k+1}} \Diag( P_{\varphi_{k+1}}^{-1}\tilde{\lambda}^{k+1})\right)^{-1} P_{\tilde{\varphi}_{k+1}}^{-1} \\
& = P_{\tilde{\varphi}_{k}}  P_{\varphi_{k}} \Diag(P_{\varphi_{k}}^{-1}\tilde{\lambda}^{k} ) \Diag(\lambda^{k} )M^{k}  \\
& \hspace{90pt} \cdot \Diag(\lambda^{k+1})^{-1} \Diag(P_{\varphi_{k+1}}^{-1}\tilde{\lambda}^{k+1})^{-1}P_{\varphi_{k+1}}^{-1}P_{\tilde{\varphi}_{k+1}}^{-1}, 
\end{aligned}
\end{equation*}
and
\begin{equation*}
\begin{aligned}
\check{b}^{k}& = P_{\tilde{\varphi}_{k}} \Diag(\tilde{\lambda}^{k} )P_{\varphi_{k}} \Diag(\lambda^{k} ) b^{k} \\
 & = P_{\tilde{\varphi}_{k}}P_{\varphi_{k}} \Diag(P_{\varphi_{k}}^{-1}\tilde{\lambda}^{k} ) \Diag(\lambda^{k} ) b^{k}.
\end{aligned}
\end{equation*}

Hence denoting $\check{\varphi}_k = \tilde{\varphi}_{k} \circ \varphi_{k}$ and $\check{\lambda}^{k} = \Diag(P_{\varphi_{k}}^{-1}\tilde{\lambda}^{k}) \lambda^{k}$, for all $k \in \llbracket 0, K \rrbracket$, we see that, for $k \in \llbracket 0, K-1 \rrbracket$,
\[\check{M}^{k} = P_{\check{\varphi}_{k}} \Diag(\check{\lambda}^{k} )M^{k}\Diag(\check{\lambda}^{k+1})^{-1}P_{\check{\varphi}_{k+1}}^{-1}\]
and
\[\check{b}^{k} = P_{\check{\varphi}_{k}} \Diag(\check{\lambda}^{k} ) b^{k}.\]
Naturally, we also have $\check{\varphi}_0 = id_{\llbracket 1, n_0 \rrbracket}$ and $\check{\varphi}_K = id_{\llbracket 1, n_K \rrbracket}$, as well as $\check{\lambda}^0 = \One_{n_0}$ and $\check{\lambda}^K = \One_{n_K}$, which shows that $\mathbf{(M,b)} \sim \mathbf{(\check{M}, \check{b})}$.
\end{itemize}
\end{proof}

Recall the objects $h_k, f_k, g_k, \Omega_k, H_i^k$ associated to a parameterization $(\Mbf, \bbf)$, defined in Definitions \ref{functions h}, \ref{functions f}, \ref{functions g}, \ref{def omega_k} and \ref{def hyperplanes}, and recall that we denote by $\tilde{h}_k, \tilde{f}_k, \tilde{g}_k, \tilde{\Omega}_k$ and $ \tilde{H}_i^k$ the corresponding objects with respect to another parameterization $\mathbf{(\tilde{M}, \tilde{b})}$. We give in the following proposition the relations that link these objects when the two parameterizations $\mathbf{(M,b)}$ and $\mathbf{(\tilde{M}, \tilde{b})}$ are equivalent.

\begin{prop}\label{object relations for equivalent networks}
Assume $(\mathbf{M},\mathbf{b}) \sim (\mathbf{\tilde{M}}, \mathbf{\tilde{b}})$ and consider $\boldsymbol{\varphi}$ and $\boldsymbol{\lambda} $ as in Definition \ref{equivalence def}. Let $\boldsymbol{\Pi}$ be a list of sets of closed polyhedra that is admissible with respect to $\mathbf{(M,b)}$.
Then:
\begin{enumerate}
\item\label{relation between h and tilde(h)} for all $k \in \llbracket 0, K-1 \rrbracket$, 
\[\tilde{h}_k = P_{\varphi_k} \Diag(\lambda^{k}) \circ h_k \circ \Diag(\lambda^{k+1})^{-1}P_{\varphi_{k+1}}^{-1},\]
\item\label{relation between f g omega and tilde(f) tilde(g) tilde(omega)} for all $k \in \llbracket 0, K \rrbracket$,
\begin{equation}\label{expression of tilde(f)}
\tilde{f}_k = P_{\varphi_k}  \Diag(\lambda^{k}) \circ f_k,
\end{equation}
\[\tilde{g}_k = g_k \circ \Diag(\lambda^{k})^{-1} P_{\varphi_{k}}^{-1},\]
\[\tilde{\Omega}_k = P_{\varphi_{k}}\Diag(\lambda^{k}) \Omega_k,\]

\item for all $k \in \llbracket 2,K \rrbracket$, for all $i \in \llbracket 1 , n_{k-1} \rrbracket$,
\[\tilde{H}_i^k = P_{\varphi_{k}}\Diag(\lambda^{k})H_{\varphi_{k-1}^{-1}(i)}^k,\]
\item\label{relations between polyhedra} for all $k \in \llbracket 1, K-1 \rrbracket$, the set of closed polyhedra $\tilde{\Pi}_k = \{ P_{\varphi_{k}} \Diag(\lambda^{k}) D , D \in \Pi_k \}$ is admissible for $\tilde{g}_k$, i.e. the list $\boldsymbol{\tilde{\Pi}} = (\tilde{\Pi}_1, \dots, \tilde{\Pi}_{K-1})$ is admissible with respect to $\mathbf{(\tilde{M}, \tilde{b})}$.
\end{enumerate}
\end{prop}

\begin{proof}
\begin{enumerate}
\item Let $k \in \llbracket 0, K-1 \rrbracket$. If $k \neq 0$, we have from Definition \ref{functions h}:

\begin{equation*}
\begin{aligned}
\tilde{h}_k(x) & = \sigma(\tilde{M}^{k} x + \tilde{b}^{k}) \\
& = \sigma \left(P_{\varphi_{k}} \Diag(\lambda^{k})M^{k}\Diag(\lambda^{k+1})^{-1}P_{\varphi_{k+1}}^{-1} x \right.\\
& \left. \hspace{35pt} + P_{\varphi_k}\Diag(\lambda^{k} ) b^{k}\right) \\
& =\sigma\left(P_{\varphi_{k}} \Diag(\lambda^{k}) \left[M^{k}\Diag(\lambda^{k+1})^{-1}P_{\varphi_{k+1}}^{-1} x + b^{k} \right] \right). \\
\end{aligned}
\end{equation*}

Denote $y:=  \left[M^{k}\Diag(\lambda^{k+1})^{-1}P_{\varphi_{k+1}}^{-1} x + b^{k} \right]$. Let $i \in \llbracket 1, n_k \rrbracket$. Using \eqref{permutation matrix} and the fact that $ \lambda^{k}_{\varphi_k^{-1}(i)}$ is nonnegative, the $i^{\text{th}}$ coordinate of $\tilde{h}_k(x)$ is

\begin{equation*}
\begin{aligned}
\tilde{h}_k(x)_i = \left[\sigma\left(P_{\varphi_{k}} \Diag(\lambda^{k})y \right) \right]_i & = \sigma\left(\left[P_{\varphi_{k}} \Diag(\lambda^{k})y \right]_i \right) \\
 & = \sigma\left(\lambda^{k}_{\varphi_k^{-1}(i)} y_{\varphi_k^{-1}(i)}\right) \\
 & = \lambda^{k}_{\varphi_k^{-1}(i)} \sigma\left(y_{\varphi_k^{-1}(i)}\right)\\
 & =  \left[P_{\varphi_{k}} \Diag(\lambda^{k}) \sigma\left(y \right) \right]_i.
\end{aligned}
\end{equation*}

Finally, we find the expression of $\tilde{h}_k(x)$:
\begin{equation*}
\begin{aligned}
\tilde{h}_k(x)  & = P_{\varphi_{k}} \Diag(\lambda^{k}) \sigma\left(y\right) \\
& = P_{\varphi_{k}} \Diag(\lambda^{k}) \sigma\left( M^{k}\Diag(\lambda^{k+1})^{-1}P_{\varphi_{k+1}}^{-1} x + b^{k}  \right) \\
& = P_{\varphi_{k}} \Diag(\lambda^{k})h_k \left( \Diag(\lambda^{k+1})^{-1}P_{\varphi_{k+1}}^{-1} (x)\right).
\end{aligned}
\end{equation*}

This concludes the proof when $k \neq 0$.

The case $k=0$ is proven similarly but replacing the ReLU function $\sigma$ by the identity.

\item \begin{itemize}
\item We prove by induction the expression of $\tilde{f}_k$. 

For $k = K$, we have $\tilde{f}_K = f_K = id_{\mathbb{R}^{n_K}}$, and since $P_{\varphi_K} = \Id_{n_K}$ and $\lambda^{K} = \One_{n_K}$ the equality $\tilde{f}_K = P_{\varphi_K}  \Diag(\lambda^{K}) f_K $ holds.

Now let $k \in \llbracket 0 , K-1 \rrbracket$. Suppose the induction hypothesis is true for $\tilde{f}_{k+1}$. 
Using the expression of $\tilde{h}_k$ we just proved in 1 and the induction hypothesis, we have
\begin{equation*}
\begin{aligned}
\tilde{f}_k & = \tilde{h}_k \circ \tilde{f}_{k+1} \\
& = \left(P_{\varphi_k} \Diag(\lambda^{k}) \circ h_k \circ \Diag(\lambda^{k+1})^{-1}P_{\varphi_{k+1}}^{-1}\right) \circ \left(P_{\varphi_{k+1}}  \Diag(\lambda^{k+1}) \circ f_{k+1} \right) \\
& = P_{\varphi_k} \Diag(\lambda^{k})  \circ h_k \circ  f_{k+1} \\
& = P_{\varphi_k} \Diag(\lambda^{k}) \circ f_k.
\end{aligned}
\end{equation*}
This concludes the induction.

\item We prove similarly the expression of $\tilde{g}_k$, but starting from $k = 0$: first we have $\tilde{g}_0 = g_0 = id_{\mathbb{R}^{n_0}}$, and then, for $k \in \llbracket 0 , K-1 \rrbracket$, we write $\tilde{g}_{k+1} = \tilde{g}_{k} \circ \tilde{h}_{k}$ and we use the induction hypothesis and the expression of $\tilde{h}_k$.

\item Using the relation \eqref{expression of tilde(f)}, that we just proved, we obtain 
\[\tilde{\Omega}_k = \tilde{f}_k(\Omega) = P_{\varphi_k} \Diag(\lambda^{k})  f_k (\Omega) = P_{\varphi_k} \Diag(\lambda^{k}) \Omega_k.\]
\end{itemize}

\item Let $k \in \llbracket 2, K \rrbracket$ and $i \in \llbracket 1, n_{k-1} \rrbracket$. For all $x \in \mathbb{R}^{n_k}$, using \eqref{equivalence expression} and \eqref{permutation matrix},
\begin{equation*}
\begin{aligned}
x \in \tilde{H}^k_i & \quad \Longleftrightarrow \quad \tilde{M}^{k-1}_{i,.}x + \tilde{b}_i^{k-1} = 0 \\
& \quad \Longleftrightarrow \quad \left[ P_{\varphi_{k-1}} \Diag(\lambda^{k-1} )M^{k-1}\Diag(\lambda^{k})^{-1}P_{\varphi_{k}}^{-1} \right]_{i,.} x \\ & \hspace{50pt} + \left[P_{\varphi_{k-1}} \Diag(\lambda^{k-1} ) b^{k-1} \right]_{i} =0 \\
&  \quad \Longleftrightarrow  \quad \lambda^{k-1}_{\varphi_{k-1}^{-1}(i)}M^{k-1}_{\varphi_{k-1}^{-1}(i),.}\Diag(\lambda^{k})^{-1}P_{\varphi_{k}}^{-1} x + \lambda^{k-1}_{\varphi_{k-1}^{-1}(i)}b^{k-1}_{\varphi_{k-1}^{-1}(i)} = 0 \\
&  \quad \Longleftrightarrow  \quad \lambda^{k-1}_{\varphi_{k-1}^{-1}(i)} \left(M^{k-1}_{\varphi_{k-1}^{-1}(i),.}\Diag(\lambda^{k})^{-1}P_{\varphi_{k}}^{-1} x + b^{k-1}_{\varphi_{k-1}^{-1}(i)} \right) = 0 \\
& \quad \Longleftrightarrow \quad  M^{k-1}_{\varphi_{k-1}^{-1}(i),.}\Diag(\lambda^{k})^{-1}P_{\varphi_{k}}^{-1} x + b^{k-1}_{\varphi_{k-1}^{-1}(i)} = 0 \\
& \quad \Longleftrightarrow \quad \Diag(\lambda^{k})^{-1}P_{\varphi_{k}}^{-1} x \in H^k_{\varphi_{k-1}^{-1}(i)}. \\
\end{aligned}
\end{equation*}
Thus, $\tilde{H}^k_i = P_{\varphi_{k}}\Diag(\lambda^{k})H_{\varphi_{k-1}^{-1}(i)}^k$.

\item For all $D \in \Pi_k$, denote $ \tilde{D} = P_{\varphi_{k}} \Diag(\lambda^{k}) D$. We have $\tilde{\Pi}_k = \{ \tilde{D}, D \in \Pi_k \}$.

Let $D \in \Pi_k$. The matrix $P_{\varphi_{k}} \Diag(\lambda^{k})$ is invertible so, according to Proposition \ref{reciprocal image of a polyhedron}, $\tilde{D} = P_{\varphi_{k}} \Diag(\lambda^{k}) D$ is a closed polyhedron, and since $\ring{D} \neq \emptyset$ we also have $\ring{\tilde{D}} \neq \emptyset$. 

Now recall from Item \ref{relation between f g omega and tilde(f) tilde(g) tilde(omega)} that:
\[\tilde{g}_k  = g_k \circ \Diag(\lambda^{k})^{-1} P_{\varphi_{k}}^{-1} .\]
 For all $x \in \tilde{D}$, we have $\Diag(\lambda^{k})^{-1} P_{\varphi_{k}}^{-1} x \in D$. Since $\Pi_k$ is admissible with respect to $g_k$ (by definition of $\boldsymbol{\Pi}$), $g_k$ is linear on $D$, and thus the function $\tilde{g}_k$ is linear on $\tilde{D}$.

Again, since $\Pi_k$ is admissible with respect to $g_k$, we have $\bigcup_{D \in \Pi_k} D = \mathbb{R}^m$, and thus 
\begin{equation*}
\begin{aligned}
\bigcup_{\tilde{D} \in \tilde{\Pi}_k} \tilde{D} & = \bigcup_{D \in \Pi_k} P_{\varphi_{k}} \Diag(\lambda^{k}) D \\
& = P_{\varphi_{k}} \Diag(\lambda^{k}) \left( \bigcup_{D \in \Pi_k} D \right) \\
& =  P_{\varphi_{k}} \Diag(\lambda^{k}) \left( \mathbb{R}^m \right) \\
& = \mathbb{R}^m ,
\end{aligned}
\end{equation*}
which shows that $\tilde{\Pi}_k$ is admissible with respect to $\tilde{g}_k$. 

This being true for any $k \in \llbracket 1, K-1 \rrbracket$, we conclude that $\boldsymbol{\tilde{\Pi}}$ is admissible with respect to $\mathbf{(\tilde{M}, \tilde{b})}$.
 \end{enumerate}
\end{proof}

\begin{coro}\label{equivalent networks have same function}
If $(\mathbf{M},\mathbf{b}) \sim (\mathbf{\tilde{M}}, \mathbf{\tilde{b}})$, then $f_{\mathbf{M},\mathbf{b}} = f_{\mathbf{\tilde{M}}, \mathbf{\tilde{b}}}$.
\end{coro}

\begin{proof}

Consider $\boldsymbol{\varphi}$ and $\boldsymbol{\lambda}$ as in Definition \ref{equivalence def}. Looking at \eqref{expression of tilde(f)} for $k=0$, and using the fact that $f_0 = f_{\mathbf{M},\mathbf{b}}$ and $\tilde{f}_0 = f_{\mathbf{\tilde{M}}, \mathbf{\tilde{b}}}$, we obtain from Proposition \ref{object relations for equivalent networks}
\[f_{\mathbf{\tilde{M}}, \mathbf{\tilde{b}}} = P_{\varphi_0}  \Diag(\lambda^{0}) f_{\mathbf{M},\mathbf{b}} .\]

By definition of $\boldsymbol{\varphi}$ and $\boldsymbol{\lambda}$, we have $P_{\varphi_0} = \Id_{n_0}$ and $\lambda^{0} = \One_{n_0}$, so we can finally conclude:

\[f_{\mathbf{\tilde{M}}, \mathbf{\tilde{b}}} = f_{\mathbf{M},\mathbf{b}} .\]
\end{proof}

\begin{df}
We say that $(\mathbf{M},\mathbf{b})$ is normalized if for all $k \in \llbracket 1, K-1 \rrbracket$, for all $i \in \llbracket 1 , n_k \rrbracket $, we have:
\[\|{M}^{k}_{i,.}\|=1.\]
\end{df}

\begin{prop}\label{equivalent normalized network}
If $(\mathbf{M},\mathbf{b})$ satisfies, for all $k \in \llbracket 1, K-1 \rrbracket$, for all $i \in \llbracket 1 , n_k \rrbracket $, $M^{k}_{i,.} \neq 0$, then there exists an equivalent parameterization $(\mathbf{\tilde{M}}, \mathbf{\tilde{b}})$ that is normalized.
\end{prop}

\begin{proof}
We define recursively the family $(\lambda^{0}, \lambda^{1}, \dots, \lambda^{K}) \in (\mathbb{R_+^*})^{n_0} \times \dots \times( \mathbb{R_+^*})^{n_{K}}$ by:
\begin{itemize}
\item $\lambda^{K} = \One_{n_K}$;
\item for all $k \in \llbracket 1, K-1 \rrbracket$, for all $i \in \llbracket 1 , n_{k}\rrbracket$,
\[\lambda^{k}_i = \frac{1}{\|M^{k}_{i,.}\Diag(\lambda^{k+1})^{-1}\|};\]
\item $\lambda^{0}= \One_{n_0}$.
\end{itemize}

Consider the parameterization $(\mathbf{\tilde{M}}, \mathbf{\tilde{b}})$ defined by, for all $k \in \llbracket 0 , K-1 \rrbracket$:
\begin{equation*}\begin{cases}\tilde{M}^{k}= \Diag(\lambda^{k} )M^{k}\Diag(\lambda^{k+1})^{-1} \\
\tilde{b}^{k} = \Diag(\lambda^{k} ) b^{k}.\end{cases} \end{equation*}

The parameterization is, by definition, equivalent to $\mathbf{(M,b)}$, and, for all $k \in \llbracket 1 , K-1 \rrbracket$, for all $i \in \llbracket 1 , n_k \rrbracket$:
\begin{equation*}
\begin{aligned}
\|\tilde{M}^{k}_{i,.}\| &= \left\| \left[\Diag(\lambda^{k}) M^{k}\Diag(\lambda^{k+1})^{-1} \right]_{i,.} \right\|\\
& = \left\|\lambda^{k}_i M^{k}_{i,.}\Diag(\lambda^{k+1})^{-1} \right\|  \\
& = \left\|\frac{1}{\|M^{k}_{i,.}\Diag(\lambda^{k+1})^{-1}\|} M^{k}_{i,.}\Diag(\lambda^{k+1})^{-1} \right\| \\
&=1.
\end{aligned}
\end{equation*}
\end{proof}

\begin{prop}\label{permuted networks}
If $(\mathbf{M},\mathbf{b})$ and $(\mathbf{\tilde{M}}, \mathbf{\tilde{b}})$ are both normalized, then they are equivalent if and only if there exists a family of permutations $(\varphi_0, \dots , \varphi_{K}) \in \mathfrak{S}_{n_0} \times \dots \times \mathfrak{S}_{n_{K}}$, with $\varphi_0 = id_{\llbracket 1 , n_0 \rrbracket}$ and $\varphi_{K} = id_{\llbracket 1 , n_K \rrbracket}$, such that for all $k \in \llbracket 0 , K-1 \rrbracket$:
\begin{equation}\label{permuted networks equation}\begin{cases}\tilde{M}^{k}= P_{\varphi_{k}} M^{k}P_{\varphi_{k+1}}^{-1} \\
\tilde{b}^{k} = P_{\varphi_k} b^{k}.\end{cases} \end{equation}
\end{prop}

\begin{proof}

Assume $(\mathbf{M},\mathbf{b})$ and $(\mathbf{\tilde{M}}, \mathbf{\tilde{b}})$ are equivalent. Then there exist a family of permutations $(\varphi_0, \dots , \varphi_{K})  \in \mathfrak{S}_{n_0} \times \dots \times \mathfrak{S}_{n_{K}}$ and a family $(\lambda^{0},\dots, \lambda^{K}) \in (\mathbb{R_+^*})^{n_0} \times \dots \times( \mathbb{R_+^*})^{n_{K}}$ as in Definition \ref{equivalence def}.

Let us prove by induction that $\lambda^{k} = \One_{n_k}$ for all $ k \in \llbracket 0, K \rrbracket$.

For $k=K$ it is true by Definition \ref{equivalence def}.

Let $k \in \llbracket 1, K-1 \rrbracket$, and suppose $\lambda^{k+1} = \One_{n_{k+1}}$. This means $\Diag(\lambda^{k+1})= \Id_{n_{k+1}}$. Let $i \in \llbracket 1 , n_k \rrbracket$. Since $\mathbf{(M,b)}$ is normalized, $\|M^{k}_{i,.} \|=1$. Since $P_{\varphi_{k+1}}^{-1}$ is a permutation matrix, it is orthogonal so $\|M^{k}_{i,.} P_{\varphi_{k+1}}^{-1} \| = \|M^{k}_{i,.} \| = 1$. Recalling \eqref{equivalence expression} and using the fact that $\mathbf{(\tilde{M}, \tilde{b})}$ is normalized, that $\Diag(\lambda^{k+1})= \Id_{n_{k+1}}$ and that $\lambda^{k}_i$ is positive, we have:
\begin{equation*}
\begin{aligned}
1 = \|\tilde{M}^{k}_{\varphi_k(i),.} \| & = \|\lambda^{k}_i M^{k}_{i,.} \Diag(\lambda^{k+1})^{-1} P_{\varphi_{k+1}}^{-1} \| \\
& = \lambda^{k}_i \|M^{k}_{i,.} P_{\varphi_{k+1}}^{-1} \|\\
& = \lambda^{k}_i.
\end{aligned}
\end{equation*}
This shows $\lambda^{k} = \One_{n_k}$.

The case $k=0$ is also true by Definition \ref{equivalence def}.

Equation \eqref{equivalence expression} with $\lambda^{k} = \One_{n_k}$ for all $k \in \llbracket 0, K\rrbracket$ is precisely equation \eqref{permuted networks equation}.

The reciprocal is clear: \eqref{permuted networks equation} is a particular case of \eqref{equivalence expression} with $\lambda^{k} = \One_{n_k}$.
\end{proof}

\section{Main theorem}\label{Appendix-B}

In Appendix \ref{Appendix-B}, we prove the main theorem using the notations and results of Appendix \ref{Appendix-A}, and admitting Lemma \ref{Fundamental lemma}, which is proven in Appendix \ref{proof of the lemma}.

More precisely, we begin by stating the conditions $\Cbf$ and $\Pbf$ in Section \ref{Conditions-sec}, we then state our main result, which is Theorem \ref{main theorem}, in Section \ref{main theorem-appendix sec}, and we give a consequence of this result in terms of risk minimization, which is Corollary \ref{risk minimization formulation}, in Section \ref{Risk minimization-sec}. Finally we prove Theorem \ref{main theorem} and Corollary \ref{risk minimization formulation} in Sections \ref{Proof of the main theorem-apdx} and \ref{Proof of the corollary-sec} respectively.

\subsection{Conditions}\label{Conditions-sec}

Assume $g : \mathbb{R}^m \rightarrow \mathbb{R}^n$ is a continuous piecewise linear function, $\Pi$ is a set of closed polyhedra admissible with respect to $g$, and let $\Omega \subset \mathbb{R}^l$, $M \in \mathbb{R}^{m \times l}$ and $b \in \mathbb{R}^{m}$.

We define
\[\fonction{h}{\mathbb{R}^l}{\mathbb{R}^m}{x}{\sigma(Mx + b)}\]
and
\[\fonction{h^{lin}}{\mathbb{R}^l}{\mathbb{R}^m}{x}{Mx + b.}\]

\begin{df}
For all $i \in \llbracket 1 , m \rrbracket$, we denote $E_i = \{x \in \mathbb{R}^m, \ x_i = 0 \}$.
\end{df}

\begin{df}\label{def v_i} Let $D \in \Pi$.  The function $g$ coincides with a linear function on $D$. Since the interior of $D$ is nonempty, we define $V(D) \in \mathbb{R}^{n \times m}$ and $c(D) \in \mathbb{R}^n$ as the unique couple satisfying, for all $x \in D$:
\[g(x) = V(D)x + c(D).\]
\end{df}

\begin{df}
We say that $(g,M,b,\Omega,\Pi)$ satisfies the conditions $\mathbf{C}$ iif:
\begin{itemize}
\item[$\mathbf{C}.a)$] $M$ is full row rank;
\item[$\mathbf{C}.b)$] for all $i \in \llbracket 1 , m \rrbracket$, there exists $x \in \ring{\Omega}$ such that
\[M_{i,.}x + b_i = 0,\]
or equivalently, \[E_i \cap h^{lin}(\ring{\Omega}) \neq \emptyset;\]
\item[$\mathbf{C}.c)$] for all $D \in \Pi$, for all $i \in \llbracket 1 , m \rrbracket$, if $E_i \cap D \cap h(\Omega) \neq \emptyset$ then $V_{.,i}(D) \neq 0$;
\item[$\mathbf{C}.d)$] for any affine hyperplane $H \subset \mathbb{R}^{l}$,
\[H \cap \ring{\Omega} \ \not\subset \bigcup_{D \in \Pi} \partial h^{-1}(D).\]
\end{itemize}
\end{df}

\begin{df}
For all $k \in \llbracket 1, K-1 \rrbracket $, for all $i \in \llbracket 1 , n_k \rrbracket$, we denote $E_i^k = \{x \in \mathbb{R}^{n_k}, x_i = 0 \}$.
\end{df}

We now state the conditions $\Pbf$ (already stated in the main text in Definition \ref{Conditions P-main}).

\begin{df}\label{Conditions P}
We say that $(\mathbf{M},\mathbf{b},\Omega, \boldsymbol{\Pi})$ satisfies the conditions $\mathbf{P}$ iif for all $k \in \llbracket 1 , K-1 \rrbracket$, $ (g_k, M^{k}, b^{k}, \Omega_{k+1}, \Pi_k)$ satisfies the conditions $\mathbf{C}$. 

Explicitly, for all $k \in \llbracket 1, K-1 \rrbracket$, the conditions are the following:
\begin{itemize}
\item[$\mathbf{P}.a)$] $M^k$ is full row rank;
\item[$\mathbf{P}.b)$] for all $i \in \llbracket 1 , n_k \rrbracket$, there exists $x \in \ring{\Omega}_{k+1}$ such that
\[M^k_{i,.}x + b^k_i = 0,\]
or equivalently
\[E_i^k \cap h_k^{lin}(\ring{\Omega}_{k+1}) \neq \emptyset;\]
\item[$\mathbf{P}.c)$] for all $D \in \Pi_k$, for all $i \in \llbracket 1 , n_k \rrbracket$, if $E^k_i \cap D \cap \Omega_{k} \neq \emptyset$ then $V^k_{.,i}(D) \neq 0$;
\item[$\mathbf{P}.d)$] for any affine hyperplane $H \subset \mathbb{R}^{n_{k+1}}$,
\[H \cap \ring{\Omega}_{k+1} \ \not\subset \bigcup_{D \in \Pi_k} \partial h_k^{-1}(D).\]
\end{itemize}
\end{df}

\begin{remnum}\label{omega has nonempty interior}
The condition $\Pbf.b)$ implies that for all $k \in \llbracket 1, K-1 \rrbracket$, $\ring{\Omega}_{k+1} \neq \emptyset$, and in particular for $k = K-1$, the set $\Omega = \Omega_K$ has nonempty interior.
\end{remnum}

The following proposition shows that the conditions $\mathbf{P}$ are stable modulo permutation and positive rescaling, as defined in Definition \ref{equivalence def}.

\begin{prop}\label{conditions P are stable by equivalence}
Suppose $(\mathbf{M}, \mathbf{b})$ and $(\mathbf{\tilde{M}}, \mathbf{\tilde{b}})$ are two equivalent network parameterizations, and suppose $(\mathbf{M}, \mathbf{b}, \Omega, \boldsymbol{\Pi})$ satisfies the conditions $\mathbf{P}$. Then, if we define $\boldsymbol{\tilde{\Pi}}$ as in Item \ref{relations between polyhedra} of Proposition \ref{object relations for equivalent networks}, $(\mathbf{\tilde{M}}, \mathbf{\tilde{b}}, \Omega, \boldsymbol{\tilde{\Pi}})$ satisfies the conditions $\mathbf{P}$.
\end{prop}

\begin{proof}
Since $(\mathbf{M}, \mathbf{b})$ and $(\mathbf{\tilde{M}}, \mathbf{\tilde{b}})$ are equivalent, by Definition \ref{equivalence def} there exist 

\begin{itemize}
\item a family of permutations $(\varphi_0, \dots , \varphi_{K}) \in \mathfrak{S}_{n_0} \times \dots \times \mathfrak{S}_{n_{K}}$, with $\varphi_0 = id_{\llbracket 1 , n_0 \rrbracket}$ and $\varphi_{K} = id_{\llbracket 1 , n_K \rrbracket}$,
\item a family $(\lambda^{0}, \lambda^{1}, \dots, \lambda^{K}) \in (\mathbb{R_+^*})^{n_0} \times \dots \times( \mathbb{R_+^*})^{n_{K}}$, with $\lambda^{0} = \One_{n_0}$ and $\lambda^{K} = \One_{n_K}$,
\end{itemize}
such that
\begin{equation}\label{eq proof 26}
\begin{cases}\tilde{M}^{k}= P_{\varphi_{k}} \Diag(\lambda^{k} )M^{k}\Diag(\lambda^{k+1})^{-1}P_{\varphi_{k+1}}^{-1} \\
\tilde{b}^{k} = P_{\varphi_k}\Diag(\lambda^{k} ) b^{k}.\end{cases} \end{equation}

Let $k \in \llbracket 1, K-1 \rrbracket$. We know the conditions $\mathbf{P}.a) - \mathbf{P}.d)$ are satisfied by  $(g_k, M^{k}, b^{k}, \Omega_{k+1}, \Pi_k)$, let us show they are satisfied by $(\tilde{g}_k, \tilde{M}^{k}, \tilde{b}^{k}, \tilde{\Omega}_{k+1}, \tilde{\Pi}_k)$.

\begin{itemize}
\item[$\mathbf{P}.a)$] Since $M^{k}$ satisfies $\Pbf.a)$, it is full row rank, and using \eqref{eq proof 26} and the fact that the matrices $P_{\varphi_{k}}, \Diag(\lambda^{k} ), \Diag(\lambda^{k+1})^{-1}$ and $P_{\varphi_{k+1}}^{-1}$ are invertible, we see that $\tilde{M}^k$ is full row rank.

\item[$\mathbf{P}.b)$] Let $i \in \llbracket 1, n_k \rrbracket$.
Since $ (g_k, M^{k}, b^{k}, \Omega_{k+1}, \Pi_k)$ satisfies the condition $\mathbf{P}.b)$, we can choose $x \in \ring{\Omega}_{k+1}$ such that 
\begin{equation}\label{eq proof 26.2}
M^k_{\varphi_k^{-1}(i),.}x + b^k_{\varphi_k^{-1}(i)} = 0.
\end{equation}
 Recall from Proposition \ref{object relations for equivalent networks} that 
\[\tilde{\Omega}_{k+1} =  P_{\varphi_{k+1}} \Diag(\lambda^{k+1}) \Omega_{k+1}.\]
 Since $P_{\varphi_{k+1}} \Diag(\lambda^{k+1})$ is an invertible matrix, it induces an homeomorphism on $\mathbb{R}^{n_{k+1}}$, and thus this identity also holds for the interiors:
\[\ring{\tilde{\Omega}}_{k+1} = P_{\varphi_{k+1}} \Diag(\lambda^{k+1}) \ring{\Omega}_{k+1}.\]
 Given that $x \in \ring{\Omega}_{k+1}$, defining $y = P_{\varphi_{k+1}} \Diag(\lambda^{k+1})x$, we have $y \in \ring{\tilde{\Omega}}_{k+1}$.
 
Using \eqref{eq proof 26}, \eqref{permutation matrix} and \eqref{eq proof 26.2}, we have
\begin{equation*}
\begin{aligned}
\tilde{M}^k_{i,.}y + \tilde{b}^k_i & = [P_{\varphi_{k}} \Diag(\lambda^{k} ) M^k\Diag(\lambda^{k+1})^{-1}P_{\varphi_{k+1}}^{-1}]_{i,.}  y + [P_{\varphi_{k}} \Diag(\lambda^{k} )b^k]_i \\
& = [\Diag(\lambda^{k} ) M^k\Diag(\lambda^{k+1})^{-1}P_{\varphi_{k+1}}^{-1}]_{\varphi_k^{-1}(i),.}  y + [ \Diag(\lambda^{k} )b^k]_{\varphi_k^{-1}(i)} \\
& = \lambda^k_{\varphi_k^{-1}(i)}M^k_{\varphi_k^{-1}(i),.}\Diag(\lambda^{k+1})^{-1}P_{\varphi_{k+1}}^{-1}  y +  \lambda^{k}_{\varphi_{k}^{-1}(i)}b^k_{\varphi_{k}^{-1}(i)} \\
& = \lambda^k_{\varphi_k^{-1}(i)}M^k_{\varphi_k^{-1}(i),.}x +  \lambda^{k}_{\varphi_{k}^{-1}(i)}b^k_{\varphi_{k}^{-1}(i)} \\
& = 0.
\end{aligned}
\end{equation*}

We showed that there exists $y \in \ring{\tilde{\Omega}}_{k+1}$ such that
\[\tilde{M}^k_{i,.}y + \tilde{b}^k_i = 0,\]
which concludes the proof of $\Pbf.b)$.

\item[$\mathbf{P}.c)$]Let $\tilde{D} \in \tilde{\Pi}_{k}$ and $i \in \llbracket 1 , n_k \rrbracket$. Suppose $E_i^k \cap \tilde{D} \cap \tilde{h}_k(\tilde{\Omega}_{k+1})\neq \emptyset$, and let us show $\tilde{V}^k_{i,.} (\tilde{D}) \neq 0$.

Let $x \in \tilde{\Omega}_{k+1}$ such that $\tilde{h}_k(x) \in E_i^k \cap \tilde{D}$. Inverting the equalities of Proposition \ref{object relations for equivalent networks} we get
\begin{itemize}
\item[•] $h_{k} =  \Diag(\lambda^{k})^{-1}P_{\varphi_{k}}^{-1} \tilde{h}_{k} \circ P_{\varphi_{k+1}} \Diag(\lambda^{k+1})$,
\item[•] $H^{k+1}_{\varphi_{k}^{-1}(i)} = \Diag(\lambda^{k+1})^{-1} P_{\varphi_{k+1}}^{-1} \tilde{H}_{i}^{k+1}$, 
\item[•] $\Omega_{k+1} =  \Diag(\lambda^{k+1})^{-1} P_{\varphi_{k+1}}^{-1} \tilde{\Omega}_{k+1}$.
\end{itemize}

Denote $D = \Diag(\lambda^{k})^{-1} P_{\varphi_{k}}^{-1} \tilde{D}$. Since $\tilde{\Pi}_k$ has been defined as in Item \ref{relations between polyhedra} of Proposition \ref{object relations for equivalent networks}, we know that $D \in \Pi_k$. Let $y = \Diag(\lambda^{k+1})^{-1} P_{\varphi_{k+1}}^{-1}x$. Let us prove that $h_k(y) \in E_{\varphi_k(i)^{-1}}^{k} \cap D \cap h_k(\Omega_{k+1})$. 

Since $x \in \tilde{\Omega}_{k+1}$, we see that $y \in \Omega_{k+1}$, so $h_k(y) \in h_k(\Omega_{k+1})$.

We also have
\begin{equation*}
\begin{aligned}
h_k(y) & = \Diag(\lambda^{k})^{-1}P_{\varphi_{k}}^{-1} \tilde{h}_{k} \circ P_{\varphi_{k+1}} \Diag(\lambda^{k+1})\left( \Diag(\lambda^{k+1})^{-1} P_{\varphi_{k+1}}^{-1} x \right) \\
& = \Diag(\lambda^{k})^{-1}P_{\varphi_{k}}^{-1} \tilde{h}_{k}\left( x\right),
\end{aligned}
\end{equation*}
which shows, since $\tilde{h}_k(x) \in \tilde{D}$, that $h_k(y) \in D$.

Since, by hypothesis, $\tilde{h}_k(x) \in E_i^k$, using \eqref{permutation matrix} and \eqref{Inverse of a permutation matrix}, we have
\begin{equation*}
\begin{aligned}
\left[h_k(y)\right]_{\varphi_k^{-1}(i)} & = \left[\Diag(\lambda^{k})^{-1}P_{\varphi_{k}}^{-1} \tilde{h}_{k}\left( x\right) \right]_{\varphi_k^{-1}(i)} \\
& = \frac{1}{\lambda^{k}_{\varphi_k^{-1}(i)}} \left[P_{\varphi_{k}}^{-1} \tilde{h}_{k}(x)\right]_{\varphi_k^{-1}(i)} \\
&  = \frac{1}{\lambda^{k}_{\varphi_k^{-1}(i)}} (\tilde{h}_k(x))_i \\
& = 0.
\end{aligned}
\end{equation*}
This proves that $h_k(y) \in E_{\varphi_k^{-1}(i)}^k$.

We proved that 
\[h_k(y) \in E_{\varphi_k^{-1}(i)}^k \cap D \cap h_k(\Omega_{k+1}),\] which shows this intersection is not empty. Since $(g_k, M^{k}, b^{k}, \Omega_{k+1}, \Pi_k)$ satisfies $\mathbf{P}.c)$, we have $V_{.,\varphi_k^{-1}(i)}^k(D) \neq 0$.

Since, according to proposition \ref{object relations for equivalent networks},
\[\tilde{g}_k = g_k \circ \Diag(\lambda^{k})^{-1} P_{\varphi_{k}}^{-1},\]
we deduce:
\begin{equation}\label{eq proof 26.3}
\tilde{V}^k(\tilde{D}) = V^k(D) \Diag(\lambda^k)^{-1} P_{\varphi_k}^{-1}.
\end{equation}
For a matrix $A$ and a permutation $\varphi$, we have $[P_{\varphi} A ]_{i,.} = A_{\varphi^{-1}(i),.}$, so by taking the transpose, we see that $[A^T P_{\varphi}^{-1}]_{.,i} = (A^T)_{.,\varphi^{-1}(i)}$.

Taking the $i^{\text{th}}$ column of \eqref{eq proof 26.3}, we thus obtain
\[\tilde{V}_{.,i}^k (\tilde{D}) = \left[V^k(D) \Diag(\lambda^k)^{-1} P_{\varphi_k}^{-1}\right]_{.,i} = \frac{1}{\lambda^{k}_{\varphi_k^{-1}(i)}}V_{.,\varphi_k^{-1}(i)}^k(D),\]
which shows that $\tilde{V}_{.,i}^k(\tilde{D}) \neq 0$.

\item[$\mathbf{P}.d)$] Let $\tilde{H} \subset \mathbb{R}^{n_{k+1}}$ be an affine hyperplane. Denote $H =  \Diag(\lambda^{k+1})^{-1} P_{\varphi_{k+1}}^{-1}  \tilde{H}$. Since $\mathbf{P}.d)$ holds for $(g_k, M^{k}, b^{k}, \Omega_{k+1}, \Pi_k)$, using Item \ref{relation between f g omega and tilde(f) tilde(g) tilde(omega)} of Proposition \ref{object relations for equivalent networks}, we have
\begin{eqnarray}\label{tilde(H) cap tilde(Omega) not included in Bigcup PDh^(-1)(D)}
\tilde{H} \cap \ring{\tilde{\Omega}}_{k+1} & = & P_{\varphi_{k+1}} \Diag(\lambda^{k+1})  \left(H \cap \ring{\Omega}_{k+1}\right) \nonumber \\& \not\subset & P_{\varphi_{k+1}} \Diag(\lambda^{k+1})  \bigcup_{D \in \Pi_k} \partial h_{k}^{-1}(D) \nonumber \\
& = &  \bigcup_{D \in \Pi_k} P_{\varphi_{k+1}} \Diag(\lambda^{k+1})  \partial h_{k}^{-1}(D).
\end{eqnarray}
For all $k$, $P_{\varphi_{k+1}} \Diag(\lambda^{k+1})$ is an invertible matrix, so it induces an homeomorphism of $\mathbb{R}^{n_{k+1}}$. We thus have
\begin{equation}\label{boundary preserved by homeomorphism}
 P_{\varphi_{k+1}} \Diag(\lambda^{k+1})  \partial h_{k}^{-1}(D) = \partial \left( P_{\varphi_{k+1}} \Diag(\lambda^{k+1})  h_{k}^{-1}(D)\right).
 \end{equation}
Furthermore, by Item \ref{relation between h and tilde(h)} of Proposition \ref{object relations for equivalent networks}, we have  $\tilde{h}_k = P_{\varphi_k} \Diag(\lambda^k) h_k \circ \Diag(\lambda^{k+1})^{-1} P_{\varphi_{k+1}}^{-1}$, so 
\[\tilde{h}_k^{-1} =   P_{\varphi_{k+1}} \Diag(\lambda^{k+1}) h_k^{-1} \circ  \Diag(\lambda^k)^{-1} P_{\varphi_k}^{-1}, \]
and since $\tilde{D} = P_{\varphi_k} \Diag(\lambda^k) D$,
\begin{equation}\label{equality h_k^-1(D)}
\tilde{h}_k^{-1} (\tilde{D}) = P_{\varphi_{k+1}} \Diag(\lambda^{k+1}) h_k^{-1}(D). 
\end{equation}
Combining \eqref{boundary preserved by homeomorphism} and \eqref{equality h_k^-1(D)}, we obtain
\[P_{\varphi_{k+1}} \Diag(\lambda^{k+1})  \partial h_{k}^{-1}(D) = \partial \tilde{h}_k^{-1}(\tilde{D}),\]
and we can thus reformulate \eqref{tilde(H) cap tilde(Omega) not included in Bigcup PDh^(-1)(D)} as
\begin{eqnarray} \tilde{H} \cap \ring{\tilde{\Omega}}_{k+1}
& \not\subset & \bigcup_{\tilde{D} \in \tilde{\Pi}_k} \partial \tilde{h}_{k}^{-1}(\tilde{D}). \nonumber 
 \end{eqnarray}

\end{itemize}
\end{proof}

\subsection{Identifiability statement}\label{main theorem-appendix sec}

We restate here the main theorem, already stated as Theorem \ref{main theorem-main} in the main part of the article.

\begin{thm}\label{main theorem}
Let $K \in \mathbb{N}$, $K \geq 2$. Suppose we are given two networks with $K$ layers, identical number of neurons per layer, and with respective parameters $(\mathbf{M}, \mathbf{b})$ and $(\mathbf{\tilde{M}}, \mathbf{\tilde{b}})$. Assume $\boldsymbol{\Pi}$ and $\boldsymbol{\tilde{\Pi}}$ are two lists of sets of closed polyhedra that are admissible with respect to $(\mathbf{M}, \mathbf{b})$ and $(\mathbf{\tilde{M}}, \mathbf{\tilde{b}})$ respectively. Denote by $n_K$ the number of neurons of the input layer, and suppose we are given a set $\Omega \subset \mathbb{R}^{n_K}$ such that $(\mathbf{M}, \mathbf{b},\Omega,\boldsymbol{\Pi})$ and $(\mathbf{\tilde{M}}, \mathbf{\tilde{b}},\Omega, \boldsymbol{\tilde{\Pi}})$ satisfy the conditions $\mathbf{\mathbf{P}}$, and such that, for all $x \in \Omega$:
 \[f_{\mathbf{M},\mathbf{b}}(x) = f_{\tilde{\mathbf{M}}, \tilde{\mathbf{b}}}(x).\]
 
Then:
\[(\mathbf{M},\mathbf{b}) \sim (\mathbf{\tilde{M}}, \mathbf{\tilde{b}}).\]
\end{thm}

\subsection{An application to risk minimization}\label{Risk minimization-sec}

We restate here the consequence of the main result in terms of minimization of the population risk, already stated as Corollary \ref{risk minimization formulation-main} in the main part.

Assume we are given a couple of input-output variables $(X,Y)$ generated by a ground truth network with parameters $(\mathbf{M}, \mathbf{b})$:
\[Y = f_{\mathbf{M,b}}(X).\]

We can use Theorem \ref{main theorem} to show that the only way to bring the population risk to $0$ is to find the ground truth parameters -modulo permutation and positive rescaling. 

Indeed, let $\Omega \subset \mathbb{R}^{n_K}$ be a domain that is contained in the support of $X$, and suppose $L : \mathbb{R}^{n_0} \times \mathbb{R}^{n_0} \rightarrow \mathbb{R}_+$ is a loss function such that $L(y,y') = 0 \Rightarrow y = y'$.
Consider the population risk:
\[R(\mathbf{\tilde{M}}, \mathbf{\tilde{b}}) = \mathbb{E}[L(f_{\mathbf{\tilde{M}},\mathbf{\tilde{b}}}(X), Y)].\]
We have the following result.

\begin{coro}\label{risk minimization formulation}
Suppose there exists a list of sets of closed polyhedra $\boldsymbol{\Pi}$ admissible with respect to $(\mathbf{M}, \mathbf{b})$ such that $(\mathbf{M}, \mathbf{b},\Omega, \boldsymbol{\Pi})$ satisfies the conditions $\mathbf{P}$. 

If $(\mathbf{\tilde{M}}, \mathbf{\tilde{b}}) $ is also such that there exists a list of sets of closed polyhedra $\boldsymbol{\tilde{\Pi}}$ admissible with respect to $(\mathbf{\tilde{M}}, \mathbf{\tilde{b}})$ such that $(\mathbf{\tilde{M}}, \mathbf{\tilde{b}},\Omega, \boldsymbol{\tilde{\Pi}})$ satisfies the conditions $\mathbf{P}$, and if $(\mathbf{M}, \mathbf{b}) \not\sim (\mathbf{\tilde{M}}, \mathbf{\tilde{b}})$, then:
\[R(\mathbf{\tilde{M}}, \mathbf{\tilde{b}}) > 0.\]
\end{coro}

\subsection{Proof of Theorem \ref{main theorem}}\label{Proof of the main theorem-apdx}

To prove Theorem \ref{main theorem}, we can assume the parameterizations $(\mathbf{M}, \mathbf{b})$ and $(\mathbf{\tilde{M}}, \mathbf{\tilde{b}})$ are normalized. Indeed, if they are not, by Proposition \ref{equivalent normalized network} there exist a normalized parameterization $(\mathbf{M'}, \mathbf{b'})$ equivalent to $(\mathbf{M}, \mathbf{b})$ and a normalized parameterization $(\mathbf{\tilde{M}'}, \mathbf{\tilde{b}'})$ equivalent to $(\mathbf{\tilde{M}}, \mathbf{\tilde{b}})$. Note that we can apply Proposition \ref{equivalent normalized network} because $M^{k}$ and $\tilde{M}^{k}$ are full row rank (condition $\mathbf{P}.a)$) for all $k \in \llbracket 1 , K-1 \rrbracket$ so their lines are always nonzero.
We derive $\boldsymbol{\Pi}'$ from $\boldsymbol{\Pi}$ and $\boldsymbol{\tilde{\Pi}}'$ from $\boldsymbol{\tilde{\Pi}}$ as in Item \ref{relations between polyhedra} of Proposition \ref{object relations for equivalent networks}. By Proposition \ref{conditions P are stable by equivalence}, $(\mathbf{M'}, \mathbf{b'},\Omega, \boldsymbol{\Pi}')$ and $(\mathbf{\tilde{M}'}, \mathbf{\tilde{b}'}, \Omega, \boldsymbol{\tilde{\Pi}}')$ also satisfy the conditions $\mathbf{P}$. By Corollary \ref{equivalent networks have same function}, $f_{\mathbf{M'}, \mathbf{b'}} = f_{\mathbf{M}, \mathbf{b}}$ and $f_{\mathbf{\tilde{M}'}, \mathbf{\tilde{b}'}} = f_{\mathbf{\tilde{M}}, \mathbf{\tilde{b}}}$, so we have, for all $x \in \Omega$:
 \[f_{\mathbf{M'}, \mathbf{b'}}(x) = f_{\mathbf{\tilde{M}'}, \mathbf{\tilde{b}'}}(x).\]
 
$(\mathbf{M'}, \mathbf{b'}, \Omega, \boldsymbol{\Pi'})$ and $(\mathbf{\tilde{M}'}, \mathbf{\tilde{b}'}, \Omega, \boldsymbol{\tilde{\Pi}'})$ satisfy the hypotheses of Theorem \ref{main theorem}. If we are able to show that $(\mathbf{M'}, \mathbf{b'}) \sim (\mathbf{\tilde{M}'}, \mathbf{\tilde{b}'})$, then $(\mathbf{M}, \mathbf{b}) \sim (\mathbf{\tilde{M}}, \mathbf{\tilde{b}})$ follows immediately from the transitivity of the equivalence relation, proven in Proposition \ref{the relation is an equivalence relation}.

Thus in the proof $(\mathbf{M}, \mathbf{b})$ and $(\mathbf{\tilde{M}}, \mathbf{\tilde{b}})$ will be assumed to be normalized.

\vspace{15pt}

To prove the theorem, we need the following fundamental lemma (already stated as Lemma \ref{Fundamental lemma-main} in the main text), that is proven in Appendix \ref{proof of the lemma}.

\begin{lem}\label{Fundamental lemma} Let $l,m,n \in \mathbb{N}^*$. Suppose $g,\tilde{g} : \mathbb{R}^m \rightarrow \mathbb{R}^n$ are continuous piecewise linear functions, $\Omega \subset \mathbb{R}^l$ is a subset and let $M, \tilde{M} \in \mathbb{R}^{m \times l}$, $b, \tilde{b} \in \mathbb{R}^{m}$. Denote $h: x \mapsto \sigma(Mx + b)$ and $\tilde{h}: x \mapsto \sigma(\tilde{M}x + \tilde{b})$. Assume $\Pi$ and $\tilde{\Pi}$ are two sets of polyhedra admissible with respect to $g$ and $\tilde{g}$ respectively as in Definition \ref{def admissible sets of polyhedra}.

Suppose  $(g,M,b,\Omega, \Pi)$ and $(\tilde{g}, \tilde{M}, \tilde{b}, \Omega, \tilde{\Pi})$ satisfy the conditions $\mathbf{C}$, and for all $i \in \llbracket 1 , m \rrbracket$, $\|M_{i,.} \| = \| \tilde{M}_{i,.} \|= 1$.

Suppose for all $x \in \Omega$:
\[g \circ h (x) = \tilde{g} \circ \tilde{h}(x).\]

Then, there exists a permutation $\varphi \in \mathfrak{S}_m$, such that:
\begin{itemize}
\item $\tilde{M} = P_{\varphi}M $;
\item $\tilde{b} = P_{\varphi}b$;
\item $g$ and $\tilde{g} \circ P_{\varphi}$ coincide on ${h}(\Omega)$.
\end{itemize}
\end{lem}

\begin{proof}[Proof of Theorem \ref{main theorem}]

We prove the theorem by induction on $K$. 

{\bf Initialization.} Assume here $K=2$. We are going to apply Lemma \ref{Fundamental lemma}. Since $(\mathbf{M}, \mathbf{b},\Omega,\boldsymbol{\Pi})$ and $(\mathbf{\tilde{M}}, \mathbf{\tilde{b}},\Omega, \boldsymbol{\tilde{\Pi}})$ satisfy the conditions $\mathbf{\mathbf{P}}$, by definition, $(g_1, M^{1}, b^{1}, \Omega_{2}, \Pi_1)$ and $(\tilde{g}_1, \tilde{M}^{1}, \tilde{b}^{1}, \Omega_{2}, \tilde{\Pi}_1)$ satisfy the conditions $\mathbf{C}$ (note that $\tilde{\Omega}_2 = \Omega_2 = \Omega$). The network is normalized, so we have, for all $i \in \llbracket 1, n_1 \rrbracket$, 
\[\|M_{i,.}^{1}\| = \| \tilde{M}_{i,.}^{1} \| = 1.\]
By the assumptions of Theorem \ref{main theorem}, for all $x \in \Omega$,
\[g_1 \circ h_1 (x) = f_{\mathbf{M,b}} (x) = f_{\mathbf{\tilde{M}, \tilde{b}}}(x) = \tilde{g}_1 \circ \tilde{h}_1(x).\]
We can thus apply Lemma \ref{Fundamental lemma}, which shows that there exists a permutation $\varphi \in \mathfrak{S}_{n_1}$ such that 
\begin{itemize}
\item $\tilde{M}^1 = P_{\varphi}M^1 $;
\item $\tilde{b}^1 = P_{\varphi}b^1$;
\item $g_1$ and $\tilde{g}_1 \circ P_{\varphi}$ coincide on ${h_1}(\Omega)$.
\end{itemize}

Recall from Definition \ref{def hyperplanes} that for all $i \in \llbracket 1, n_1 \rrbracket$, we denote \mbox{$H^2_i = \{ x \in \mathbb{R}^{n_2} , \ M^1_{i,.} x + b^1_i = 0\}.$} Let $(v_1, \dots, v_{n_1})$ be the canonical basis of $\mathbb{R}^{n_1}$. Let us show that for all $i \in \llbracket 1, n_1 \rrbracket$,
\[M^0  v_i  = \tilde{M}^0 P_{\varphi} v_i.\]
Let $i \in \llbracket 1, n_1 \rrbracket$. By $\Pbf.b)$, $H^2_i \cap \ring{\Omega} \neq \emptyset$. Since $M^1$ is full row rank by $\Pbf.a)$, none of the hyperplanes $H^2_j$, with $j \neq i$, is parallel to $H^2_i$. As a consequence, the intersections $H^2_i \cap H^2_j$ have Hausdorff dimension smaller than $n_2 - 2$, so there exists $x \in \ring{\Omega} \cap H^2_i \backslash \left(\bigcup_{j\neq i} H^2_j \right)$, and $\epsilon > 0$ such that $B(x,\epsilon) \cap H^2_j = \emptyset$ for all $j \neq i$. Let $u$ be a unit vector such that $M^1_{j,.}u = 0$ for all $j \neq i$ and $M^1_{i,.}u = \alpha > 0$ (this is possible again since $M^1$ is full row rank).

For all $j \in \llbracket 1, n_1 \rrbracket \backslash \{ i \}$, we have
\[\sigma(M^1_{j,.} (x + \epsilon u) + b^1_j) - \sigma(M^1_{j,.} x  + b^1_j) = \sigma(M^1_{j,.}x + b^1_j) - \sigma(M^1_{j,.} x  + b^1_j) = 0.\]
At the same time, we have
\begin{equation*}
\begin{aligned}
\sigma(M^1_{i,.} (x + \epsilon u) + b^1_i) - \sigma(M^1_{i,.} x  + b^1_i) & = M^1_{i,.} (x + \epsilon u) + b^1_i - M^1_{i,.} x  + b^1_i \\
& = \epsilon M^1_{i,.} u  \\
& = \epsilon \alpha.
\end{aligned}
\end{equation*}
Summarizing, 
\begin{equation*}
\begin{aligned}
h_1(x + \epsilon u)-h_1(x) &= \sigma(M^1(x + \epsilon u) + b^1) - \sigma(M^1x + b^1) \\
& = \epsilon \alpha v_i.
\end{aligned}
\end{equation*}

Let us denote $y_2 = h_1 ( x + \epsilon u) \in h_1 (\Omega)$ and $y_1 = h_1 (x) \in h_1 ( \Omega)$. We have shown $y_2 - y_1 = \epsilon \alpha v_i$, and since $g_1$ and $\tilde{g}_1 \circ P_{\varphi}$ coincide on $h_1 ( \Omega)$, we have
\begin{equation*}
\begin{aligned}
g_1 ( y_2) - g_1 (y_1) & = \tilde{g}_1 \circ P_{\varphi} ( y_2) - \tilde{g}_1 \circ P_{\varphi}(y_1) \\
\Longleftrightarrow \ M^0(y_2 - y_1)& =  \tilde{M}^0 P_{\varphi} (y_2 - y_1) \\
 \Longleftrightarrow \ \epsilon \alpha M^0 v_i  & = \epsilon \alpha \tilde{M}^0 P_{\varphi} v_i \\
\Longleftrightarrow \ M^0 v_i & = \tilde{M}^0 P_{\varphi} v_i.
\end{aligned}
\end{equation*}

Since this last equality holds for any $i \in \llbracket 1, n_1 \rrbracket$, we conclude that 
\[ M^0  = \tilde{M}^0 P_{\varphi}, \]
and using one last time that $g_1$ and $\tilde{g}_1 \circ P_{\varphi}$ coincide on $h_1(\Omega)$, we obtain 
\[b^0 = \tilde{b}^0,\]
i.e. we have shown
\[\begin{cases} \tilde{M}^0  = M^0 P_{\varphi}^{-1}\\
\tilde{b}^0 = b^0.
\end{cases}\]

Defining $P_{\varphi_1} = P_{\varphi}$, $P_{\varphi_0} = \Id_{n_0}$ and $P_{\varphi_2} = \Id_{n_2}$, we can use Proposition \ref{permuted networks} to conclude that 
\[\mathbf{(M,b)} \sim \mathbf{(\tilde{M}, \tilde{b})}.\]

{\bf Induction step.}
Let $K \geq 3
$ be an integer. Suppose Theorem \ref{main theorem} is true for all networks with $K-1$ layers.

Consider two networks with parameters $(\mathbf{M}, \mathbf{b})$ and $(\mathbf{\tilde{M}}, \mathbf{\tilde{b}})$, with $K$ layers and, for all $k \in \llbracket 0, K \rrbracket$, same number $n_k$ of neurons per layer. Let $\boldsymbol{\Pi}$ and $\boldsymbol{\tilde{\Pi}}$ be two list of sets of closed polyhedra that are admissible with respect to 
$\mathbf{(M,b)}$ and $\mathbf{(\tilde{M}, \tilde{b})}$ respectively (Definition \ref{def of an admissible list Pi}), and let $\Omega \subset \mathbb{R}^{n_K}$ such that $(\mathbf{M}, \mathbf{b}, \Omega, \boldsymbol{\Pi})$ and $(\mathbf{\tilde{M}}, \mathbf{\tilde{b}}, \Omega, \boldsymbol{\tilde{\Pi}})$ satisfy the conditions $\mathbf{P}$ and $f_{\mathbf{M}, \mathbf{b}}$ and $f_{\mathbf{\tilde{M}}, \mathbf{\tilde{b}}}$ coincide on $\Omega$.

\vspace{10pt}

Recall the functions $h_k$ and $g_k$ associated to $(\mathbf{M}, \mathbf{b})$, defined in Definition \ref{functions h} and Definition \ref{functions g} respectively, and the corresponding functions $\tilde{h}_k$ and $\tilde{g}_k$ associated to $(\mathbf{\tilde{M}}, \mathbf{\tilde{b}})$.

We have two matrices $M^{K-1}$ and $\tilde{M}^{K-1} \in \mathbb{R}^{n_{K-1} \times n_K}$, two vectors $b^{K-1}$ and $\tilde{b}^{K-1} \in \mathbb{R}^{n_{K-1}}$, two functions $g_{K-1}$ and $\tilde{g}_{K-1}:\mathbb{R}^{n_{K-1}} \rightarrow \mathbb{R}^{n_0}$, two sets $\Pi_{K-1}$ and $\tilde{\Pi}_{K-1}$ such that:
\begin{itemize}
\item $\forall x \in \Omega$, \quad $g_{K-1}\circ h_{K-1}(x) = g_K(x) = f_{\Mbf,\bbf} (x) = f_{\tilde{\Mbf}, \tilde{\bbf}} (x) = \tilde{g}_K(x) = \tilde{g}_{K-1} \circ \tilde{h}_{K-1}(x)$,
\item $g_{K-1}$ and $\tilde{g}_{K-1}$ are continuous piecewise linear, and $\Pi_{K-1}$ and $\tilde{\Pi}_{K-1}$ are admissible with respect to $g_{K-1}$ and $\tilde{g}_{K-1}$ respectively,
\item $(g_{K-1},M^{K-1}, b^{K-1}, \Omega, \Pi_{K-1})$ and $(\tilde{g}_{K-1}, \tilde{M}^{K-1}, \tilde{b}^{K-1}, \Omega, \tilde{\Pi}_{K-1})$ satisfy the conditions $\mathbf{C}$,
\item $\forall i \in \llbracket 1, n_{K-1} \rrbracket$, \quad $\| M^{K-1}_{i,.} \| = \| \tilde{M}^{K-1}_{i,.} \| = 1$.
\end{itemize}

The third point comes from the fact that the conditions $\mathbf{P}$ hold for $(\mathbf{M}, \mathbf{b},\Omega,\boldsymbol{\Pi})$ and $(\mathbf{\tilde{M}}, \mathbf{\tilde{b}},\Omega, {\boldsymbol{\tilde{\Pi}}})$, and the fourth point comes from the fact that $(\Mbf,\bbf)$ and $(\tilde{\Mbf}, \tilde{\bbf})$ are normalized.

Thus, the objects $g_{K-1}, \tilde{g}_{K-1}, M^{K-1}, b^{K-1}, \tilde{M}^{K-1}, \tilde{b}^{K-1}, \Pi_{K-1}$ and $\tilde{\Pi}_{K-1}$ satisfy the hypotheses of Lemma \ref{Fundamental lemma} and hence there exists $\varphi \in \mathfrak{S}_{n_{K-1}}$ such that 
\begin{equation}\label{tilde(M)^k-1 = P M^k-1}
\begin{cases}
\tilde{M}^{K-1} = P_{\varphi} M^{K-1}, \\
\tilde{b}^{K-1} = P_{\varphi} b^{K-1},
\end{cases}
\end{equation} and $g_{K-1}$ and $ \tilde{g}_{K-1} \circ P_{\varphi}$ coincide on $\Omega_{K-1}$.

\vspace{10pt}

Let us denote $\mathbf{M^*} = (M^{0}, \dots, M^{K-3}, M^{K-2}P_{\varphi}^{-1})$.
The functions $g_{K-1} \circ P_{\varphi}^{-1}$ and $\tilde{g}_{K-1}$ are implemented by two networks with $K-1$ layers, indexed from $K-1$ up to $0$, with parameters $(\mathbf{M^*},\mathbf{b}^{\leq K-2})$ and $(\mathbf{\tilde{M}}^{\leq K-2},\mathbf{\tilde{b}}^{\leq K-2})$ respectively. The previous paragraph shows these functions coincide on $P_{\varphi} \Omega_{K-1}$.
Recalling the definition of $\tilde{\Omega}_{K-1}$ and since, by \eqref{tilde(M)^k-1 = P M^k-1}, $\tilde{f}_{K-1} = \tilde{h}_{K-1} = P_{\varphi} h_{K-1}$, we have 
\[ \tilde{\Omega}_{K-1} = \tilde{f}_{K-1}(\Omega) = P_{\varphi} h_{K-1}(\Omega)  = P_{\varphi} \Omega_{K-1},\] 
i.e. the functions $g_{K-1} \circ P_{\varphi}^{-1} = f_{\mathbf{M^*,b^{\leq K-2}}} $ and $\tilde{g}_{K-1}= f_{\mathbf{\tilde{M}^{\leq K-2}, \tilde{b}^{\leq K-2}}}$ coincide on $\tilde{\Omega}_{K-1}$.

Since $(\mathbf{M},\mathbf{b},\Omega, \boldsymbol{\Pi})$ and $(\mathbf{\tilde{M}},\mathbf{\tilde{b}}, \Omega, \boldsymbol{\tilde{\Pi}})$ satisfy the conditions $\mathbf{P}$, $(g_k, M^{k}, b^{k}, \Omega_{k+1}, \Pi_k)$ and $(\tilde{g}_k, \tilde{M}^{k}, \tilde{b}^{k}, \tilde{\Omega}_{k+1}, \tilde{\Pi}_k)$ satisfy the conditions $\mathbf{C}$ for all $k \in \llbracket 1,  K-1 \rrbracket$ so in particular these conditions are satisfied for $k \in \llbracket 1,  K-2 \rrbracket$, so $(\mathbf{M}^{\leq K-2},\mathbf{b}^{\leq K-2},\Omega_{K-1}, \boldsymbol{\Pi}^{\leq K-2})$ and $(\mathbf{\tilde{M}}^{\leq K-2},\mathbf{\tilde{b}}^{\leq K-2}, \tilde{\Omega}_{K-1}, \boldsymbol{\tilde{\Pi}}^{\leq K-2 )})$ satisfy the conditions $\mathbf{P}$. 

Let us verify that $(\mathbf{M^*},\mathbf{b}^{\leq K-2}, \tilde{\Omega}_{K-1}, \boldsymbol{\Pi}^{\leq K-2})$ also satisfies the conditions $\mathbf{P}$.  Indeed, the only thing that differs from $(\mathbf{M}^{\leq K-2},\mathbf{b}^{\leq K-2},\Omega_{K-1}, \boldsymbol{\Pi}^{\leq K-2})$ is $\tilde{\Omega}_{K-1}$ and the weights $ M^{*K-2}$ between the layer $K-1$ and the layer $K-2$. Writing that $M^{*K-2} = M^{K-2} P_{\varphi}^{-1}$, $h^*_{K-2} = h_{K-2} \circ P_{\varphi}^{-1}$, $\tilde{\Omega}_{K-1} = P_{\varphi} \Omega_{K-1}$ and $H^{*K-1}_i = P_{\varphi} H_i^{K-1}$, let us check that the conditions $\mathbf{C}$ also hold for $ (g_{K-2}, M^{*K-2}, b^{K-2}, \tilde{\Omega}_{K-1}, \Pi_{K-2})$. 

Indeed $P_{\varphi}^{-1}$ is invertible, so $M^{*K-2}$ is full row rank and $\mathbf{C}.a)$ holds.

 If $x \in \ring{\Omega}$ satisfies $M^{K-2}_{i,.}x + b^{K-2}_i = 0$, we define $h_{K-2}^{*lin}(x) = M^{*K-2}x + b^{K-2}$, we have $h_{K-2}^{*lin} = h_{K-2}^{lin} \circ P_{\varphi}^{-1}$, so 
\[E_i \cap h_{K-2}^{*lin}(\ring{\tilde{\Omega}}_{K-1}) = E_i \cap h_{K-2}^{lin}(\ring{\Omega}_{K-1}) \neq \emptyset,\]
and $\mathbf{C}.b)$ is satisfied. 

Similarly, the observation $h_{K-2}^*(\tilde{\Omega}_{K-1}) = h_{K-2}(\Omega_{K-1})$ yields $\mathbf{C}.c)$.

Finally, assume $H^* \subset \mathbb{R}^{n_{K-1}}$ is an affine hyperplane. Let $H = P_{\varphi}^{-1} H^*$. We have by hypothesis 
\[H \cap \ring{\Omega}_{K-1} \not\subset \bigcup_{D \in \Pi_{K-2}} \partial h_{K-2}^{-1} (D),\]
thus
\begin{equation*}
\begin{aligned}
H^* \cap \ring{\tilde{\Omega}}_{K-1} & = P_{\varphi} \left( H \cap \ring{\Omega}_{K-1} \right)\\
& \not\subset  P_{\varphi}  \bigcup_{D \in \Pi_{K-2}} \partial h_{K-2}^{-1} (D)\\
& =   \bigcup_{D \in \Pi_{K-2}} \partial (P_{\varphi} h_{K-2}^{-1} (D)).
\end{aligned}
\end{equation*}
For all $D \in \Pi_{K-2}$ we have
\begin{equation*}
\begin{aligned}
P_{\varphi} h_{K-2}^{-1} (D) & =  P_{\varphi}\{ y , \ h_{K-2}(y) \in D \}\\
& = P_{\varphi}\{ P_{\varphi}^{-1} x , \ h_{K-2} \circ P_{\varphi}^{-1}(x) \in D \}\\
& = \{ x, \ h_{K-2}^*(x) \in D \}\\
& = h_{K-2}^{*-1} (D).
\end{aligned}
\end{equation*}
Therefore,
\begin{equation*}
\begin{aligned}
H^* \cap \ring{\tilde{\Omega}}_{K-1} & =   \bigcup_{D \in \Pi_{K-2}} \partial h_{K-2}^{*-1} (D),
\end{aligned}
\end{equation*}
which proves $\mathbf{C}.d)$.

Since the rest stays unchanged, we can conclude.

The induction hypothesis can thus be applied to $(\mathbf{M^*},\mathbf{b}^{\leq K-2}, \tilde{\Omega}_{K-1}, \boldsymbol{\Pi}^{\leq K-2})$ and $(\mathbf{\tilde{M}}^{\leq K-2},\mathbf{\tilde{b}}^{\leq K-2}, \tilde{\Omega}_{K-1} , \boldsymbol{\tilde{\Pi}}^{\leq K-2})$, to obtain:
\[(\mathbf{M^*},\mathbf{b}^{\leq K-2}) \sim (\mathbf{\tilde{M}}^{\leq K-2},\mathbf{\tilde{b}}^{\leq K-2}).\] 

Since we also have
\[\forall k \in \llbracket 1, K-3 \rrbracket, \ \forall i \in \llbracket 1 , n_k \rrbracket, \qquad \|M^{*k}_{i,.}\|=\|M_{i,.}^{k}\| = 1 \quad \text{and} \quad \|\tilde{M}_{i,.}^{k}\|=1,\] 
\[ \forall i \in \llbracket 1 , n_{K-2} \rrbracket, \quad \|M^{*K-2}_{i,.}\| = \|M_{i,.}^{K-2}P_{\varphi}^{-1}\|=\|M_{i,.}^{K-2}\| = 1 \quad \text{and} \quad \|\tilde{M}_{i,.}^{K-2}\| =1,\]
Proposition \ref{permuted networks} shows that there exists a family of permutations $(\varphi_0, \dots , \varphi_{K-1}) \in \mathfrak{S}_{n_0} \times \dots \times \mathfrak{S}_{n_{K-1}}$, with $\varphi_0 = id_{\llbracket 1 , n_0 \rrbracket}$ and $\varphi_{K-1} = id_{\llbracket 1 , n_{K-1} \rrbracket}$, such that:

\begin{equation}\label{eq heredite-1} \forall k \in \llbracket 0 , K-3 \rrbracket, \quad \begin{cases}\tilde{M}^{k}=  P_{\varphi_{k}} M^{*k}P_{\varphi_{k+1}}^{-1}  = P_{\varphi_{k}} M^{k}P_{\varphi_{k+1}}^{-1} \\
\tilde{b}^{k} = P_{\varphi_k} b^{k},\end{cases} 
\end{equation}
and:
\begin{equation}\label{eq heredite-2}\begin{cases}\tilde{M}^{K-2}=  P_{\varphi_{K-2}} M^{*K-2}P_{\varphi_{K-1}}^{-1} = P_{\varphi_{K-2}} (M^{K-2} P_{\varphi}^{-1})P_{\varphi_{K-1}}^{-1} = P_{\varphi_{K-2}} M^{K-2} P_{\varphi}^{-1}\\
\tilde{b}^{K-2} = P_{\varphi_{K-2}} b^{K-2}.\end{cases}
\end{equation}

We can define $(\psi_0, \dots, \psi_K) \in \mathfrak{S}_{n_0} \times \dots \times \mathfrak{S}_{n_K} $ by:
 \begin{itemize}
 \item $\psi_0 = id_{\llbracket 1, n_0 \rrbracket}$, $\psi_K = id_{\llbracket 1, n_K \rrbracket}$;
 \item $\forall k \in \llbracket 1 , K-2 \rrbracket$, $\psi_k = \varphi_k$;
 \item $\psi_{K-1} = \varphi$;
 \end{itemize}
 
and using \eqref{eq heredite-1}, \eqref{eq heredite-2} and \eqref{tilde(M)^k-1 = P M^k-1} altogether, we then have, for all $k \in \llbracket 0, K-1 \rrbracket$:
 \begin{equation*}\begin{cases}\tilde{M}^{k}= P_{\psi_{k}} M^{k}P_{\psi_{k+1}}^{-1} \\
\tilde{b}^{k} = P_{\psi_k} b^{k}.\end{cases} \end{equation*}

It follows from Proposition \ref{permuted networks} that $(\mathbf{M},\mathbf{b}) \sim (\mathbf{\tilde{M}}, \mathbf{\tilde{b}})$.

\end{proof}

\subsection{Proof of Corollary \ref{risk minimization formulation}}\label{Proof of the corollary-sec}

Theorem \ref{risk minimization formulation} is an immediate consequence of Theorem \ref{main theorem}.

Since $(\mathbf{M},\mathbf{b}, \Omega, \Pi)$ and $(\mathbf{\tilde{M}}, \mathbf{\tilde{b}}, \Omega, \tilde{\Pi})$ satisfy the conditions $\mathbf{P}$ and $(\mathbf{M},\mathbf{b}) \not\sim (\mathbf{\tilde{M}}, \mathbf{\tilde{b}})$, the contrapositive of Theorem \ref{main theorem} shows that there exists $x \in \Omega$ such that $f_{\mathbf{M},\mathbf{b}}(x) \neq f_{\tilde{\mathbf{M}}, \tilde{\mathbf{b}}}(x)$. 
The function $f_{\mathbf{M},\mathbf{b}} - f_{\tilde{\mathbf{M}}, \tilde{\mathbf{b}}}$ is continuous so there exists $r > 0$ such that for all $u \in B(x,r)$, $f_{\mathbf{M},\mathbf{b}}(u) \neq f_{\tilde{\mathbf{M}}, \tilde{\mathbf{b}}}(u)$ so $ L(f_{\mathbf{M},\mathbf{b}}(u),f_{\tilde{\mathbf{M}}, \tilde{\mathbf{b}}}(u)) > 0 $. Since $\Omega$ is included in the support of $X$ and $x \in \Omega$, denoting $\mathbb{P}_X$ the law of $X$ we have $\mathbb{P}_X(B(x,r)) > 0$ and thus
\begin{equation*}
\begin{aligned}
R(\mathbf{\tilde{M}}, \mathbf{\tilde{b}}) & = \mathbb{E}[L(f_{\mathbf{\tilde{M}},\mathbf{\tilde{b}}}(X), f_{\mathbf{M}, \mathbf{b}}(X))] \\
& \geq \int_{B(x,r)} L(f_{\mathbf{M},\mathbf{b}}(u), f_{\tilde{\mathbf{M}}, \tilde{\mathbf{b}}}(u)) \mathrm{d} \mathbb{P}_X(u) \\
& > 0.
\end{aligned}
\end{equation*}

\section{Proof of Lemma \ref{Fundamental lemma}} \label{proof of the lemma}

In this section we prove Lemma \ref{Fundamental lemma}.

Let $(g,M,b,\Omega, \Pi)$ and $(\tilde{g}, \tilde{M}, \tilde{b}, \Omega, \tilde{\Pi})$ be as in the lemma. In particular, we assume they satisfy the conditions $\mathbf{C}$ all along Appendix \ref{proof of the lemma}.

We denote, for all $x \in \mathbb{R}^l$:
\[f(x) = g\left( \sigma ( M x + b)\right) .\]

Recall that, for all $x \in \mathbb{R}^l$, $h(x) = \sigma(Mx + b)$ and $\tilde{h}(x) = \sigma(\tilde{M} x + \tilde{b})$.

Recall that, as in Definition \ref{def hyperplanes}, we define for all $i \in \llbracket 1 , m \rrbracket$ the sets $H_i = \{x \in \mathbb{R}^l \ , \ M_{i,.}x + b_i = 0 \}$ and $\tilde{H}_i = \{ x \in \mathbb{R}^l \ , \ \tilde{M}_{i,.} x + \tilde{b}_i = 0 \}$. By condition $C.a)$, for all $i \in \llbracket 1 , m \rrbracket$, $M_{i,.} \neq 0$ and $\tilde{M}_{i,.} \neq 0$ so $H_i$ and $\tilde{H}_i$ are hyperplanes.

Recall that for all $D \in \Pi$, we define $V(D) \in \mathbb{R}^{n \times m}$ and $c(D) \in \mathbb{R}^n$ as in Definition \ref{def v_i}, and similarly for all $\tilde{D} \in \tilde{\Pi}$, we define $\tilde{V}(\tilde{D}) \in \mathbb{R}^{n \times m}$ and $\tilde{c}(\tilde{D}) \in \mathbb{R}^n$ associated to $\tilde{g}$.

We now define $s : \mathbb{R}^{l} \rightarrow \{0,1 \}^{m} $ as follows:

\begin{equation}\label{definition s} \forall i \in \llbracket 1 , m \rrbracket, \quad s_i(x) := \begin{cases} 1 & \text{if }M_{i, .} x + b_i \geq 0 \\ 0 & \text{otherwise.}\end{cases}\end{equation}

We define similarly $\tilde{s}$ for $(\tilde{M}, \tilde{b})$. We thus have, for all $i \in \llbracket 1, m \rrbracket$, 
\[\sigma(M_{i,.}x + b_i) = s_i(x) (M_{i,.}x + b_i)\]
and 
\[\sigma(\tilde{M}_{i,.}x + \tilde{b}_i) = \tilde{s}_i(x) (\tilde{M}_{i,.}x + \tilde{b}_i).\]

Let $D \in \Pi$. For all $y \in D$, we have, by definition,
\[g(y)  = V(D) y + c(D),\]
thus, for all $x \in h^{-1}(D)$,
\begin{eqnarray}\label{expression of f with V}
f(x) & = & V(D)h(x)+ c(D) \nonumber \\
& = & V(D) \sigma(Mx + b) + c(D) \nonumber \\
& = & \sum_{k=1}^m V_{.,k}(D) s_k(x) (M_{k,.}x + b_k) +c(D).
\end{eqnarray}
Similarly, for all $\tilde{D} \in \tilde{\Pi}$, for all $x \in \tilde{h}^{-1}(\tilde{D})$,
\begin{equation}\label{expression of f with V-tilde}
f(x) = \sum_{k=1}^m \tilde{V}_{.,k}(\tilde{D}) \tilde{s}_k(x) (\tilde{M}_{k,.}x + \tilde{b}_k) +\tilde{c}(\tilde{D}).
\end{equation}

\begin{prop}\label{non differentiable points}
Let $D \in \Pi$. For all $i \in \llbracket 1, m \rrbracket$, for all $x \in H_i  \cap \overset{\circ}{\overbrace{h^{-1}(D)}} \cap \ring{\Omega} \backslash \left( \bigcup_{k \neq i} H_k \right) $, $f$ is not differentiable at the point $x$.
\end{prop}

\begin{proof}

 Let $i \in \llbracket 1, m \rrbracket$ and suppose $x \in  H_i  \cap \overset{\circ}{\overbrace{h^{-1}(D)}} \cap \ring{\Omega} \backslash \left( \bigcup_{k \neq i} H_k \right)$. 
Let us consider the function $t \mapsto f(x + t M_{i,.}^{ \hspace{5pt} T})$. Since $x \in H_i$ and $\|M_{i,.} \| = 1$ by hypothesis, 
\begin{equation}\label{simplification Mi(x + tMi) + bi}
M_{i, .}(x + t M_{i,.}^{ \hspace{5pt} T}) + b_i = t M_{i, .}M_{i,.}^{ \hspace{5pt} T} + M_{i, .}x + b_i = t \| M_{i, .}\|^2 = t.
\end{equation}
Given the definition of $s$ in \eqref{definition s}, we thus have
\[s_i(x + t M^{\hspace{5pt} T}_{i,.}) = \begin{cases} 1 & \text{if } t \geq 0 \\ 0 & \text{if } t < 0.\end{cases}\]

Since $x \in \overset{\circ}{\overbrace{h^{-1}(D)}}$ which is an open set, for $t$ small enough we have $x + t M_{i,.}^{ \hspace{5pt} T} \in \overset{\circ}{\overbrace{h^{-1}(D)}}$ and thus, using \eqref{expression of f with V} and \eqref{simplification Mi(x + tMi) + bi},
 
\begin{equation*}
\begin{aligned}
f(x + t M_{i,.}^{ \hspace{5pt} T}) & = \sum_{k =1}^m V_{.,k}(D) s_k(x + t M_{i,.}^{ \hspace{5pt} T}) \left( M_{k,.}( x + t M_{i,.}^{ \hspace{5pt} T}) + b_k \right) + c(D)  \\
& = \begin{cases} \sum_{k \neq i} V_{.,k}(D) s_k(x + tM_{i,.}^{ \hspace{5pt} T})\left( M_{k,.}(x + t M_{i,.}^{ \hspace{5pt} T})  + b_k \right) \\ +c(D)  + t V_{.,i}(D) & \text{ if } t \geq 0 \\
\sum_{k \neq i} V_{.,k}(D) s_k(x + t M_{i,.}^{ \hspace{5pt} T}) \left( M_{k,.}( x + t M_{i,.}^{ \hspace{5pt} T}) + b_k \right) \\ + c(D) & \text{ if } t < 0 .\end{cases}  
\end{aligned}
\end{equation*}

Since $x$ does not belong to any of the hyperplanes $H_k$ for $k \neq i$, which are closed, there exists $\epsilon > 0$ such that for all $ t \in \ ] - \epsilon, \epsilon [ $ and for all $k \neq i$, $x + tM^{\hspace{5pt} T}_{i,.} \notin H_k$. Therefore, for all $t \in \ ] - \epsilon, \epsilon [$, for all $k \in \llbracket 1 , m \rrbracket \backslash \{i\}$, $s_k(x + t M_{i,.}^{ \hspace{5pt} T}) = s_k(x)$ and 
 
 \[f(x + t M_{i,.}^{ \hspace{5pt} T}) = \begin{cases} \sum_{k \neq i} V_{.,k}(D) s_k(x)(  M_{k,.} (x + t M_{i,.}^{ \hspace{5pt} T}) + b_k) + c(D) \\ +  t V_{.,i}(D)   & \text{ if } t \geq 0 \\
\sum_{k \neq i} V_{.,k}(D) s_k(x) (  M_{k,.} (x + t M_{i,.}^{ \hspace{5pt} T})  + b_k) + c(D) & \text{ if } t < 0. \end{cases}  \]
 
The right derivative of $t \mapsto f(x + t M_{i,.}^{ \hspace{5pt} T})$ at $0$ is:

\[ \sum_{k \neq i} V_{.,k}(D) s_k(x)  M_{k, .}M_{i,.}^{ \hspace{5pt} T} +  V_{.,i}(D)  .\]

 The left derivative of $t \mapsto f(x + t M_{i,.}^{ \hspace{5pt} T})$ at $0$ is:
 
 \[\sum_{k \neq i} V_{.,k}(D) s_k(x) M_{k, .} M_{i,.}^{ \hspace{5pt} T} . \phantom{ +  V_{.,i}(D) }  \]
 
  Since $x \in H_i \cap h^{-1}(D) \cap \Omega$, we have $h(x) \in E_i \cap D \cap h(\Omega)$ so the condition $\mathbf{C}.c)$ implies that $V_{.,i}(D) \neq 0$. We conclude that the left and right derivatives at $x$ do not coincide and thus $f$ is not differentiable at $x$.
  \end{proof}

\begin{lem}\label{differentiable points} Let $D \in \Pi$. For all $x \in \overset{\circ}{\overbrace{h^{-1}(D)}}  \ \backslash \left( \bigcup_{i=1}^m H_i \right) $, there exists $r>0$ such that $f$ is differentiable on $B(x,r)$.
\end{lem}

\begin{proof}
Consider $x \in \overset{\circ}{\overbrace{h^{-1}(D)}}  \ \backslash \left( \bigcup_{i=1}^m H_i \right) $. Since the hyperplanes $H_i$ are closed, there exists a ball $B(x, r) \subset \overset{\circ}{\overbrace{h^{-1}(D)}}$ such that for all $i \in \llbracket 1 , m \rrbracket$, $B(x,r) \cap H_i = \emptyset$. As a consequence, for all $y \in B(x,r)$, $s(y) = s(x)$. Using \eqref{expression of f with V} we get, for all $y \in B(x,r)$,
\[f(y) = \sum_{i=1}^m V_{.,i}(D) s_i(x) \left( M_{i, .}  y  + b_i \right) + c(D). \]

The right side of this equality is affine in the variable $y$, so $f$ is differentiable on $B(x,r)$.
\end{proof}

\begin{lem} \label{D_- and D_+}
Let $\gamma : \mathbb{R}^l \rightarrow \mathbb{R}^m$ be a continuous piecewise linear function. Let $\mathcal{P}$ be a finite set of polyhedra of $\mathbb{R}^m$ such that $\bigcup_{D \in \mathcal{P}} D = \mathbb{R}^m$. Let $A_1, \dots A_s$ be a set of hyperplanes such that $\bigcup_{D \in \mathcal{P}} \partial \gamma^{-1} (D) \subset \bigcup_{k=1}^s A_k$ (Proposition \ref{boundary included in a union of hyperplanes} shows the existence of such hyperplanes). Let $H$ be an affine hyperplane and $a \in \mathbb{R}^l, b \in \mathbb{R}$ such that $H = \{ x \in \mathbb{R}^l, a^Tx + b = 0 \}$. Denote \mbox{$I = \{ k \in \llbracket 1 , s \rrbracket, A_k = H \}$}. Let $x \in H$ such that for all $k \in \llbracket 1 , s \rrbracket \backslash I$, $x \notin A_k$. Then there exists $r>0$, $D_-$ and $D_+ \in \mathcal{P}$ (not necessarily distinct) such that 
\begin{eqnarray*}
B(x,r) \cap \{ y \in \mathbb{R}^l, a^T y + b < 0 \} \quad \subset \quad \gamma^{-1}(D_-)\phantom{.} \\
B(x,r) \cap \{ y \in \mathbb{R}^l, a^T y + b > 0 \} \quad \subset \quad \gamma^{-1}(D_+).
\end{eqnarray*}
\end{lem}

\begin{proof}

Let $r > 0$ such that 
\[B(x, r) \cap \left( \bigcup_{k \notin I} A_k \right) = \emptyset .\]
$B(x,r) \backslash H$ has two connected components: $B_- = B(x,r) \cap \{ y \in \mathbb{R}^l , a^Ty + b <0 \}$ and $ B_+ = B(x,r) \cap \{ y \in \mathbb{R}^l , a^Ty + b > 0 \}$. The set $B_-$ (resp. $B_+$) is convex as an intersection of two convex sets.

Since $\bigcup_{D \in \mathcal{P}} D = \mathbb{R}^m$, there exists $D_- \in \mathcal{P}$ such that $\gamma^{-1}(D_-) \cap B_- \neq \emptyset$. Let us show that 
\[B_- \ \subset \ \gamma^{-1}(D_-).\]

Indeed, $B_- \cap \left(\bigcup_{k \notin I} A_k \right) = \emptyset$ and $B_- \cap H = \emptyset$ so $B_- \cap \left(\bigcup_{k \in I} A_k \right) = \emptyset$, therefore we have
\[B_- \cap \left(\bigcup_{D \in \mathcal{P}} \partial \gamma^{-1}(D) \right) \  \subset \ B_- \cap \left( \bigcup_{k=1}^s A_k \right) = \emptyset.\] 
In particular, $B_- \cap \partial \gamma^{-1}(D_-) = \emptyset$. Let $Y = \gamma^{-1}(D_-) \cap B_-$. Let us denote by $\partial_{B_-} Y$ the topological boundary of $Y$ with respect to the topology of $B_-$. Let us show the following inclusion:
\[\partial_{B_-} Y \ \subset \  \partial \gamma^{-1}(D_-) \cap B_-.\]
Indeed, let $y \in \partial_{B_-} Y$. By definition, there exist two sequences $(u_n)$ and $(v_n)$ such that $u_n \in Y$, $v_n \in B_- \backslash Y$, and both $u_n$ and $v_n$ tend to $y$. In particular, $u_n \in  \gamma^{-1} (D_-)$ and $v_n \in \mathbb{R}^l \backslash  \gamma^{-1}(D_-)$, so $y \in  \partial \gamma^{-1}(D_-) $. Since $y \in B_-$, we have $y \in \partial \gamma^{-1}(D_-) \cap B_-$.

This shows $\partial_{B_-} Y = \emptyset $, and as a consequence $Y$ is open and closed in $B_-$. Since $B_-$ is connex and $Y$ is not empty, we conclude that $Y = B_-$, i.e. $B_- \ \subset \ \gamma^{-1}(D_-)$.

We show similarly that there exists $D_+ \in \Pi$ such that $B_+ \subset \gamma^{-1}(D_+)$.
\end{proof}

\begin{prop}There exists a bijection $\varphi \in \mathfrak{S}_m$ such that for all $i \in \llbracket 1, m \rrbracket$, $\tilde{H}_i = H_{\varphi^{-1}(i)}$.
\end{prop}

\begin{proof}
We denote by $X$ the set of all points of $\ring{\Omega}$ at which $f$ is not differentiable. 
We denote by $\mathcal{G}$ the set of all hyperplanes of $\mathbb{R}^l$. We denote $\mathcal{H} = \{ H \in \mathcal{G} \ , \ H \cap \ring{\Omega} \neq \emptyset \text{ and } H \cap \ring{\Omega} \subset \overline{X} \}$. We want to show $\mathcal{H} = \{ H_i, i\in \llbracket 1 , m \rrbracket 
\}$.

Indeed, once this established, since $\mathcal{H}$ only depends on $\Omega$ and $f$, we also have $\mathcal{H} = \{ \tilde{H}_i, i \in \llbracket 1, m \rrbracket \}$, and thus $\{ H_i, i \in \llbracket 1, m \rrbracket \} = \{ \tilde{H}_i, i \in \llbracket 1, m \rrbracket \}$. Since, using $C.a)$, for all $i,j $, $i \neq j$, we have $H_i \neq H_j$ and $\tilde{H}_i \neq \tilde{H}_j$, we can conclude that there exists a permutation $\varphi \in \mathfrak{S}_m$ such that, for all $i \in \llbracket 1 , m \rrbracket$, $\tilde{H}_i = H_{\varphi^{-1}(i)}$.

\vspace{5pt}
{\bf $-$ Let us show $\mathcal{H} \subset \{ H_i, i\in \llbracket 1 , m \rrbracket \}$.}

To begin, let us show that $\overline{X} \cap \ring{\Omega} \ \subset \ \bigcup_{D \in \Pi} \partial h^{-1}(D) \cup \bigcup_{i=1}^m H_i$. Let $x \in \overline{X} \cap \ring{\Omega}$. Let $D \in \Pi$ such that $h(x) \in D$. Since $x \in \overline{X}$, there does not exist any $r>0$ such that $f$ is differentiable on $B(x,r)$. The contrapositive of Lemma \ref{differentiable points} shows that $x \notin \overset{\circ}{\overbrace{h^{-1}(D)}}  \backslash \left( \bigcup_{i=1}^m H_i \right)$, so either $x \in  \bigcup_{i=1}^m H_i$ or $x \notin \overset{\circ}{\overbrace{h^{-1}(D)}} $. In the latter case, since $x \in h^{-1}(D)$ by definition of $D$, we have $x \in h^{-1}(D) \backslash  \overset{\circ}{\overbrace{h^{-1}(D)}} \subset \partial h^{-1}(D)$.

This shows:
\begin{equation}\label{overline X}\overline{X} \cap \ring{\Omega} \quad \subset \quad \bigcup_{D \in \Pi} \partial h^{-1}(D) \cup \bigcup_{i=1}^m H_i.\end{equation}

Let $H \in \mathcal{H}$. We are going to show that there exists $i \in \llbracket 1, m \rrbracket$ such that $H = H_i$.

 We know by condition $C.d$ that $H \cap \ring{\Omega} \not\subset \bigcup_{D \in \Pi} \partial h^{-1}(D)$. Let $x \in (H \cap \ring{\Omega} )\backslash \left( \bigcup_{D \in \Pi} \partial h^{-1}(D) \right)$. The set $ \bigcup_{D \in \Pi} \partial h^{-1}(D)$ is closed, so there exists a ball 
\begin{equation}\label{ball inclusion}
B(x,r) \subset \ring{\Omega} \backslash \left( \bigcup_{D \in \Pi} \partial h^{-1}(D) \right). 
\end{equation}
By definition of $\mathcal{H}$, 
\[H \cap \ring{\Omega} \ \subset \ \overline{X} \cap \ring{\Omega},\]
 so using the fact that $ B(x,r) \subset \ring{\Omega}$ we have:

\[B(x,r) \cap H \ = \ B(x,r) \cap H \cap \ring{\Omega} \ \subset \ B(x,r) \cap \overline{X} \cap \ring{\Omega}. \]

Thus, using \eqref{overline X},
\begin{equation*}
\begin{aligned}   B(x,r) \cap H \ &  \subset \ B(x,r) \cap \overline{X} \cap \ring{\Omega}  \\
& \subset \ B(x,r) \cap \left( \bigcup_{D \in \Pi} \partial h^{-1}(D) \cup \bigcup_{i=1}^m H_i \right)  \\
&  = \ \left( B(x,r) \cap \bigcup_{D \in \Pi} \partial h^{-1}(D) \right)  \cup \left( B(x,r) \cap   \bigcup_{i=1}^m H_i   \right),
\end{aligned}
\end{equation*}
and since by \eqref{ball inclusion} the first set of the last equality is empty, we have
\begin{equation*}
\begin{aligned}
 B(x,r) \cap H \ \subset \ B(x,r) \cap   \bigcup_{i=1}^m H_i. 
\end{aligned}
\end{equation*}
Therefore,
\begin{equation*}
\begin{aligned} B(x,r) \cap H & =  \left( B(x,r) \cap H \right) \cap \left( B(x,r)\cap \bigcup_{i = 1}^m H_i\right) \\
& =  B(x,r) \cap H \cap \bigcup_{i = 1}^m H_i \\
& = B(x,r) \cap  \bigcup_{i = 1}^m \left(H \cap H_i \right).
\end{aligned}
\end{equation*}

Assume, by contradiction, that for all $i \in \llbracket 1 , m \rrbracket$ we have $H \neq H_i $. Then $H \cap H_i$ is an affine space of dimension less or equal to $l-2$ so it has Hausdorff dimension smaller or equal to $l-2$. A finite union of sets of Hausdorff dimension smaller or equal to $l-2$ has Hausdorff dimension smaller or equal to $l-2$. Thus, $B(x,r) \cap H = B(x,r) \cap \bigcup_{i = 1}^m \left(H \cap H_i \right)$ has Hausdorff dimension smaller or equal to $l-2$, which is absurd since $x \in H$ so $B(x,r) \cap H$ has Hausdorff dimension $l-1$. Hence there exists $i \in \llbracket 1 , m \rrbracket$ such that $H = H_i$.

We have shown 
\begin{equation}\label{mathcal H subset Hi}
\mathcal{H} \subset \{ H_i, i\in \llbracket 1 , m \rrbracket \}.
\end{equation}

\vspace{5pt}

{\bf $-$ Let us show $\{ H_i, i\in \llbracket 1 , m \rrbracket \} \subset \mathcal{H}$.}

Let $i \in \llbracket 1 , m \rrbracket$. Let us prove $H_i \in \mathcal{H}$.

First, by condition $C.b)$ we know that $ E_i \cap h^{lin}(\ring{\Omega}) \neq \emptyset$, so there exists $x \in \ring{\Omega}$ such that $h^{lin}(x) \in E_i$. Since $h^{lin}(x) = Mx +b$ and $E_i$ is the space of vectors whose $i^{\text{th}}$ coordinate is $0$, this is equivalent to 
\[M_{i,.}x + b_i = 0,\]
or said otherwise $x \in H_i$. This proves that $H_i \cap \ring{\Omega} \neq \emptyset$. We still need to prove $H_i \cap \ring{\Omega} \subset \overline{X}$.

Let $x \in H_i \cap \ring{\Omega} $. Let us prove $x \in \overline{X}$.

Since $M$ is full row rank, the line vectors $M_{1,.}, \dots , M_{m,.}$ are linearly independent, and thus for all $k \in \llbracket 1 , m \rrbracket\backslash \{ i \}$, $H_k \cap H_i$ has Hausdorff dimension smaller or equal to $l-2$.

Proposition \ref{boundary included in a union of hyperplanes} shows that $\bigcup_{D \in \Pi} \partial h^{-1}(D)$ is contained in a finite union of hyperplanes $\bigcup_{k=1}^s A_k$. Let \mbox{$I = \{ k \in \llbracket 1 , s \rrbracket \ , \ A_k = H_i \}.$} For all $k \in \llbracket 1 , s \rrbracket \backslash I$, $A_k \cap H_i$ is either empty, or an intersection of two non parallel hyperplanes, in both cases it is an affine space of dimension smaller than $l-2$.

Thus, 
\[H_i \cap \left( (\bigcup_{k \neq i} H_k ) \cup (\bigcup_{k \notin I} A_k) \right) \]
has Hausdorff dimension strictly smaller than $l-1$, so for any $r>0$ there exists 
\begin{equation}\label{y dans une intersection d'ensembles}
y \in B(x,r) \cap H_i \cap \ring{\Omega} \backslash \left( (\bigcup_{k \neq i} H_k ) \cup (\bigcup_{k \notin I} A_k) \right).
\end{equation}

In the rest of the proof, we show that such a $y$ is an element of $X$. Once this is established, since it is true for all $r>0$, we conclude that $x \in \overline{X}$ and therefore $H_i \in \mathcal{H}$.

If there exists $D \in \Pi$ such that $y \in \overset{\circ}{\overbrace{h^{-1}(D)}}$, then  
\[ y \in H_i  \cap \overset{\circ}{\overbrace{h^{-1}(D)}} \cap \ring{\Omega} \backslash \left( \bigcup_{k \neq i} H_k \right)\]
therefore we can use Proposition \ref{non differentiable points} to conclude that $f$ is not differentiable at $y$.
 
Otherwise we can use Lemma \ref{D_- and D_+} to find $R_1>0$, $D_-$ and $D_+ \in \Pi$ such that
\begin{eqnarray*}
B(y,R_1) \cap  \{ z \in \mathbb{R}^l, M_{i,.} z + b_i < 0 \} \quad \subset \quad h^{-1}(D_-)\phantom{.} \\
B(y,R_1) \cap \{ z \in \mathbb{R}^l, M_{i,.} z + b_i > 0  \} \quad \subset \quad h^{-1}(D_+).
\end{eqnarray*}

Since for all $j \neq i$, $y \notin H_j$ and since these hyperplanes are closed, there exists $R_2 > 0$ such that for all $j \neq i$, $B(y,R_2) \cap H_j = \emptyset$. Let $R = \min(R_1,R_2)$ and denote $B_- = B(y,R) \cap  \{ z \in \mathbb{R}^l, M_{i,.} z + b_i < 0 \} $ and $B_+ = B(y,R) \cap \{ z \in \mathbb{R}^l, M_{i,.} z + b_i > 0  \}$.

For all $z \in B_-$, using \eqref{expression of f with V} with the fact that $s_i(z) = 0$ and $s_k(z) = s_k(y)$ for all $k \neq i$, we have
\begin{equation}{\label{expression of f on B-}}f(z) = \sum_{k \neq i} V_{.,k}(D_-) s_k(y) (M_{k,.}z + b_k) +c(D_-).  \end{equation}
For all $z \in B_+$, using this time that $s_i(z) = 1$, we have
\begin{equation}\label{expression of f on B+}
f(z) = \sum_{k \neq i} V_{.,k}(D_+) s_k(y) (M_{k,.}z + b_k) +c(D_+) + V_{.,i}(D_+) (M_{i,.}z + b_i) .
\end{equation}

If $f$ was differentiable at $y$, we would derive from \eqref{expression of f on B-} the expression of the Jacobian matrix
\begin{equation}\label{expression of df on B-}
J_f(y) = \sum_{k \neq i} V_{.,k}(D_-) s_k(y) M_{k,.},\end{equation}
but we would also derive from \eqref{expression of f on B+} the expression
\begin{equation}\label{expression of df on B+}
J_f(y) = \sum_{k \neq i} V_{.,k}(D_+) s_k(y) M_{k,.} + V_{.,i}(D_+) M_{i,.},
\end{equation}
hence subtracting \eqref{expression of df on B-} to \eqref{expression of df on B+} we would find
\[\sum_{k \neq i} (V_{.,k}(D_+) -  V_{.,k}(D_-)) s_k(y) M_{k,.} + V_{.,i}(D_+) M_{i,.} = 0.\]
Since $M$ is full row rank, this would imply that $V_{.,i}(D_+) = 0$. 

However since $h^{-1}(D_+)$ is closed and contains $B_+$, we have $y \in \overline{B_+} \subset h^{-1}(D_+)$. Recalling \eqref{y dans une intersection d'ensembles} we thus have
\[y \in H_i \cap h^{-1}(D_+) \cap \ring{\Omega},\]
thus
\[h(y) \in E_i \cap D_+ \cap h(\ring{\Omega}),\]
 which shows the latter intersection is not empty. By assumption $C.c)$ this implies that $V_{.,i}(D_+) \neq 0$, which is a contradiction. Therefore $f$ is not differentiable at $y$.

As a conclusion, we have showed that for all $r>0$, there exists $y \in B(x,r)$ such that $f$ is not differentiable at $y$ and $y \in \ring{\Omega}$. In other words, $x \in \overline{X}$.

Since $x$ is arbitrary in $H_i \cap \ring{\Omega}$, we have shown that for all $i \in \llbracket 1, m \rrbracket$,
\[H_i \cap \ring{\Omega} \ \subset \ \overline{X},\]
i.e., since we have already shown that $H_i \cap \ring{\Omega} \neq \emptyset$,
\[H_i \in \mathcal{H}.\]

Finally $\{ H_i, i\in \llbracket 1 , m \rrbracket \} \subset \mathcal{H}$, and, using \eqref{mathcal H subset Hi},
\[\mathcal{H} = \{ H_i, i\in \llbracket 1 , m \rrbracket \}.\]
\end{proof}

\begin{prop}\label{relations}
For all $i \in \llbracket 1 , m \rrbracket$, there exists $\epsilon_{\varphi^{-1}(i)} \in \{ -1 , 1 \}$ such that 
\[\tilde{M}_{i,.} = \epsilon_{\varphi^{-1}(i)} M_{{\varphi^{-1}(i)},.} \qquad \text{and} \qquad \tilde{b}_i = \epsilon_{\varphi^{-1}(i)} b_{\varphi^{-1}(i)}.\]
\end{prop}

\begin{proof}
Let $i \in \llbracket 1 , m \rrbracket$. We know that $\tilde{H}_i = H_{\varphi^{-1}(i)}$, so the equations $ \tilde{M}_{i,.} x + \tilde{b}_i = 0 $ and $ M_{{\varphi^{-1}(i)},.} x   + b_{\varphi^{-1}(i)} = 0$ define the same hyperplanes. This is only possible if the parameters of the equation are proportional (but nonzero): there exists $\epsilon_{\varphi^{-1}(i)} \in \mathbb{R}^*$ such that $\tilde{M}_{i,.} = \epsilon_{\varphi^{-1}(i)} M_{\varphi^{-1}(i),.}$ and $\tilde{b}_i = \epsilon_{\varphi^{-1}(i)} b_{\varphi^{-1}(i)}$. But since $\|\tilde{M}_{i,.}\| = \|M_{{\varphi^{-1}(i)},.}\|=1$ by hypothesis, we necessarily have $\epsilon_{\varphi^{-1}(i)} \in \{-1,1\}$.
\end{proof}

\begin{prop}
 For all $i \in \llbracket 1 , m \rrbracket$,
 \begin{itemize}
 \item $\tilde{M}_{i,.} = M_{{\varphi^{-1}(i)},.} $;
 \item $\tilde{b}_i = b_{\varphi^{-1}(i)} $.
 \end{itemize}
\end{prop}

\begin{proof}
By Proposition \ref{relations}, we know that there exists $(\epsilon_i)_{1 \leq i \leq m} \in \{ -1 , 1 \}^{m}$  such that for all $i \in \llbracket 1 , m \rrbracket$,
\begin{equation}\label{M_i = ei Mtilde_i}
\tilde{M}_{i,.} = \epsilon_{\varphi^{-1}(i)} M_{{\varphi^{-1}(i)},.} \qquad \text{ and } \qquad  \tilde{b}_i = \epsilon_{\varphi^{-1}(i)} b_{\varphi^{-1}(i)}.
\end{equation}
We need to prove that for all $i \in \llbracket 1 , m \rrbracket$,  $\epsilon_{\varphi^{-1}(i)} = 1$. 

Let $i \in \llbracket 1 , m \rrbracket$. 

Applying Proposition \ref{boundary included in a union of hyperplanes} to $h$ and $\Pi$, we see that $\bigcup_{D \in \Pi} \partial h^{-1}(D)$ is contained in a finite union of hyperplanes $\bigcup_{k=1}^s A_k$. Applying it to $\tilde{h}$ and $\tilde{\Pi}$, we see similarly that $\bigcup_{\tilde{D} \in \tilde{\Pi}} \partial \tilde{h}^{-1}(\tilde{D})$ is contained in a finite union of hyperplanes $\bigcup_{k=1}^r B_k$.

Let \mbox{$I = \{ k \in \llbracket 1 , s \rrbracket \ , \ A_k = H_i \}$} and \mbox{$J = \{ k \in \llbracket 1 , r \rrbracket \ , \ B_k = H_i \}.$} For all $k \in \llbracket 1 , s \rrbracket \backslash I$, since $A_k \neq H_i$, $A_k \cap H_i$ is either empty, or an intersection of two non parallel hyperplanes, in both cases it is an affine space of dimension smaller than $l-2$. The same applies for all $k \in \llbracket 1 , r \rrbracket \backslash J$ to $B_k \cap H_i$. For all $j \neq i$, $H_j \neq H_i$ so $H_j \cap H_i$ is also an affine space of dimension smaller than $l-2$.  Since $H_i \cap \ring{\Omega}$ is nonempty by $\mathbf{C}.b)$, we can thus find a vector
\[x \quad \in \quad \ring{\Omega} \ \cap \ H_i \ \backslash \ \left( (\bigcup_{k \notin I} A_k ) \cup (\bigcup_{k \notin J } B_k ) \cup ( \bigcup_{j \neq i} H_j ) \right) .\]

Applying Lemma \ref{D_- and D_+} with $\Pi$, $h$, $H_i$ and $(M_{i,.}, b_i)$, we find $r_1 > 0$, $D_-$ and $D_+ \in \Pi$ such that 
\begin{eqnarray}\label{application lemma 39-1}
B(x,r_1) \cap \{ y \in \mathbb{R}^l, M_{i,.} y + b_i < 0 \} \quad \subset \quad h^{-1}(D_-)\phantom{.} \nonumber \\
B(x,r_1) \cap \{ y \in \mathbb{R}^l, M_{i,.} y + b_i > 0 \} \quad \subset \quad h^{-1}(D_+).
\end{eqnarray}
Applying the same lemma with $\tilde{\Pi}$, $\tilde{h}$, $H_i$ and $(M_{i,.}, b_i)$ we find $r_2 > 0$, $\tilde{D}_-$ and $\tilde{D}_+ \in \tilde{\Pi}$ such that 
\begin{eqnarray}\label{application lemma 39-2}
B(x,r_2) \cap \{ y \in \mathbb{R}^l, M_{i,.} y + b_i < 0 \} \quad \subset \quad \tilde{h}^{-1}(\tilde{D}_-)\phantom{.} \nonumber \\
B(x,r_2) \cap \{ y \in \mathbb{R}^l, M_{i,.} y + b_i > 0 \} \quad \subset \quad \tilde{h}^{-1}(\tilde{D}_+).
\end{eqnarray}

Since the hyperplanes $H_j$ are closed, we can also find $r_3 > 0$ such that for all $j \neq i$,  $B(x,r_3) \cap H_j  = \emptyset$. Taking $r = \min(r_1,r_2,r_3)$ and denoting \mbox{$B_+ = B(x,r) \cap \{ y \in \mathbb{R}^l, M_{i,.} y + b_i > 0 \}$}, we derive from \eqref{application lemma 39-1} and \eqref{application lemma 39-2} that 
\[B_+ \quad \subset \quad h^{-1} (D_+) \cap \tilde{h}^{-1} ( \tilde{D}_+).\]

Since $r \leq r_3$, we have $B_+ \cap \left( \bigcup_{j \neq i} H_j \right) = \emptyset$, and by definition $B_+ \cap \{y \in \mathbb{R}^l, M_{i,.}y + b_i = 0 \} = \emptyset$, so $B_+ \cap H_i = \emptyset$. We have $B_+ \cap \left( \bigcup_{j =1 }^m H_j \right) = \emptyset$, so for all $j \in \llbracket 1 , m \rrbracket $, there exist $\delta_j \in \{0,1\}$ such that for all $y \in B_+$, $s_j(y) = \delta_j$. We have $\bigcup_{j=1}^m \tilde{H}_j = \bigcup_{j=1}^m H_j$ so similarly, $B_+ \cap \bigcup_{j=1}^m \tilde{H}_j = \emptyset$ and there exists $\tilde{\delta}_j \in \{0,1\}$ such that for all $j \in \llbracket 1 , m \rrbracket$, for all $y \in B_+$, $\tilde{s}_j(y) = \tilde{\delta}_j$. 

For all $y \in B_+$, we thus have, using \eqref{expression of f with V},
\[\sum_{j=1}^m V_{.,j}(D_+) \delta_j \left( M_{j,.}y + b_j \right) + c(D_+) = \sum_{j=1}^m \tilde{V}_{.,j}(\tilde{D}_+) \tilde{\delta}_j \left( \tilde{M}_{j,.}y + \tilde{b}_j \right) + \tilde{c}(\tilde{D}_+).\]
 $B_+$ is a nonempty open set so we have the equality
 \begin{eqnarray}\label{equality 2}
 \sum_{j=1}^m V_{.,j}(D_+) \delta_j  M_{j,.} & = & \sum_{j=1}^m \tilde{V}_{.,j}(\tilde{D}_+) \tilde{\delta}_j  \tilde{M}_{j,.} \nonumber \\
 & = & \sum_{j=1}^m \tilde{V}_{.,j}(\tilde{D}_+) \tilde{\delta}_j \epsilon_{\varphi^{-1}(j)} M_{{\varphi^{-1}(j)},.} \nonumber \\
 & = & \sum_{j=1}^m \tilde{V}_{.,\varphi(j)}(\tilde{D}_+) \tilde{\delta}_{\varphi(j)} \epsilon_{j} M_{j,.}.
 \end{eqnarray}
The condition $\mathbf{C}.a)$ states that $M$ is full row rank, so the vectors $M_{j,.}$ are linearly independent. Applied to \eqref{equality 2}, this information yields, for all $j \in \llbracket 1 , m \rrbracket$,
\[V_{.,j}(D_+) \delta_j =  \tilde{V}_{.,\varphi(j)}(\tilde{D}_+) \tilde{\delta}_{\varphi(j)} \epsilon_{j}, \]
and in particular,
\begin{equation}\label{equality 3}
V_{.,i}(D_+) \delta_i =  \tilde{V}_{.,\varphi(i)}(\tilde{D}_+) \tilde{\delta}_{\varphi(i)} \epsilon_{i}. \end{equation}
Since $h^{-1}(D_+)$ and $\tilde{h}^{-1}(\tilde{D}_+)$ are closed, we have
\[\overline{B_+} \ \subset \ h^{-1}(D_+) \cap \tilde{h}^{-1}(\tilde{D}_+),\]
and since $x \in \overline{B_+}$, we have $h^{-1}(D_+) \cap H_i \neq \emptyset$ and $\tilde{h}^{-1}(\tilde{D}_+) \cap H_i \neq \emptyset$. The condition $C.c)$ implies that $V_{.,i}(D_+) \neq 0$ and $\tilde{V}_{.,{\varphi(i)}}(\tilde{D}_+) \neq 0$ (recall that $H_i = \tilde{H}_{\varphi(i)}$). We also have $\epsilon_i \neq 0$, so from \eqref{equality 3} we obtain
\[\delta_i = 0 \Leftrightarrow \tilde{\delta}_{\varphi(i)} = 0.\]

By definition, the coefficient $\delta_i$ depends on the sign of $M_{i,.}y + b_i$: if $M_{i,.}y + b_i$ is positive, $\delta_i =1$ and if $M_{i,.}y + b_i$ is negative then $\delta_i = 0$ ($M_{i,.}y + b_i$ cannot be equal to zero since $y \notin H_i$). The coefficient $\tilde{\delta}_{\varphi(i)}$ depends similarly on the sign of $\tilde{M}_{{\varphi(i)},.}y + \tilde{b}_{\varphi(i)}$. Thus, $M_{i,.}y + b_i$ and $\tilde{M}_{{\varphi(i)},.}y + \tilde{b}_{\varphi(i)}$ have same sign.

Since $\epsilon_i \in \{-1,1 \}$ and
\[\tilde{M}_{{\varphi(i)},.}y + \tilde{b}_{\varphi(i)} = \epsilon_i {M}_{i,.}y + \epsilon_i {b}_i = \epsilon_i \left( {M}_{i,.}y +  {b}_i \right),\]
we conclude that $\epsilon_i=1$.

\end{proof}

We can now finish the proof of Lemma \ref{Fundamental lemma}. It results from the above that:
\[\tilde{M} = P_{\varphi}M \]
\[\tilde{b} = P_{\varphi} b.\]

We have by hypothesis, for all $x \in \Omega$,
\[\tilde{g}(\sigma(\tilde{M}x + \tilde{b})) = g(\sigma(Mx + b)) ,\]

but since $\tilde{M} = P_{\varphi}M $ and $\tilde{b} = P_{\varphi} b$ we also have:
\[\tilde{g}(\sigma(\tilde{M}x + \tilde{b})) = \tilde{g}(\sigma(P_{\varphi}Mx + P_{\varphi} b)) =  \tilde{g}(P_{\varphi}\sigma(Mx + b)).\]

Combining these, we have for all $x \in \Omega$,
\[\tilde{g} \circ P_{\varphi} (h(x))  = g(h(x)) ,\]

i.e. $\tilde{g} \circ P_{\varphi} $ and $g$ coincide on $h(\Omega)$.

\end{appendix}


\bibliographystyle{plain}
\bibliography{biblio}% common bib file

\end{document}